\documentclass[final,3p, sort&compress, times]{elsarticle}

\usepackage{amssymb}

\journal{arXiv}

\usepackage{amsmath,amssymb,amsthm}
\usepackage{esint}
\usepackage{bm}
\usepackage{url}
\usepackage{subfigure}
\usepackage{graphicx,color}
\usepackage{natbib}
\usepackage{algorithm}
\usepackage{algpseudocode}
\usepackage{algorithmicx}
\usepackage{booktabs,multirow}
\usepackage{hhline}
\usepackage{setspace}
\usepackage{caption}
\usepackage{hyphsubst}
\usepackage{float}
\usepackage{filecontents}
\usepackage{hyperref}
\hypersetup{colorlinks,citecolor=blue,linkcolor=blue, urlcolor=blue}
\usepackage{pdflscape}
\usepackage{breakurl}

\renewcommand{\Vec}[1]{\overrightarrow{#1}}
\newcommand{\Vcurl}{\Vec{\rm{\bf curl\,}}}
\newcommand{\curl}{{\rm curl\,}}

\DeclareMathAlphabet{\itbf}{OML}{cmm}{b}{it}

\def\by{{{\bf y}}}
\def\bx{{{\bf x}}}
\def\bz{{{\bf z}}}
\def\hbx{{\hat{\bx}}}
\def\hby{{\hat{\by}}}
\def\hbz{{\hat{\bz}}}

\def\be{{{\itbf e}}}

\def\bu{{{\bf u}}}
\def\bv{{{\bf v}}}
\def\bw{{{\bf w}}}

\newcommand{\hbe}{\widehat{\be}}

\newcommand{\bGam}{\mathbf{\Gamma}}

\renewcommand{\i}{\rm i}

\def\btheta{{\boldsymbol{\theta}}}
\def\bTheta{{\boldsymbol{\Theta}}}

\def\bxi{{\boldsymbol{\xi}}}

\newcommand{\RR}{\mathbb{R}}
\newcommand{\R}{\mathbb{R}}

\newcommand{\K}{{\kappa}}
\newcommand{\dis}{\displaystyle}
\newcommand{\OL}{\mathcal{L}}

\newcommand{\I}{\mathbf{I}}

\newcommand{\C}{\mathcal{C}}

\newcommand{\CC}{\mathbb{C}}
\newcommand{\MM}{\mathbb{M}}

\newcommand{\ds}{\displaystyle}
\newcommand{\NN}{\mathbb{N}}

\def\nm{\noalign{\medskip}}

\newcommand{\bnu}{\bm{\nu}}

\newcommand{\Pk}{\Pi}

\newtheorem{thm}{Theorem}[section]
\newtheorem{cor}[thm]{Corollary}
\newtheorem{lem}[thm]{Lemma}
\newtheorem{prop}[thm]{Proposition}
\newtheorem{defn}[thm]{Definition}
\newtheorem{rem}[thm]{Remark}
\newtheorem{prob}{Problem}

\usepackage{epstopdf}
\begin{document}

\begin{frontmatter}
\title{Topological sensitivity based far-field detection of elastic inclusions\tnoteref{t1}}
\tnotetext[t1]{This work was supported by the National Research Foundation of Korea (NRF) under Grant (NRF-2016R1A2B3008104) and by the Korea Research Fellowship Program through the National Research Foundation (NRF) funded by the Ministry of Science and ICT (NRF- 2015H1D3A1062400).
}
 
 \author[els]{Tasawar Abbas}
 \ead{tasawar44@hotmail.com}
 
 \author[rvt]{Shujaat Khan}
 \ead{shujaat@kaist.ac.kr}

 \author[els]{Muhammad Sajid}
 \ead{muhammad.sajid@iiu.edu.pk}
 
 \author[rvt]{Abdul Wahab\corref{cor1}}
 \ead{wahab@kaist.ac.kr}
 
 \author[rvt]{Jong Chul Ye\fnref{el}}
 \ead{jong.ye@kaist.ac.kr}

 \address[els]{Department of Mathematics and Statistics, Faculty of Basic and Applied Sciences, International Islamic University Islamabad, 44000, Pakistan.}
 
 \address[rvt]{Bio Imaging and Signal Processing Laboratory, Department of Bio and Brain Engineering, Korea Advanced Institute of Science and Technology, 291 Daehak-ro, Yuseong-gu, Daejeon 34141, South Korea.}
 
 \address[el]{Department of Mathematical Sciences, Korea Advanced Institute of Science and Technology, 291 Daehak-ro, Yuseong-gu, Daejeon 34141, South Korea.}
 
 \cortext[cor1]{Address all correspondence to A. Wahab at Email: wahab@kaist.ac.kr or Ph.: +82-42-350-4360.}

\begin{abstract}
The aim of this article is to present and rigorously analyze topological sensitivity based algorithms for detection of diametrically small inclusions in an isotropic homogeneous elastic formation using single and multiple measurements of the far-field scattering amplitudes. A  $L^2-$cost functional is considered and a location indicator is constructed from its topological derivative. The performance of the indicator  is analyzed in terms of the topological sensitivity for location detection and stability with respect to measurement and medium noises. It is established that the location indicator does not guarantee inclusion detection and achieves only a low resolution when there is mode-conversion in an elastic formation.  Accordingly, a weighted location indicator is designed to tackle the mode-conversion phenomenon. It is substantiated that the weighted function renders the location of an inclusion  stably with resolution as per Rayleigh criterion.
\end{abstract}


\end{frontmatter}

\section{Introduction}

Inverse scattering has been of prime interest in recent decades due to a variety of applications in various branches of engineering and applied sciences \cite{Princeton, colton2, colton, ESC, OptEx}. The goal of these problems is to locate and characterize scatterers of different geometric nature such as inclusions, cracks, and cavities from the limited information of single or multiple scattered fields.  Many promising computational and mathematical frameworks adaptive to different imaging and experimental setups have been developed to address these inverse problems over a span of last few decades  (see, e.g., \cite{DirectElast, Elena, Anna, JS, DL, Li1, Li2, Li3, GWL, Gintides, Amer}). In particular, topological sensitivity frameworks have received significant attention for the reconstruction of location, shape or constitutive parameters of anomalies due to their simplicity and robustness (see, e.g., \cite{Ahn, Bellis, BG, CGGM, DG, DGE, Guzina, HL,  MPS,  park, park2,  Bao, bonnet, partial, carpio1, carpio2, Louer}). 

In topological sensitivity frameworks, an inverse scattering problem is first converted to a minimization problem for a discrepancy functional by nucleating an anomaly at a search location in the background medium. The \emph{topological sensitivity} of the misfit to nucleating an anomaly at different search locations, is determined by its topological derivative which serves as a location indicator or  shape classifier.

Despite their extensive use, the quantitative analysis of the topological sensitivity  inverse scattering frameworks for anomaly detection, in terms of resolution limit, signal-to-noise ratio and stability in the presence of medium or measurement noise, remains heuristic at large. First rigorous quantitative analysis for anti-plane elasticity was performed by Ammari et al \cite{AGJK} using asymptotic expansions with respect to the size of inclusion.   It was established that inclusion detection inside a bounded domain is only guaranteed if the measurements are pre-processed for the boundary effects and the inclusions are separated apart from the boundary. Towards this end,  the \emph{Calder\'{o}n preconditioner} associated to the boundary of the domain was used as a filter. The performance of the framework was compared with a number of contemporary algorithms including Kirchhoff migration, back-propagation, and Multiple Signal Classification (MUSIC). The topological sensitivity algorithm with preconditioning appeared to be more stable and resolved than the listed algorithms but at a cost of increased computational complexity. It was observed in \cite{TDelastic} that, even with the preconditioning, the localization using near-field measurements of the scattered fields was compromised in the general isotropic elasticity setting due to non-linear coupling  and conversion of shear and pressure wave-modes.  Accordingly, the classical framework was slightly modified using Helmholtz vector decomposition and assigning proper weights to the pressure and shear  components in terms of respective wave speeds. In three-dimensional electromagnetic media in full Maxwell setting, the performance of topological sensitivity based inclusion detection functions with filtered boundary measurements of the tangential components of the magnetic fields was debated in \cite{wahab}. The analysis was further extended in \cite{JCM, JPC} to the detection of electromagnetic inclusions using single and multiple far-field scattering amplitudes. 

The aim in this article is to design and debate the performance of the topological sensitivity algorithms for locating small inclusions of vanishing sizes in an unbounded elastic formation with single or multiple measurements of the elastic far-field patterns. The general case of a linear isotropic elastic medium is considered  and a location indicator is constructed from a  $L^2-$cost functional.  A rigorous sensitivity analysis is performed based on the asymptotic expansions of the far-field scattering patterns versus the scale factor of the inclusion. It is established that the performance of the classical location indicator degenerates due to the mode-conversion (from longitudinal (P) to transverse (S) waveforms and vise-verse). Specifically, it is established that the indicator does not guarantee inclusion detection and achieves a low resolution than expected in the presence of mode-conversion. Accordingly, a weighted correction to the location indicator is proposed. Firstly, the pressure and the shear components of the  back-propagator are decoupled using Helmholtz mode decomposition. Then, the modes of the back-propagator are correlated with the corresponding components of the incident fields. Finally, the resulting components are aggregated after assigning proper weights in terms of the corresponding wave speeds. In order to debate the capabilities of the weighted indicator, a rigorous sensitivity and stability analysis is performed when the measurements are corrupted by an additive noise or there is a random medium noise contaminating the far-field patterns.

It is worthwhile highlighting that the inverse elastic scattering problem caters to various applications including non-destructive evaluation of an elastic structure for integrity and material impurities \cite{A1}, prospecting of mineral reservoirs \cite{A2}, and medical diagnostics for detecting and classifying small tumors and locating tissue abnormalities of vanishing sizes \cite{A3, A4}.  

The contents of this article are arranged in the following order. A mathematical description of the inverse problem is furnished in Section \ref{s:form} along with a few preliminaries. The cost functional and the corresponding topological sensitivity based location indicators are introduced in Section \ref{s:indicators} and their topological sensitivity and resolution limits are analyzed. In Section \ref{s:measurement}, the statistical stability of the proposed location indicators is discussed when the measurements are corrupted by an additive Gaussian noise. The statistical stability of the proposed indicators in a randomly fluctuating medium is debated in Section \ref{s:Medium}. Finally, the main results of this article are summarized in Section \ref{s:conclusion}.

\section{Mathematical formulation}\label{s:form}

Let us mathematically introduce the inverse scattering problem undertaken in this article. The nomenclature of this investigation is provided in Section \ref{ss:prob} along with the mathematical description of the inverse problem dealt with. A few preliminaries are given in Section \ref{ss:prelim} to facilitate the ensuing discussion. For details beyond those provided in this section, the readers are suggested to consult monograph \cite{Princeton}.

\subsection{Nomenclature and problem formulation}\label{ss:prob}

Let $\RR^d$, $d=2$ or $3$, be loaded by an isotropic homogeneous elastic material (hereinafter referred to as  the \emph{background medium}) that has the volume density  $\rho_0\in\RR_+$, and the Lam\'e parameters $\lambda_0$ (compressional modulus) and $\mu_0$ (shear modulus) satisfying the strong convexity conditions, 
\begin{align*}
\qquad \mu_0>0\quad\text{and}\quad d\lambda_0+2\mu_0>0.
\end{align*}
Let an isotropic and homogeneous elastic inclusion, represented by a bounded domain $D:=\bz_D+\epsilon B_D$ with $\C^2$ boundary $\partial D$, be embedded in the background medium. The position vector $\bz_D\in\RR^d$ is the center of mass of the inclusion $D$ and the reference domain $B_D\subset\RR^d$, assumed to be smooth and containing the origin, is the bulk of $D$. The scale factor  $\epsilon\in\RR_+$ determines the characteristic size of $D$.  The inclusion is supposed to have the corresponding parameters   $\rho_1\in\RR_+$, $\lambda_1$, and $\mu_1$   which satisfy
\begin{align}
\mu_1>0 \quad\text{and}\quad d\lambda_1+2\mu_1>0.\label{constraint1}
\end{align}
It is further assumed that 
\begin{align}
(\lambda_1-\lambda_0)(\mu_1-\mu_0)\geq 0,\label{constraint2}
\end{align}
which is required to ensure the positive (or negative)-definiteness of the associated EMT (see, Section \ref{sss:EMT}). Henceforth, the constitutive parameters of the medium in the presence of inclusion $D$ are denoted by  $\lambda$,  $\mu$, and $\rho$, i.e.,
\begin{align*}
(\lambda; \mu; \rho)(\bx):=(\lambda_0; \mu_0; \rho_0)\chi_{\RR^d\setminus{\overline{D}}}(\bx)+(\lambda_1; \mu_1; \rho_1)\chi_{D}(\bx),
\end{align*}
where $\chi_{D}$ represents the characteristic function of domain $D$. 

Let $\bw:\RR^d\to\CC^d$ be a generic vector field and $\bnu(\bx):\partial D\to\RR^d$ denote the outward unit normal at $\bx\in\partial D$. Then, the  linear elasticity operator and the surface traction on $\partial D$ are, respectively, defined as
\begin{align*}
&
\OL_{\lambda_0,\mu_0} [\bw](\bx):=\left(\lambda_0\Delta\bw+(\lambda_0+\mu_0)\nabla\nabla\cdot\bw\right)(\bx), \qquad\bx\in \RR^d,
\\\nm
&
\frac{\partial\bw}{\partial\bnu}(\bx):= \left(\lambda_0(\nabla\cdot\bw)\bnu+2\nabla^s\bw\,\bnu\right)(\bx), \qquad\bx\in\partial D.
\end{align*}
The elasticity and  surface traction operators associated with the parameters $(\lambda_1;\mu_1)$ are defined analogously and are denoted by $\OL_{\lambda,\mu} $  and ${\partial}/{\partial\widetilde{\bnu}}$, respectively.  Here $\nabla^s\bw$ is the symmetric gradient of the vector $\bw$, i.e., 
$\nabla^s\bw:=\left(\nabla\bw+\left(\nabla\bw\right)^\top\right)/2$,
with  $\nabla\bw$ being the Jaccobian matrix of $\bw$ and  superposed $\top$ indicating the transposition.

Let $\bu^{\rm inc}:\RR^d\to \CC^d$ be an incident time-harmonic elastic field  impinging on $D$ with time variations $e^{-i\omega t}$ being suppressed, where $\omega\in\RR_+$ denotes the frequency of the mechanical oscillations. Precisely, $\bu^{\rm inc}$ satisfies the Lam\'e system,
\begin{align*}
\OL_{\lambda_0,\mu_0}[\bu^{\rm inc}](\bx)+\rho_0\omega^2\bu^{\rm inc}(\bx)=\mathbf{0}, \qquad \bx\in\RR^d.
\end{align*}
In this article, only plane incident fields of the type
\begin{equation}
\bu^{\rm inc}(\bx)=\bu^P(\bx) := \btheta e^{\i\K_P\bx\cdot\btheta} 
\quad\text{and}\quad 
\bu^{\rm inc}(\bx)=\bu^{S}(\bx):= \btheta^{\perp} e^{\i\K_S\bx\cdot\btheta},
\label{eq:plane}
\end{equation}
will be considered. Here  $\btheta\in\RR^d$ is a unit vector along the direction of  incidence and  $\btheta^{\perp}\in\RR^d$ is any vector perpendicular to $\btheta$. The constants $\K_P$ and $\K_S$   are the compression and shear wavenumbers respectively, i.e.,
\begin{align*}
\K_P:= \frac{\omega}{c_P}\quad\text{and}\quad\K_S:= \frac{\omega}{c_S}\quad\text{with}\quad \quad c_P= \sqrt{\frac{\lambda_0+2\mu_0}{\rho_0}}\quad\text{and}\quad c_S= \sqrt{\frac{\mu_0}{\rho_0}}.
\end{align*}

Let $\bu^{\rm sc}:{\RR^d} \to\CC^d$ be the scattered field generated by the interaction of the incident field $\bu^{\rm inc}$ with the inclusion $D$. If the two-dimensional operators $\Vcurl$ and $\curl$  are introduced by  
\begin{align*}
\Vcurl f:=(\partial_2 f, -\partial_1 f)^\top \quad\text{and}\quad \curl\,\bv:= \partial_1 v_2- \partial_2 v_1, 
\end{align*}
for all sufficiently smooth fields $f:\RR^2\to \RR$ and $\bw=(w_1, w_2)^\top :\RR^2 \to \RR^2$, then the pressure and shear parts of $\bu^{\rm sc}$ are, respectively, given by 
\begin{align*}
\bu^{\rm sc}_{P}(\bx) :=\ds  -\frac{1}{\K_P^2}\nabla\nabla\cdot\bu^{\rm sc}(\bx)
\quad\text{and}\quad
\bu^{\rm sc}_{S}(\bx) :=\frac{1}{\K_S^2}
\begin{cases}
 \ds \nabla\times\nabla\times\bu^{\rm sc}(\bx),  & d=3,
\\
\ds\Vcurl\curl\bu^{\rm sc}(\bx), & d=2,
\end{cases}
\end{align*}
for all $\bx\in \RR^d\setminus{\overline{D}}$. It can be easily verified that  $\bu^{\rm sc}_{P}$ and $\bu^{\rm sc}_{S}$ satisfy 
\begin{align*}
&\left(\Delta+\K_P^2\right)\bu^{\rm sc}_{P}=\mathbf{0} 
\quad\text{and}\quad
\left(\Delta+\K_S^2\right)\bu^{\rm sc}_{S}=\mathbf{0}  \quad\text{in }\RR^d\setminus\overline{D},
\\
&\nabla\cdot \bu^{\rm sc}_{S}=0 
\quad\text{and}\quad
\nabla\times \bu^{\rm sc}_{P}=\mathbf{0}\,\text{ for }\, d=3 
\quad\text{or}\quad
\curl\bu^{\rm sc}_{P}= {0}\,\text{ for } \,d=2. 
\end{align*}
The scattered field, $\bu^{\rm sc}$, is said to satisfy the \emph{Kupradze radiation conditions} if 
\begin{align*}
\ds\lim_{|\bx|\to+\infty} |\bx|^{(d-1)/2}\left(\frac{\partial \bu^{\rm sc}_{P}}{\partial |\bx|}-\i\K_P\bu^{\rm sc}_{P}\right)=\mathbf{0}
\quad\text{and}\quad
\ds\lim_{|\bx|\to+\infty} |\bx|^{(d-1)/2}\left(\frac{\partial \bu^{\rm sc}_{S}}{\partial |\bx|}-\i\K_S\bu^{\rm sc}_{S}\right)=\mathbf{0},
\end{align*} 
uniformly in all directions $\hbx\in\mathbb{S}^{d-1}:=\{\bx\in\RR^d\,:\,\, |\bx|=1\}$. Here $\hbx:=\bx/|\bx|$ for any $\bx\in\RR^d\setminus\{\mathbf{0}\}$ and $\partial/\partial |\bx|$ denotes the derivative in the radial direction. 

The propagation of the total elastic wave, $\bu^{\rm tot}:=\bu^{\rm sc}+\bu^{\rm inc}$,  in the presence of $D$, is governed by the  Lam\'e system,
\begin{equation}
\label{sys-utot}
\begin{cases}
\OL_{\lambda,\mu}\bu^{\rm tot}+\rho\omega^2\bu^{\rm tot}=\mathbf{0},  \quad \text{in }\RR^d, 
\\\nm 
\bu^{\rm sc}(\bx) \quad \text{satisfies the Kupradze radiation conditions as $|\bx|\to+\infty$}.
\end{cases}
\end{equation}
It is well known that the scattering problem \eqref{sys-utot}  is well-posed in $H^1(\RR^d\setminus\partial D)$. The interested readers are referred to see, e.g., \cite{Kup, Kup2}. 

The Kupradze radiation condition in Lam\'e system \eqref{sys-utot}  guarantees the existence of two vector fields  $\bu^{\infty}_{\epsilon,P}\in L^2_P(\mathbb{S}^{d-1})$ and $\bu^{\infty}_{\epsilon,S} \in L^2_S(\mathbb{S}^{d-1})$  characterized by the far-field asymptotic expansion of the scattered field as $|\bx|\to+\infty$, i.e., 
\begin{align*}
\bu^{\rm sc}(\bx)= \frac{e^{i\K_P|\bx|}}{|\bx|^{(d-1)/2}}\bu^{\infty}_{\epsilon,P}(\hbx)+ \frac{e^{i\K_S|\bx|}}{|\bx|^{(d-1)/2}}\bu^{\infty}_{\epsilon,S}(\hbx)+O\left(\frac{1}{|\bx|^{(d+1)/2}}\right), \qquad |\bx|\to +\infty,
\end{align*}
uniformly in all directions $\hbx\in\mathbb{S}^{d-1}$  (see, for instance, \cite{HahnerHsiao}).
The functions $\bu^{\infty}_{\epsilon,P}$ and $\bu^{\infty}_{\epsilon,S}$ are respectively called the longitudinal and transversal far-field patterns or far-field amplitudes of the scattered field $\bu^{\rm sc}$. Here 
$L^2_P(\mathbb{S}^{d-1})$ and $L^2_P(\mathbb{S}^{d-1})$ denote the spaces of square integrable longitudinal and transverse vector fields  respectively, i.e.,  
\begin{align*}
&L^2_P(\mathbb{S}^{d-1}):= \Big\{\bw\in L^2(\mathbb{S}^{d-1};\mathbb{C}^d)^d\, :\, \bw(\hbx)\times \hbx =\mathbf{0}, \quad \forall \hbx\in\mathbb{S}^{d-1}\Big\},
\\
&L^2_S(\mathbb{S}^{d-1}):= \Big\{\bw\in L^2(\mathbb{S}^{d-1};\mathbb{C}^d)^d\, :\, \bw(\hbx)\cdot \hbx =\mathbf{0}, \quad \forall \hbx\in\mathbb{S}^{d-1}\Big\}.
\end{align*}
Henceforth, the \emph{full far-field pattern}, $\bu^{\infty}_\epsilon$, is defined as the sum  $\bu^{\infty}_\epsilon(\hbx):=\bu^{\infty}_{\epsilon,P}(\hbx)+\bu^{\infty}_{\epsilon,S}(\hbx)$. 

We are now ready to introduce the following inverse scattering problem dealt with in this article.
\begin{prob}
\label{problem}
Let $\bu^{\rm inc}_j(\bx)$, for $j=1,\cdots, n\in\mathbb{N}$, be $n$ plane waves of the form \eqref{eq:plane} with  uniformly distributed directions of incidence $\btheta_j\in\mathbb{S}^{d-1}$. Let $\bu^{\infty}_{\epsilon,P,j}(\hbx)$ and $\bu^{\infty}_{\epsilon,S,j}(\hbx)$ be the longitudinal and transverse far-field patterns of the scattered fields radiated by the inclusion $D=\bz_D+\epsilon B_D$. Given the set of measurements,
$\{\bu^{\rm\infty}_{\epsilon,j}(\hbx):=\bu^{\infty}_{\epsilon,P,j}(\hbx)+\bu^{\infty}_{\epsilon,S,j}(\hbx):\,\, \forall\hbx\in\mathbb{S}^{d-1}, \,j=1,\cdots, n\}$,
find the location $\bz_D$ of the inclusion $D$. 
\end{prob}

This section is concluded with the following remarks. The Problem \ref{problem} turns out to be a \emph{single-short problem} when $n=1$. Secondly, although it is assumed that the measurements of the total far-field patterns are available, the  other situations when either the measurements of $\bu^{\infty}_{\epsilon,P,j}(\hbx)$ or $\bu^{\infty}_{\epsilon,S,j}(\hbx)$ are only accessible can be dealt with analogously and are amenable to the same treatment. Note also that the functions $\bu^{\infty}_{\epsilon,P}$ and $\bu^{\infty}_{\epsilon,S}$ are analytic on $\mathbb{S}^{d-1}$. Therefore, the measurements of these functions on any open subset of $\mathbb{S}^{d-1}$ lead to those on the entire $\mathbb{S}^{d-1}$ through analytic continuation so that the problem with limited-view measurements is also tractable.

\subsection{Preliminaries}\label{ss:prelim} 

In this article,  $(\hbe_1,\cdots,\hbe_d)$ denotes the standard basis for $\RR^d$, $\I_d\in\RR^{d\times d}$ denotes the identity matrix, and $\bx\otimes\bx=\bx\bx^T$ for any vector $\bx\in\RR^d$.  
Let $\mathbf{r}=(r_i)_{i=1}^d$, $\mathbf{A}=(a_{pq})_{p,q=1}^d$, $\mathbf{B}=(b_{pq})_{p,q=1}^d$, $\mathbb{F}=(f_{mpq})_{m,p,q=1}^d$, $\mathbb{G}=(g_{mpq})_{m,p,q=1}^d$, $\mathbb{H}=(h_{lmpq})_{l,m,p,q=1}^d$, and $\mathbb{J}=(j_{lmpq})_{l,m,p,q=1}^d$ be arbitrary.
Then, the following conventions are used henceforth: 
\begin{align*}
&\mathbf{A}\cdot\mathbf{r}:=\sum_{p,q=1}^d a_{pq}r_q\hbe_p=\mathbf{A}\mathbf{r},
\qquad
\left(\nabla\cdot\mathbf{A}\right)\mathbf{r} = \nabla\cdot\left(\mathbf{A}\mathbf{r}\right),
\qquad
\left(\nabla\times\mathbf{A}\right)\mathbf{r} = \nabla\times\left(\mathbf{A}\mathbf{r}\right),
\qquad
\left(\nabla\mathbf{A}\right)_{mpq}:=\partial_p a_{mq},
\\
&
\left(\left(\nabla\mathbf{A}\right)^\top\right)_{mpq}:=\partial_m a_{pq},
\qquad
\left(\nabla^2\mathbf{A}\right)_{lmpq}:=\partial_{lp} a_{mq},
\qquad
\mathbf{A}:\mathbf{B}:=\sum_{p,q=1}^d a_{pq}b_{pq}, 
\qquad 
\mathbb{F}:\mathbf{A}:=\sum_{m,p,q=1}^d f_{mpq}a_{pq}\hbe_m,
\\
& 
\mathbb{F}:\mathbb{G}:=\sum_{m,p,q=1}^d f_{mpq} g_{mpq},
\qquad
\mathbb{H}:\mathbf{A}:=\sum_{l,m,p,q=1}^d h_{lmpq}a_{pq}\hbe_l\otimes\hbe_m,
\qquad
\mathbb{H}:\mathbb{F}:=\sum_{l,m,p,q=1}^d h_{lmpq}f_{mpq} \hbe_m,
\\
&
\mathbb{H}\bullet\mathbb{G}:=\sum_{l,m,p,q,r=1}^d h_{lmpq}g_{pqr}\hbe_l\otimes\hbe_m\otimes\hbe_r,
\qquad
\mathbb{H}\bullet\mathbb{J}:=\sum_{l,m,p,q,r,s=1}^d h_{lmpq}j_{pqrs} \hbe_l\otimes\hbe_m\otimes\hbe_r\otimes\hbe_s,
\\
&
\|\mathbf{A}\|:=\sqrt{\mathbf{A}:\mathbf{A}},
\qquad
\|\mathbb{F}\|:=\sqrt{\mathbb{F}:\mathbb{F}}.
\end{align*}
Throughout, the notations  $\alpha$ and $\beta$ are reserved to mark quantities related to P- or S-wave modes, i.e., $\alpha,\beta=P, S$. We will also require the orthogonal projections $\Pk_P$ and $\Pk_S$, defined by 
\begin{align*}
\Pk_P[\bw]:= \bw_P\quad\text{and}\quad\Pk_S[\bw]:= \bw_S, \qquad \forall\bw\in L^2(\RR^d),
\end{align*}  
where $\bw_P$ and $\bw_{S}$ are respectively the pressure and shear parts of $\bw$ rendered by the Helmholtz  decomposition, i.e., 
\begin{align*}
\bw=\bw_P+\bw_S
\quad\text{with}\quad 
\nabla\cdot \bw_S = 0 
\quad\text{and}\quad
\nabla\times\bw_P=\mathbf{0}\,\text{ for }\, d=3 
\quad\text{or}\quad
{\rm curl}\,\bw_P=\mathbf{0}\,\text{ for }\, d=2.
\end{align*}

\subsubsection{Elastic moment tensor}\label{sss:EMT}
Let  $\bw_{pq}$, for $ l,m,p,q\in\{1,\cdots,d\}$,  be the solution  of the transmission problem 
\begin{align*}
\begin{cases}
\OL_{\lambda,\mu}\bw_{pq}=0,  & \text{in }\RR^d\backslash\partial B_D,
\\\nm
\bw_{pq}{\big|_-}=\bw_{pq}{\big|_+}, 
\qquad
\ds\frac{\partial}{\partial\widetilde{\bnu}}[\bw_{pq}]{\Bigg|_-}=\ds\frac{\partial}{\partial\bnu}[\bw_{pq}]{\Bigg|_+},  & \text{on } \partial B_D,
\\\nm
\bw_{pq}(\bx)-x_p\hbe_q=\mathcal{O}\left(|\bx|^{1-d}\right), &\text{as}\, |\bx|\to\infty,
\end{cases}
\end{align*} 
where $\bw_{pq}|_{\pm}(\bx) :=\lim_{t\to 0^+} \bw_{pq}(\bx\pm t\bnu(\bx))$, for all $\bx\in\partial B_D$. Then, the elastic moment tensor (EMT), $\MM(B_D)$, of the domain $B_D$, is defined by 
\begin{equation*}
m_{pq}^{lm}(B_D;\lambda_0,\lambda_1, \mu_0,\mu_1)=
\int_{\partial B_D}\left(\frac{\partial}{\partial\widetilde{\bnu}}[\xi_l\hbe_m]-\frac{\partial}{\partial\bnu}[\xi_l\hbe_m]\right)\cdot \bw_{pq}d\sigma.
\end{equation*}
It is well known that the EMT has the symmetries  $m_{pq}^{lm}=m^{pq}_{lm}=m_{pq}^{ml}=m_{qp}^{lm}$, which allow us to identify  $\MM(B_D)$ with a symmetric linear transformation on the space of symmetric $d\times d-$matrices. Further, if $\mu_1>\mu_0$ and $\lambda_1>\lambda_0$ then $\MM(B_D)$ is positive definite. Moreover, if $B_D$ is a ball in $\RR^d$ then 
\begin{align}
m_{pq}^{lm}(B_D)= \frac{a}{2}(\delta_{lp}\delta_{mq}+\delta_{lq}\delta_{mp})+ b\delta_{lm}\delta_{pq},\label{Mball}
\end{align}
with $\delta_{pq}$ representing the Kroneckar's delta function. The constants $a$ and $b$ (with no significance hereinafter) only depend on $d$ and the parameters $\mu_0,\mu_1,\lambda_0,\lambda_1$. The readers interested in further details  are referred to  \cite{AK-Pol, AKNT-02}.

\subsubsection{Kupradze matrix of fundamental solutions}\label{sss:KM}
The  fundamental solution ($\bGam^\omega(\bx,\by)$ or $\bGam^\omega(\bx-\by)$) of the time-harmonic elastic wave equation in $\RR^d$ with parameters $(\lambda_0,\mu_0,\rho_0)$, is defined as the solution to 
\begin{align*}
\OL_{\lambda_0,\mu_0}\bGam^\omega(\bx-\by)+\rho_0\omega^2\bGam^\omega(\bx-\by)=\delta_{\mathbf{0}}(\bx-\by)\I_d,\qquad\forall \bx\in\RR^d,   
\end{align*}
subject to the Kupradze outgoing radiation conditions. Here $\delta_\by$ is the Dirac mass at $\by$. It is well known (see \cite{Morse}) that
\begin{equation}
\label{Green_fun}
\bGam^\omega(\bx)=\frac{1}{\mu_0}\left[\dis\left(\I_d+\frac{1}{\K_S^2}\nabla\otimes\nabla\right)g^\omega_S(\bx)-\frac{1}{\K_S^2}\nabla\otimes\nabla g^\omega_P(\bx)\right],\quad \bx\in\R^d\setminus\{\mathbf{0}\},
\end{equation}
with  
\begin{align*}
g^\omega_\alpha(\bx)=
\ds\frac{i}{4} H_0^{(1)}(\K_{\alpha}|\bx|)
\,\text{ for }\,d=2
\quad\text{and}\quad
g^\omega_\alpha(\bx)=\ds\frac{e^{i\K_\alpha|\bx|}}{4\pi |\bx|} 
\,\text{ for }\, d=3, 
\qquad\forall\bx\in\RR^d\setminus\{\mathbf{0}\},
\end{align*}
where  $H_n^{(1)}$ is the order $n$ Hankel function of first kind. Hereinafter, $\bGam^\omega$ is decomposed into its P- and S-parts as 
\begin{align*}
&\bGam^\omega_P(\bx,\by):=\Pk_P\left[\bGam^\omega(\cdot,\by)\right](\bx)= -\frac{1}{\mu_0\K_S^2}\nabla_\bx\otimes\nabla_\bx g_P^\omega(\bx-\by),
\\
&\bGam^\omega_S(\bx,\by):=\Pk_S\left[\bGam^\omega(\cdot,\by)\right](\bx)=\frac{1}{\mu_0}\left(\I_d+\frac{1}{\K_S^2}\nabla_\bx\otimes\nabla_\bx \right)g_S^\omega(\bx-\by),
\end{align*}
such that $\nabla\cdot\bGam^\omega_S =\mathbf{0}$ and $\nabla\times\bGam^\omega_P =\mathbf{0}$ for $d=3$ or ${\rm curl}\,\bGam^\omega_P=\mathbf{0}$ for $d=2$.

Following identities will be  useful in the ensuing analysis (see, \ref{A:Identities} for proof).
\begin{prop}\label{Prop:identities} 
Let $n\in\mathbb{N}$ be sufficiently large. Let $\btheta_1,\cdots,\btheta_n\in\mathbb{S}^{d-1}$  be uniformly distributed directions and $\btheta_1^{\perp,\ell},\cdots,\btheta_n^{\perp,\ell}$ be such that $\{\btheta_j, \btheta_j^{\perp,\ell}\, :\, 1\leq \ell \leq d-1\}$ forms an orthonormal basis of $\RR^d$ for each $j=1,\cdots, n$. Then,
\begin{align}
\label{identity3}
 &\int_{\mathbb{S}^{d-1}} e^{i\K_\alpha\hbx\cdot(\by-\bz)}d\sigma(\hbx) = 4\left(\frac{\pi}{\K_\alpha}\right)^{d-2}\Im\big\{g^\omega_\alpha(\by-\bz)\big\},
\\
&\frac{1}{n}\sum_{j=1}^n\btheta_j\otimes\btheta_j e^{i\K_P(\by-\bz)\cdot\btheta_j}
\approx 
4\rho_0 c_P^2\left(\frac{\pi}{\K_P}\right)^{d-2}  \Im\left\{ \bGam^\omega_P(\by,\bz)\right\},
\label{identity4}
\\
&\frac{1}{n}\sum_{j=1}^n\sum_{\ell=1}^{d-1} \btheta_j^{\perp,\ell}\otimes\btheta_j^{\perp,\ell} e^{i\K_S(\by-\bz)\cdot\btheta_j}
\approx 
4\rho_0 c_S^2\left(\frac{\pi}{\K_S}\right)^{d-2}  \Im\left\{ \bGam^\omega_S(\by,\bz)\right\},
\label{identity6}
\\
&\frac{1}{n}\sum_{j=1}^n i\K_P \btheta_j\otimes\btheta_j\otimes\btheta_je^{i\K_P\btheta_j\cdot(\by-\bz)}\approx 4\rho_0 c_P^2\left(\frac{\pi}{\K_P}\right)^{d-2}\Im\left\{\nabla_\by\bGam^{\omega, 0}_P(\by,\bz)\right\},
\label{identity10}
\\
&\frac{1}{n}\sum_{j=1}^n e^{i\K_P(\by-\bz)\cdot\btheta_j}\btheta^j\otimes\btheta^j\otimes\btheta^j\otimes\btheta^j \approx
-4\rho_0 \frac{c_P^4}{\omega^2}\left(\frac{\pi}{\K_P}\right)^{d-2}\Im \left\{\nabla^2_{\by}\bGam_P(\by,\bz)\right\},
\label{identity7} 
\\ 
&\frac{1}{n}\sum_{j=1}^n\sum_{\ell=1}^{d-1} e^{i\K_S(\by-\bz)\cdot\btheta_j}\btheta_j\otimes\btheta_j^{\perp,\ell}\otimes\btheta_j\otimes\btheta_j^{\perp,\ell} \approx
-4\rho_0 \frac{c_S^4}{\omega^2}\left(\frac{\pi}{\K_S}\right)^{d-2}\Im \left\{\nabla^2_{\by}\bGam_S(\by,\bz)\right\}.
\label{identity8}
\end{align}
\end{prop}

\subsubsection{Asymptotic expansions}\label{sss:AE}

The following far-field expansions of $\bGam^\omega(\bx,\by)$  and $\nabla_\by\bGam^\omega(\bx,\by)$ hold  (see \ref{A:PropFar} for proof). 
\begin{prop}
\label{PropFar}
If $|\bx-\by|\to +\infty$ for a fixed $\by\in \RR^d$ then
\begin{align}
&\bGam^\omega(\bx,\by)= a^P_de^{-i\K_P\hbx\cdot\by}\frac{e^{i\K_P|\bx|}}{|\bx|^{(d-1)/2}}\hbx\otimes\hbx +a^S_d e^{-i\K_S\hbx\cdot\by} \frac{e^{i\K_S|\bx|}}{|\bx|^{(d-1)/2}}(\I_d-\hbx\otimes\hbx)
+O\left(\frac{1}{|\bx|^{(d+1)/2}}\right),
\label{GamFar}
\end{align}
\begin{align}
&\nabla_\by\bGam^\omega(\bx,\by)= i\K_P a^P_de^{-i\K_P\hbx\cdot\by}\frac{e^{i\K_P|\bx|}}{|\bx|^{(d-1)/2}}\hbx\otimes\hbx\otimes\hbx 
+i\K_S a^S_d e^{-i\K_S\hbx\cdot\by} \frac{e^{i\K_S|\bx|}}{|\bx|^{(d-1)/2}}\hbx\otimes(\I_d-\hbx\otimes\hbx)
+O\left(\frac{1}{|\bx|^{(d+1)/2}}\right),
\label{GrGamFar}
\end{align}
where  
\begin{align}
a^\beta_d:=\ds e^{i\pi/4}/ {\rho_0c^2_\beta}\sqrt{8\pi \K_\beta}\,\, \text{ for }\,\, d=2 
\qquad\text{and}\qquad
a^\beta_d:=1/4\pi {\rho_0c^2_\beta}\,\,\text{ for }\,\, d=3.
\label{a_d}
\end{align}
\end{prop}

Let us also recall the following asymptotic expansion of the scattered field versus scale factor $\epsilon$ (see, e.g., \cite{DirectElast, Princeton}).
\begin{thm}\label{thm:Asymp}
Let $\bu^{\rm tot}$ be the solution of Lam\'e system \eqref{sys-utot} and $\rho_1\omega^2$ be different from the Dirichlet eigenvalues of the operator $-\OL_{\lambda_1,\mu_1}$ in $L^2(D)^d$.  If  $\epsilon\omega\ll 1$ then the asymptotic expansion  
\begin{align*}
\bu^{\rm sc}(\bx) =& 
-\epsilon^d\Big(\nabla_{\bz_{D}}\bGam^\omega (\bx,\bz_{D}):\MM(B_D):\nabla\bu^{\rm inc}(\bz_{D})+\omega^2(\rho_0-\rho_1)|B_D|\bGam^\omega(\bx,\bz_{D})\cdot\bu^{\rm inc}(\bz_{D})\Big)
+{o}(\epsilon^d\omega^d),
\end{align*}
holds uniformly for all $\bx\in\RR^d\setminus\overline{D}$ far from $\partial D$.
\end{thm}

In view of the Proposition \ref{PropFar}, the following result is readily proved. 
\begin{cor} 
\label{corAsymp}
Under the assumptions of Theorem \ref{thm:Asymp}, for all $\bx\in \RR^d\setminus\overline{D}$  such that $|\bx-\bz_D|\to + \infty$, the far-field patterns $\bu^{\infty}_{\epsilon,P}$ and  $\bu^{\infty}_{\epsilon,S}$ admit the asymptotic expansions
\begin{align*}
&\bu^{\infty}_{\epsilon,P}(\hbx) = -\epsilon^d a^P_d e^{-i\K_P\hbx\cdot\bz_D}\Big[i\K_P(\hbx\otimes\hbx\otimes\hbx): \MM(B_D):\nabla\bu^{\rm inc}(\bz_D)
+\omega^2(\rho_0-\rho_1)|B_D|\hbx\otimes\hbx\cdot \bu^{\rm inc}(\bz_D)\Big]+o(\epsilon^d\omega^d),
\\\nm
&\bu^{\infty}_{\epsilon,S}(\hbx) = -\epsilon^d a^S_d e^{-i\K_S\hbx\cdot\bz_D}\Big[i\K_S\hbx\otimes(\I_d-\hbx\otimes\hbx) : \MM(B_D):\nabla\bu^{\rm inc}(\bz_D)+\omega^2(\rho_0-\rho_1)|B_D|(\I_d-\hbx\otimes\hbx)\cdot \bu^{\rm inc}(\bz_D)\Big]+o(\epsilon^d\omega^d).
\end{align*}
\end{cor}

\subsubsection{Elastic Herglotz wave fields}\label{sss:EHF}
If $\bw=\bw_P\oplus\bw_S\in L^2_P(\mathbb{S}^{d-1})\oplus L^2_S(\mathbb{S}^{d-1})$ then the superposition of plane waves, 
\begin{align*}
\mathcal{H}[\bw](\bx):=\int_{\mathbb{S}^{d-1}}\left\{e^{i\K_P\hbx\cdot\hby}\bw_P(\hby)+e^{i\K_S\hbx\cdot\hby}\bw_S(\hby)\right\} d\sigma(\hby),\qquad\bx\in\RR^d,
\end{align*}
is coined as the elastic Herglotz wave field with kernel $\bw$. Note that, for all $\bw\in L^2(\mathbb{S}^{d-1}; \mathbb{C}^d)^d$, the field $\mathcal{H}[\mathbf{\bw}]$ is an entire solution to the time-harmonic elastic wave equation. For further details on elastic Herglotz functions, the reader is suggested to consult Dassios and Rigou \cite{Dassios93,Dassios95}. 

\section{Location search using topological sensitivity indicators}\label{s:indicators}

In this investigation, a  $L^2$-discrepancy based cost functional and topological sensitivity based location indicators resulting therefrom are considered for identifying the   location $\bz_D$ of the inclusion $D$. Towards this end, a search point $\bz_S\in\RR^{d}$ is considered and  a trial inclusion $D_S:=\delta B_{S}+\bz_S$ is nucleated inside the background medium with scale factor $\delta\in\RR_+$, smooth reference domain $B_S$ and  corresponding parameters $(\lambda_2,\mu_2,\rho_2)$ which satisfy similar constraints as in Eqs. \eqref{constraint1} and \eqref{constraint2}. The interaction of the probing incident field $\bu^{\rm inc}$ with the medium containing trial inclusion $D_S$ generates the corresponding far-field patterns  $\bu^{\infty}_{\delta,P}$, $\bu^{\infty}_{\delta,S}$ and $\bu^{\infty}_{\delta}$, which satisfy similar equations and asymptotic expansions as $\bu^{\infty}_{\epsilon,P}$, $\bu^{\infty}_{\epsilon,S}$ and $\bu^{\infty}_{\epsilon}$ with $(\lambda_1,\mu_1,\rho_1, \bz_D, B_D, \epsilon)$ replaced by $(\lambda_2,\mu_2,\rho_2, \bz_S, B_S, \delta)$ under the assumption that $\omega\delta\ll 1$. 
Then, for any search point $\bz_S\in\RR^d$, the discrepancy functional 
\begin{align}
\mathcal{J}[\bu^{\rm inc}](\bz_S):=\frac{1}{2}\int_{\mathbb{S}^{d-1}} | \bu^{\infty}_\epsilon(\hbx) -\bu^{\infty}_\delta(\hbx) |^2 d\sigma(\hbx),
\label{J}
\end{align}
is considered in the hope that the point $\bz_S\in\RR^d$ for which far-field pattern $\bu^{\infty}_{\delta}$ minimizes $\mathcal{J}[\bu^{\rm inc}]$ is the sought point $\bz_D$. Instead of directly solving the optimization problem \eqref{J}, the topological derivative of $\mathcal{J}[\bu^{\rm inc}]$ is derived and the location indicators are constructed therefrom. Towards this end, the definition of the topological derivative of cost functional $\mathcal{J}[\bu^{\rm inc}]$ is reminded below.

\begin{defn} 
For any $\bz_S\in\RR^d$ and incident wave $\bu^{\rm inc}$, the topological derivative  of $\mathcal{J}[\bu^{\rm inc}]$ is defined as
\begin{align*}
\partial_T\mathcal{J}[\bu^{\rm inc}](\bz_S):= {\partial\mathcal{J}[\bu^{\rm inc}]}/{\partial\delta^d}\big|_{\delta=0}(\bz_S).
\end{align*}
\end{defn}

\subsection{Location indicators}
The first topological sensitivity based single-short location indicator  is defined as
\begin{align}
\mathcal{I}_{\rm TD}[\bu^{\rm inc}](\bz_S):=-\partial_T\mathcal{J}[\bu^{\rm inc}](\bz_S), \quad\forall\bz_S\in\RR^d.
\label{ITD-Def}
\end{align}
In the same spirit, the corresponding multi-short location indicator can be constructed from indicator \eqref{ITD-Def} for plane incident fields of the form \eqref{eq:plane}. Precisely, for $n$ uniformly distributed directions  $\btheta_1,\cdots,\btheta_n\in\mathbb{S}^{d-1}$ and the corresponding incident plane waves $\bu_j^P$  or $\bu_{j,\ell}^S$ defined as 
\begin{align}
\label{Uj}
\bu^{P}_j(\bx) := \btheta_j e^{i\K_P\hbx\cdot\btheta_j}\quad\text{and}\quad 
\bu^{S}_{j,\ell}(\bx) := \btheta_j^{\perp,\ell}  e^{i\K_S\hbx\cdot\btheta_j}, \quad 1\leq \ell\leq d-1,
\end{align}
the multi-short location indicator $\mathcal{I}_{\rm TD}^n:\RR^d\to \RR$ is defined as
\begin{align}
\label{ITDn-def}
\mathcal{I}_{\rm TD}^n(\bz_S; P) :=\frac{1}{n}\sum_{j=1}^n \mathcal{I}_{\rm TD}[\bu^{P}_j](\bz_S) 
\quad\text{or}\quad
\mathcal{I}_{\rm TD}^n(\bz_S; S) :=\frac{1}{n}\sum_{j=1}^n \sum_{\ell=1}^{d-1}\mathcal{I}_{\rm TD}[\bu^{S}_{j,\ell}](\bz_S).
\end{align}

By construction,  if the contrasts $(\lambda_2-\lambda_0)$ and $(\mu_2-\mu_0)$ have the same signs as $(\lambda_1-\lambda_0)$ and $(\mu_1-\mu_0)$, respectively, then the indicators $\mathcal{I}_{\rm TD}[\bu^{\rm inc}](\bz_S)$ and $\mathcal{I}_{\rm TD}^n(\bz_S;\alpha)$, defined in Eq. \eqref{ITD-Def} and \eqref{ITDn-def},  are expected to achieve the most prominent increase at  $\bz_S$ near $\bz_D$ and should decay very rapidly for $\bz_S$ away from $\bz_D$. It is simply due to the fact that $\mathcal{I}_{\rm TD}[\bu^{\rm inc}](\bz_S)$ is  negative of the derivative of  $\mathcal{J}[\bu^{\rm inc}]$ which should observe most significant decrease near $\bz_D$. In order to verify these capabilities of the indicators $\bz_S\mapsto\mathcal{I}_{\rm TD}[\bu^{\rm inc}](\bz_S)$ and $\bz_S\mapsto\mathcal{I}_{\rm TD}^n(\bz_S;\alpha)$, the sensitivity and resolution analysis will be performed in Section \ref{ss:sensitivity}. Towards this end, the expression for $\mathcal{I}_{\rm TD}[\bu^{\rm inc}]$ is derived in Theorem \ref{Thm:TD}. 

\begin{thm}\label{Thm:TD}
For  all $\bz_S\in\RR^d$,
\begin{align}
\mathcal{I}_{\rm TD}[\bu^{\rm inc}](\bz_S)
= &\sum_{\alpha=P,S}\Re\Bigg\{ a^\alpha_d\overline{\nabla\mathcal{H}\left[\bu^{\infty}_{\epsilon,\alpha}\right](\bz_S)}:\MM(B_S):\nabla\bu^{\rm inc}(\bz_S) 
 - a^\alpha_d\omega^2(\rho_0-\rho_2)|B_S|\overline{\mathcal{H}\left[\bu^{\infty}_{\epsilon,\alpha}\right](\bz_S)}\cdot\bu^{\rm inc}(\bz_S)\Bigg\}.
\label{ITD}
\end{align}
\end{thm}

\begin{proof}
Let us first expand $\mathcal{J}[\bu^{\rm inc}](\bz_S)$ as 
\begin{align*}
\mathcal{J}[\bu^{\rm inc}](\bz_S)= \frac{1}{2}\int_{\mathbb{S}^{d-1}}\left|\bu^{\infty}_\epsilon(\hbx)\right|^2d\sigma(\hbx) - \Re\left\{\int_{\mathbb{S}}\overline{\bu^{\infty}_\epsilon(\hbx)}\cdot\bu^{\infty}_\delta(\hbx)d\sigma(\bx)\right\}+ \frac{1}{2}\int_{\mathbb{S}^{d-1}}\left|\bu^{\infty}_\delta(\hbx)\right|^2d\sigma(\hbx).
\end{align*}
In view of Theorem \ref{thm:Asymp}, the last term on the RHS is of order $o(\delta^{2d})$. Consequently,
\begin{align*}
\mathcal{J}&[\bu^{\rm inc}](\bz_S)=\frac{1}{2}\int_{\mathbb{S}^{d-1}}\left|\bu^{\infty}_\epsilon(\hbx)\right|^2d\sigma(\hbx) 
- \Re\left\{\int_{\mathbb{S}^{d-1}}\overline{\bu^{\infty}_\epsilon(\hbx)}\cdot\bu^{\infty}_\delta(\hbx)d\sigma(\hbx)\right\}+o(\delta^{2d}).
\end{align*}
Note  that the total far-field  pattern $\bu^\infty_\delta$ can be expanded as well using  Corollary \ref{corAsymp}  so that 
\begin{align}
\mathcal{J}[\bu^{\rm inc}](\bz_S)
=&
\frac{1}{2}\int_{\mathbb{S}^{d-1}}\left|\bu^{\infty}_\epsilon(\hbx)\right|^2d\sigma(\hbx) 
+\Re\left\{i\K_Pa^P_d\delta^d\int_{\mathbb{S}^{d-1}} \overline{(\bu^{\infty}_\epsilon(\hbx)e^{i\K_P\hbx\cdot\bz_S})}\cdot\left[(\hbx\otimes\hbx\otimes\hbx):\MM(B_S):\nabla\bu^{\rm inc}(\bz_S)\right]d\sigma(\hbx)\right\}
\nonumber
\\
& +\Re\left\{a^P_d\omega^2(\rho_0-\rho_2)|B_S|\delta^d \int_{\mathbb{S}^{d-1}}\overline{(\bu^{\infty}_\epsilon(\hbx)e^{i\K_P\hbx\cdot\bz_S})}\cdot\left[(\hbx\otimes\hbx)\cdot\bu^{\rm inc}(\bz_S)\right]d\sigma(\hbx)\right\}
\nonumber
\\
& +\Re\left\{i\K_Sa^S_d\delta^d \int_{\mathbb{S}^{d-1}}\overline{(\bu^{\infty}_\epsilon(\hbx)e^{i\K_S\hbx\cdot\bz_S})}\cdot\left[(\hbx\otimes[\I_d-\hbx\otimes\hbx]):\MM(B_S):\nabla\bu^{\rm inc}(\bz_S)\right]d\sigma(\hbx)\right\}
\nonumber
\\
& +\Re\left\{a^S_d\omega^2(\rho_0-\rho_2)|B_S|\delta^d \int_{\mathbb{S}^{d-1}}\overline{(\bu^{\infty}_\epsilon(\hbx)e^{i\K_S\hbx\cdot\bz_S})}\cdot\left[\I_d - (\hbx\otimes\hbx)\cdot\bu^{\rm inc}(\bz_S)\right]d\sigma(\hbx)\right\}
+ o(\delta^{2d}).
\label{29b}
\end{align}
Recall that, for any vector $\bw$,
\begin{align}
\left(\hbx\otimes\hbx\right)\bw=\left(\hbx\cdot\bw\right)\hbx
\quad\text{and}\quad
\left(\I_d- \hbx\otimes\hbx\right)\bw=
\begin{cases}
\left(\hbx^\perp\cdot\bw \right)\hbx^\perp, & d=2,
\\\nm
\hbx\times\left(\bw \times\hbx\right), & d=3.
\end{cases}
\label{identities1}
\end{align}
Moreover, since  $\bu^{\infty}_{\epsilon,P}\in L^2_P(\mathbb{S}^{d-1})$ and $\bu^{\infty}_{\epsilon,S}\in L^2_S(\mathbb{S}^{d-1})$, it is easy to see that
\begin{align}
\bu^{\infty}_{\epsilon,P}(\hbx)=(\hbx\cdot\bu^{\infty}_\epsilon(\hbx))\hbx
\quad\text{and}\quad
\bu^{\infty}_{\epsilon,S}(\hbx)=
\begin{cases}
(\hbx^\perp\cdot\bu^{\infty}_\epsilon(\hbx))\hbx^\perp ,  &d=2,
\\\nm
\hbx\times (\bu^{\infty}_\epsilon(\hbx)\times\hbx),  &d=3.
\end{cases}
\label{uPuS}
\end{align}
Consequently, the expression \eqref{29b}, together with the identities \eqref{identities1} and \eqref{uPuS}, renders
\begin{align*}
\mathcal{J}[\bu^{\rm inc}](\bz_S)=&\frac{1}{2}\int_{\mathbb{S}^{d-1}}\left|\bu^{\infty}_\epsilon(\hbx)\right|^2d\sigma(\hbx) 
+\sum_{\alpha=P,S}\Bigg[
\Re\left\{i\K_\alpha a^\alpha_d\delta^d\int_{\mathbb{S}^{d-1}} \overline{\hbx\otimes \bu^{\infty}_{\epsilon,\alpha}(\hbx) e^{i\K_\alpha\hbx\cdot\bz_S}} d\sigma(\hbx):\MM(B_S):\nabla\bu^{\rm inc}(\bz_S)\right\}
\\
& +\Re\left\{a^\alpha_d\omega^2(\rho_0-\rho_2)|B_S|\delta^d \int_{\mathbb{S}^{d-1}} \overline{\bu^{\infty}_{\epsilon,\alpha}(\hbx) e^{i\K_\alpha\hbx\cdot\bz_S}} d\sigma(\hbx)\cdot\bu^{\rm inc}(\bz_S)\right\}
\Bigg]
+o(\delta^{2d})
\\
= &\sum_{\alpha=P,S}\Bigg[
-\Re\left\{ a^\alpha_d\delta^d\overline{\nabla\int_{\mathbb{S}^{d-1}} \bu^{\infty}_{\epsilon,\alpha}(\hbx) e^{i\K_\alpha\hbx\cdot\bz_S} d\sigma(\hbx)}:\MM(B_S):\nabla\bu^{\rm inc}(\bz_S)\right\}
\\
& +\Re\left\{a^\alpha_d\omega^2(\rho_0-\rho_2)|B_S|\delta^d \int_{\mathbb{S}^{d-1}} \overline{\bu^{\infty}_{\epsilon,\alpha}(\hbx) e^{i\K_\alpha\hbx\cdot\bz_S}} d\sigma(\hbx)\cdot\bu^{\rm inc}(\bz_S)\right\}
\Bigg]
+ o(\delta^{2d}).
\end{align*}
Finally, using the definition of the Herglotz wave field with kernel $\bu^{\infty}_{\epsilon,\alpha}$, one arrives at
\begin{align*}
\frac{1}{\delta^d}&\left(\mathcal{J}[\bu^{\rm inc}](\bz_S)-\frac{1}{2}\int_{\mathbb{S}^{d-1}}\left|\bu^{\infty}_\epsilon(\hbx)\right|^2d\sigma(\hbx) 
\right) 
\\
&=\sum_{\alpha=P,S}\Bigg[
-\Re\left\{ a^\alpha_d\overline{\nabla \mathcal{H}\left[\bu^{\infty}_{\epsilon,\alpha}\right](\bz_S)}:\MM(B_S):\nabla\bu^{\rm inc}(\bz_S)\right\}
+\Re\left\{a^\alpha_d\omega^2(\rho_0-\rho_2)|B_S|\overline{\mathcal{H}\left[\bu^{\infty}_{\epsilon,\alpha}\right](\bz_S)}\cdot\bu^{\rm inc}(\bz_S)\right\}
\Bigg]
+ o(\delta^{d}).
\end{align*}
By passing the limit $\delta^d\to 0$, the topological derivative of the cost functional $\mathcal{J}[\bu^{\rm inc}](\bz_S)$  is obtained and the required expression \eqref{ITD} for $\mathcal{I}_{\rm TD}[\bu^{\rm inc}]$ follows immediately.

\end{proof}

Before proceeding to the analysis of the topological sensitivity of $\mathcal{I}_{\rm TD}[\bu^{\rm inc}]$,  note that the asymptotic expansions of the Herglotz wave fields $\mathcal{H}[\bu^{\infty}_{\epsilon,\alpha}]$ as $\epsilon^d\to 0$ are required. Towards this end, the following result holds.  

\begin{lem}\label{LemHexp}
Let $\rho_1\omega^2$ be different from the eigenvalues of  $-\OL_{\lambda_1,\mu_1}$ in $L^2(D)^d$ and $\omega\epsilon\ll 1$. Then, for all $\bz\in\RR^d$,
\begin{align*}
\mathcal{H}\left[\bu^{\infty}_{\epsilon,\alpha} \right](\bz) 
=&
-4\rho_0c_\alpha^2 a^\alpha_d\epsilon^d\left(\frac{\pi}{\K_\alpha}\right)^{d-2}\left(\nabla_{\bz} \Im\big\{\bGam^\omega_\alpha(\bz,\bz_D)\big\}:\MM(B_D):\nabla\bu^{\rm inc}(\bz_D) 
+\omega^2(\rho_0-\rho_1)|B_D|\Im\big\{\bGam^\omega_\alpha(\bz,\bz_D)\big\}\cdot\bu^{\rm inc}(\bz_D)\right)
\nonumber
\\
&+o(\epsilon^d\omega^d).
\end{align*}
\end{lem}

\begin{proof}

Let us first consider the case $\alpha=P$.  By definition of the Herglotz wave field and Corollary \ref{corAsymp},  for all $\bz_S\in\RR^d$,
\begin{align*}
\mathcal{H}\left[\bu^{\infty}_{\epsilon,P} \right](\bz) 
=&
-i\K_P a^P_d\epsilon^d \int_{\mathbb{S}^{d-1}} e^{i\K_P\hbx\cdot(\bz-\bz_D)}(\hbx\otimes\hbx\otimes\hbx):\MM(B_D):\nabla\bu^{\rm inc}(\bz_D) d\sigma(\hbx)
\\
&-a^P_d\omega^2(\rho_0-\rho_1)|B_D|\epsilon^d\int_{\mathbb{S}^{d-1}} e^{i\K_P\hbx\cdot(\bz-\bz_D)}(\hbx\otimes\hbx)\cdot\bu^{\rm inc}(\bz_D) d\sigma(\hbx)+o(\epsilon^d\omega^d),
\end{align*}
or  equivalently, 
\begin{align*}
\mathcal{H}\left[\bu^{\infty}_{\epsilon,P} \right](\bz) 
=&\frac{a^P_d\epsilon^d}{\K_P^2}\nabla_{\bz}\otimes\nabla_{\bz}\otimes \nabla_{\bz} \int_{\mathbb{S}^{d-1}} \left( e^{i\K_P\hbx\cdot(\bz-\bz_D)} \right)d\sigma(\hbx) :\MM(B_D):\nabla\bu^{\rm inc}(\bz_D) 
\\
&+\frac{a^P_d\omega^2(\rho_0-\rho_1)|B_D|\epsilon^d}{\K_P^2}\nabla_{\bz}\otimes \nabla_{\bz}\int_{\mathbb{S}^{d-1}} \left( e^{i\K_P\hbx\cdot(\bz-\bz_D)} \right)d\sigma(\hbx)\cdot\bu^{\rm inc}(\bz_D)
+o(\epsilon^d\omega^d).
\end{align*}
This, together with the identity \eqref{identity3}, yields  
\begin{align*}
\mathcal{H}\left[\bu^{\infty}_{\epsilon,P} \right](\bz) 
=&
4\left(\frac{\pi}{\K_P}\right)^{d-2}\frac{a^P_d\epsilon^d}{\K_P^2}\nabla_{\bz}\otimes \left(\nabla_{\bz}\otimes \nabla_{\bz}\Im\big\{g^\omega_P(\bz-\bz_D)\big\}\right):\MM(B_D):\nabla\bu^{\rm inc}(\bz_D) 
\\
&+4\left(\frac{\pi}{\K_P}\right)^{d-2}\frac{a^P_d\omega^2(\rho_0-\rho_1)|B_D|\epsilon^d}{\K_P^2}\left(\nabla_{\bz}\otimes \nabla_{\bz}\Im\big\{g^\omega_P(\bz-\bz_D)\big\}\right)\cdot\bu^{\rm inc}(\bz_D)
+o(\epsilon^d\omega^d).
\end{align*}
Finally, by the definition of $\bGam^\omega_P$, it can be easily seen that
\begin{align*}
\mathcal{H}\left[\bu^{\infty}_{\epsilon,P} \right](\bz) 
=&
-4\left(\frac{\pi}{\K_P}\right)^{d-2}\rho_0c_P^2 a^P_d\epsilon^d\nabla_{\bz} \Im\big\{\bGam^\omega_P(\bz,\bz_D)\big\}:\MM(B_D):\nabla\bu^{\rm inc}(\bz_D) 
\nonumber
\\
&-4\left(\frac{\pi}{\K_P}\right)^{d-2}\rho_0 c_P^2 a^P_d\omega^2(\rho_0-\rho_1)|B_D|\epsilon^d\Im\big\{\bGam^\omega_P(\bz,\bz_D)\big\}\cdot\bu^{\rm inc}(\bz_D)
+o(\epsilon^d\omega^d). 
\end{align*}
For $\alpha=S$, Corollary \ref{corAsymp} is once again invoked so that, for all $\bz_S\in\RR^d$,
\begin{align*}
\mathcal{H}\left[\bu^{\infty}_{\epsilon,S} \right](\bz) 
=&
-i\K_S a^S_d\epsilon^d \int_{\mathbb{S}^{d-1}} e^{i\K_S\hbx\cdot(\bz-\bz_D)}\hbx\otimes(\I_d-\hbx\otimes\hbx):\MM(B_D):\nabla\bu^{\rm inc}(\bz_D) d\sigma(\hbx)
\\
&-a^S_d\omega^2(\rho_0-\rho_1)|B_D|\epsilon^d\int_{\mathbb{S}^{d-1}} e^{i\K_S\hbx\cdot(\bz-\bz_D)}(\I_d-\hbx\otimes\hbx)\cdot\bu^{\rm inc}(\bz_D) d\sigma(\hbx)+o(\epsilon^d\omega^d).
\end{align*}
Again using identity \eqref{identity3}, one gets
\begin{align*}
\mathcal{H}\left[\bu^{\infty}_{\epsilon,S} \right](\bz) 
=&
-{a^S_d\epsilon^d}\nabla_{\bz}\left(\I_d+\frac{1}{\K_S^2}\nabla_{\bz}\otimes \nabla_{\bz}\right)\int_{\mathbb{S}^{d-1}} \left( e^{i\K_S\hbx\cdot(\bz-\bz_D)} \right)d\sigma(\hbx) :\MM(B_D):\nabla\bu^{\rm inc}(\bz_D) 
\\
&-{a^S_d\omega^2(\rho_0-\rho_1)|B_D|\epsilon^d}\left(\I_d+\frac{1}{\K_S^2}\nabla_{\bz}\otimes \nabla_{\bz}\right)\int_{\mathbb{S}^{d-1}} \left( e^{i\K_S\hbx\cdot(\bz-\bz_D)} \right)d\sigma(\hbx)\cdot\bu^{\rm inc}(\bz_D)
+o(\epsilon^d\omega^d) 
\\
=&
-4\left(\frac{\pi}{\K_S}\right)^{d-2}{a^S_d\epsilon^d}\nabla_{\bz}\left(\I_d+\frac{1}{\K_S^2}\nabla_{\bz}\otimes \nabla_{\bz}\right)\left[\Im\big\{g_S^\omega(\bz-\bz_D)\big\}\right] :\MM(B_D):\nabla\bu^{\rm inc}(\bz_D) 
\\
&-4\left(\frac{\pi}{\K_S}\right)^{d-2}{a^S_d\omega^2(\rho_0-\rho_1)|B_D|\epsilon^d}\left(\I_d+\frac{1}{\K_S^2}\nabla_{\bz}\otimes \nabla_{\bz}\right)\left[\Im\big\{g_S^\omega(\bz-\bz_D)\big\}\right]\cdot\bu^{\rm inc}(\bz_D)
+o(\epsilon^d\omega^d).
\end{align*}
The desired result  easily follows from the expression of  $\bGam^\omega_S$. This completes the proof.
\end{proof}

\subsection{Sensitivity and resolution analysis}\label{ss:sensitivity}

Let us now discuss the sensitivity of the location indicators. Remark that   $\mathcal{I}_{\rm TD}[\bu^{\rm inc}]$ depends on the back-propagator  $\mathcal{H}[\bu^\infty_{\epsilon,\alpha}]$ which consists of two terms, a \emph{density contrast} term (involving  $(\rho_0-\rho_1)$) and an \emph{elastic contrast} term (involving $\MM(B_D)$), as suggested by Lemma \ref{LemHexp}. Accordingly, two separate cases are considered for brevity. Specifically, the inclusion is assumed to have either a density contrast (i.e., $\lambda_0=\lambda_1$, $\mu_0=\mu_1$ and $\rho_0\neq\rho_1$) or an elasticity contrast (i.e., $\lambda_0\neq\lambda_1$, $\mu_0\neq\mu_1$ and $\rho_0=\rho_1$) with the background medium throughout in this article. The general case is delicate but amenable to the same treatment.

\subsubsection{Density contrast}

If it is a priori known that the true inclusion has only a density contrast with the background medium then  the trial inclusion is also nucleated with a density contrast only. 
Moreover, it is assumed that the density contrasts of the trial and the true inclusions  with the background have the same signs, i.e., the parameter $\rho_2$ is such that $(\rho_0-\rho_1)(\rho_0-\rho_2) > 0$. For ease of notation, a parameter $\tau$ is introduced by
\begin{align*}
\tau:=|B_S||B_D|(\rho_0-\rho_1)(\rho_0-\rho_2).
\end{align*}

The location indicator $\mathcal{I}_{\rm TD}[\bu^{\rm inc}]$ and back-propagator $\mathcal{H}[\bu^{\infty}_{\epsilon,\alpha}]$ in the density contrast case appear to be 
\begin{align}
\label{ITD-DC}
&\mathcal{I}_{\rm TD}[\bu^{\rm inc}](\bz_S)= -\sum_{\alpha=P,S} \Re\left\{a^\alpha_d\omega^2(\rho_0-\rho_2)|B_S|\overline{\mathcal{H}[\bu^{\infty}_{\epsilon,\alpha}](\bz_S)}\cdot\bu^{\rm inc}(\bz_S)\right\},
\\
\label{Herglotz-DC}
&\mathcal{H}[\bu^{\infty}_{\epsilon,\alpha}](\bz_S) \approx -4\left(\frac{\pi}{\K_\alpha}\right)^{d-2}\rho_0 c_\alpha^2 a^\alpha_d \omega^2(\rho_0-\rho_1)|B_D|\epsilon^d \Im\left\{\bGam^\omega_\alpha(\bz_S,\bz_D)\right\}\cdot \bu^{\rm inc}(\bz_D),\qquad\forall\bz_S\in\RR^d.
\end{align}
It is worthwhile remarking that the back-propagators $\mathcal{H}[\bu^{\infty}_{\epsilon, P}](\bz_S)$ and $\mathcal{H}[\bu^{\infty}_{\epsilon, S}](\bz_S)$ are respectively solenoidal and irrotational in the regime $\epsilon^d\to 0$. In fact, it is easy to check, e.g., for $d=3$,  that
\begin{align*}
&\nabla\times\mathcal{H}[\bu^{\infty}_{\epsilon, P}](\bz_S)
\propto 
\nabla\times\left(\Im\left\{\bGam^\omega_P(\bz_S,\bz_D)\right\}\cdot\bu^{\rm inc}(\bz_D)\right) 
=
\left(\nabla\times\Im\left\{\bGam^\omega_P(\bz_S,\bz_D)\right\}\right)\cdot\bu^{\rm inc}(\bz_D)
= \mathbf{0},
\\
&\nabla\cdot\mathcal{H}[\bu^{\infty}_{\epsilon, S}](\bz_S)
\propto 
\nabla\cdot\left(\Im\left\{\bGam^\omega_S(\bz_S,\bz_D)\right\}\cdot\bu^{\rm inc}(\bz_D)\right) 
=
\left(\nabla\cdot\Im\left\{\bGam^\omega_S(\bz_S,\bz_D)\right\}\right)\cdot\bu^{\rm inc}(\bz_D)
=0.
\end{align*}
This observation, together with \eqref{Herglotz-DC} in \eqref{ITD-DC}, leads us to the simplified form
\begin{align}
\mathcal{I}_{\rm TD}[\bu^{\rm inc}](\bz_S)\approx \omega^3\gamma_d\tau\epsilon^d \Re\left\{\sum_{\alpha=P,S}\frac{1}{c_\alpha}\Im\left\{\bGam^\omega_\alpha(\bz_S,\bz_D)\right\} \overline{\bu^{\rm inc}(\bz_D)}\cdot\Pk_\alpha[\bu^{\rm inc}](\bz_S)
\right\}, 
\label{ITD-density}
\end{align}
where $\gamma_d :={1}/({2^{d-1}\pi \rho_0})$. In particular, when $\bu^{\rm inc}$ is a  $\beta-$plane wave, with $\beta=P$ or $S$,
\begin{align}
\mathcal{I}_{\rm TD}[\bu^\beta](\bz_S) 
\approx 
& \omega^3\gamma_d\tau\epsilon^d \Re\left\{\sum_{\alpha=P,S}\frac{1}{c_\alpha}\Im\left\{\bGam^\omega_\alpha(\bz_S,\bz_D)\right\} \overline{\bu^\beta(\bz_D)}\cdot\Pk_\alpha[\bu^{\beta}](\bz_S)
\right\}
\nonumber
\\
=&  \frac{\omega^3\gamma_d\tau\epsilon^d }{c_\beta}\Re\left\{\Im\left\{\bGam^\omega_\beta(\bz_S,\bz_D)\right\} \overline{\bu^\beta(\bz_D)}\cdot\bu^\beta(\bz_S)
\right\}. 
\label{ITD-densityB}
\end{align}

The expression \eqref{ITD-densityB} clearly shows that  if the medium is probed by an incident $\beta-$wave then the single-short location indicator is proportional to the imaginary part of $\beta-$component of the Kupradze matrix, which observes a sharp peak at the point $\bz_S$ near the true location $\bz_D$ of the inclusion $D$ with a focal spot size of the order of half the wavelength of $\beta-$wave. For instance, when $\beta=S$ and $d=3$, 
\begin{align*}
\mathcal{I}_{\rm TD}[\bu^S](\bz_S) 
\propto
&\cos (\K_S\btheta\cdot(\bz_S-\bz_D)) \Im\left\{\bGam^\omega_S(\bz_S,\bz_D)\right\} \btheta^\perp\cdot\btheta^\perp 
\nonumber
\\
\propto 
&\cos (\K_S\btheta\cdot \mathbf{r})\left[\frac{2}{3}j_0\left(\K_S|\mathbf{r}|\right)\I_3+j_2\left(\K_S|\mathbf{r}|\right)\left(\hat{\mathbf{r}}\otimes \hat{\mathbf{r}}-\frac{1}{3}\I_3\right)\right] \btheta^\perp\cdot\btheta^\perp,
\end{align*}
where $\mathbf{r}:=\bz_S-\bz_D$ and $j_m$ is the spherical Bessel function of first kind and order $m$. Recall that $j_m(\K_S|\mathbf{r}|)= O(1/\K_S|\mathbf{r}|)$ as  $\K_S|\mathbf{r}| \to +\infty$ 
(see, e.g., \cite[Sec. 10.52]{NIST}).  Moreover, its focal spot size is of the order of the wavelength of S-wave. Since $\cos(\K_S\btheta\cdot\mathbf{r})$ is a bounded function, it is evident that $\bz_S\mapsto\mathcal{I}_{\rm TD}[\bu^S]$ observes a very prominent peak that has diameter determined by the wavelength of S-wave as $\bz_S\to\bz_D$. In general, it tells us that the location indicator $\mathcal{I}_{\rm TD}[\bu^{\rm inc}]$ in density contrast case is felicitous to identify $\bz_D$ with Rayleigh resolution.

Let us now discuss the sensitivity of the multi-short location indicator. The following result holds. 
\begin{lem}\label{LemTD-density}
If  $n\in\NN$ is sufficiently large then for any search point $\bz_S\in\RR^d$, 
\begin{align*}
\mathcal{I}_{\rm TD}^n(\bz_S; \alpha) \approx 4\rho_0\omega^3\tau\gamma_dc_\alpha \epsilon^d\left( \frac{\pi}{\K_\alpha}\right)^{d-2}\Phi_{\alpha,\alpha}(\bz_S,\bz_D),
\end{align*}
where $\Phi_{\alpha,\beta}:\RR^d\times\RR^d\to\RR$, for $\alpha,\beta=P,S$, is defined as
\begin{align}
\label{T}
\Phi_{\alpha,\beta}(\bz,\bz'):=\Im\left\{\bGam^\omega_\alpha(\bz,\bz')\right\}:\Im\left\{\bGam^\omega_\beta(\bz,\bz')\right\}.
\end{align}
\end{lem}
\begin{proof}
For  $\alpha=P$, by virtue of expression \eqref{ITD-density}, 
\begin{align*}
\mathcal{I}_{\rm TD}^n(\bz_S; P) \approx & \frac{1}{n}\sum_{j=1}^n \omega^3\gamma_d\tau\epsilon^d \Re\left\{\frac{1}{c_P}\Im\left\{\bGam^\omega_P(\bz_S,\bz_D)\right\} \overline{\bu^P_j(\bz_D)}\cdot\bu^P_j(\bz_S)
\right\}.
\end{align*}
Substituting the expression  \eqref{Uj} for $\bu^P_j$ and using $\mathbf{A}\btheta:\btheta=\mathbf{A}:\btheta\otimes\btheta$ for arbitrary  $\mathbf{A}\in\RR^{d\times d}$ and $\btheta\in\RR^d$, one gets
\begin{align*}
\mathcal{I}_{\rm TD}^n(\bz_S; P) \approx &  \omega^3\gamma_d\tau\epsilon^d \Re\left\{\frac{1}{c_P}\Im\left\{\bGam^\omega_P(\bz_S,\bz_D)\right\}: \frac{1}{n}\sum_{j=1}^n\btheta_j\otimes\btheta_j e^{i\K_P(\bz_S-\bz_D)\cdot\btheta_j}
\right\}.
\end{align*}
Therefore, the desired result is obtained by invoking approximation \eqref{identity4}.

Similarly, for $\alpha=S$, 
\begin{align*}
\mathcal{I}_{\rm TD}^n(\bz_S; S) \approx & \frac{1}{n}\sum_{j=1}^n \sum_{\ell=1}^{d-1}\omega^3\gamma_d\tau\epsilon^d \Re\left\{\frac{1}{c_S}\Im\left\{\bGam^\omega_S(\bz_S,\bz_D)\right\}\overline{\bu^S_{j,\ell}(\bz_D)}\cdot\bu^S_{j,\ell}(\bz_S)
\right\}.
\end{align*}
Substituting the expression \eqref{Uj}  for $\bu^S_{j,\ell}$ and once again using   $\mathbf{A}\btheta:\btheta=\mathbf{A}:\btheta\otimes\btheta$, one obtains
\begin{align*}
\mathcal{I}_{\rm TD}^n(\bz_S; S) \approx&  \omega^3\gamma_d\tau\epsilon^d \Re\left\{\frac{1}{c_S}\Im\left\{\bGam^\omega_S(\bz_S,\bz_D)\right\}: \frac{1}{n}\sum_{j=1}^n\sum_{\ell=1}^{d-1}\btheta_j^{\perp,\ell}\otimes\btheta_j^{\perp,\ell} e^{i\K_S(\bz_S-\bz_D)\cdot\btheta_j}
\right\}.
\end{align*}
Since $\{\btheta_j,\btheta_j^{\perp,\ell}\}$ forms an orthonormal basis of $\RR^d$ for each $j=1,\cdots,n$ and  $n\in\NN$ is sufficiently large, therefore identity \eqref{identity6} can be invoked to get the desired result. This completes the proof.
\end{proof}

As discussed in the single-short case, it can be observed from Lemma \ref{LemTD-density} that the multi-short indicator $\bz_S\mapsto\mathcal{I}^n_{\rm TD}(\bz_S;\alpha)$ is proportional to $\|\Im\left\{\bGam^\omega_\alpha(\bz_S,\bz_D)\right\}\|^2$ (instead of $\|\Im\left\{\bGam^\omega_\alpha(\bz_S,\bz_D)\right\}\|$ as in the single-short case). Therefore, multi-short indicator in density contrast case attains most prominent increase as $\bz_S$ approaches $\bz_D$ and decays rapidly as $\bz_S$ moves away from $\bz_D$. Moreover, the resolution of the multi-short indicator is better than that of the single-short indicator.

\subsubsection{Elasticity contrast}

If  $D$ has only an elasticity contrast then the trial inclusion is also nucleated with an elastic contrast only.  Therefore, location indicator $\mathcal{I}_{\rm TD}[\bu^{\rm inc}]$ and back-propagator $\mathcal{H}[\bu^{\infty}_{\epsilon,\alpha}]$ for an elasticity contrast admit expressions
\begin{align}
\label{ITD-EC}
&\mathcal{I}_{\rm TD}[\bu^{\rm inc}](\bz_S) = \sum_{\alpha=P,S} \Re\left\{a^\alpha_d\overline{\nabla\mathcal{H}[\bu^\infty_{\epsilon,\alpha}](\bz_S)}:\MM(B_S):\nabla\bu^{\rm inc}(\bz_S)\right\},
\\
\label{Herglotz-EC}
&\mathcal{H}[\bu^\infty_{\epsilon,\alpha}](\bz_S)\approx - 4\left(\frac{\pi}{\K_\alpha}\right)^{d-2}\rho_0 c_\alpha^2 a^\alpha_d \epsilon^d \nabla_{\bz_S}\Im\left\{\bGam^\omega_\alpha(\bz_S,\bz_D)\right\}:\MM(B_D):\nabla\bu^{\rm inc}(\bz_D). 
\end{align}
Unlike the density contrast case, the back-propagators have both solenoidal and irrotational components. This is because of the mode-conversion phenomenon due to a discrepancy between the Lam\'e parameters of the inclusion and the background medium.  The asymptotic expression of $\mathcal{I}_{\rm TD}[\bu^{\rm inc}]$ is obtained by substituting Eq. \eqref{Herglotz-EC} in Eq. \eqref{ITD-EC} as
\begin{align}
\mathcal{I}_{\rm TD}[\bu^{\rm inc}](\bz_S) 
\approx &
-\frac{\gamma_d}{\omega}\epsilon^d \Re\Bigg\{\nabla_{\bz_S}\left(\sum_{\alpha=P,S}\frac{1}{c_\alpha}\nabla_{\bz_S}\Im\left\{\bGam^\omega_\alpha(\bz_S,\bz_D)\right\}:\MM(B_D):\overline{\nabla\bu^{\rm inc}(\bz_D)}\right)
 :\MM(B_S):\nabla\bu^{\rm inc}(\bz_S)\Bigg\}.
\label{ITD-elasticity}
\end{align}

Apparently, expression \eqref{ITD-elasticity} of the indicator $\mathcal{I}_{\rm TD}[\bu^{\rm inc}]$ in the elasticity contrast case is similar to that in the density contrast case, specifically,  $\mathcal{I}_{\rm TD}[\bu^{\rm inc}](\bz_S)$ is proportional to some function of $\Im\left\{\bGam^\omega_\alpha(\bz_S,\bz_D)\right\}$. However, there are significant differences between expressions \eqref{ITD-density} and \eqref{ITD-elasticity}. First of all, the location indicator in elasticity contrast case is a weighted sum of the P- and S-components of the Kupradze matrix wherein the weights depend on the corresponding wave speeds and the EMT. It is not clear if the weighted sum achieves the maximum increase when $\bz_S\to\bz_D$. Secondly, there is a mode-conversion occurring through the EMT due to the contrast between the Lam\'e parameters (see, e.g., \cite{Mode}). Therefore, it is not guaranteed that the topological sensitivity based location indicator renders the true location of the inclusion. Even if it does so, the mode-coupling between P- and S-components will degenerate the resolution of the localization.   These observations will be backed by the analysis of the multi-short location indicator. Towards this end, the following result holds. 

\begin{lem}\label{LemTD-elasticity}
If  $n\in\NN$  is sufficiently large then, for any search point $\bz_S\in\RR^d$, 
\begin{align}
\mathcal{I}_{\rm TD}^n(\bz_S; \alpha) \approx \frac{4\gamma_d\rho_0c_\alpha^2\epsilon^d}{\omega} \left( \frac{\pi}{\K_\alpha}\right)^{d-2}\sum_{\beta=P,S}\frac{1}{c_\beta}\Psi_{\alpha,\beta}(\bz_S,\bz_D),\label{intermid14}
\end{align}
where $\Psi_{\alpha,\beta}:\RR^d\times\RR^d\to\RR$, for $\alpha,\beta=P,S$, is defined as 
\begin{align}
\Psi_{\alpha,\beta}(\bz,\bz'):=\left(\MM(B_D)\bullet\Im\left\{\nabla^2_\bz\bGam^\omega_\alpha(\bz,\bz')\right\}\right):\left(\MM(B_S)\bullet\Im\left\{\nabla^2_\bz\bGam^\omega_\beta(\bz,\bz')\right\}\right)^\top.
\label{S}
\end{align}
\end{lem}
\begin{proof}
Let $\alpha=P$ and $\bTheta^j:=\btheta_j\otimes \btheta_j$. Then, 
\begin{align}
\mathcal{I}_{\rm TD}^n (\bz_S;P) 
\approx
&
-\frac{\gamma_d}{\omega}\epsilon^d \frac{1}{n}\sum_{j=1}^n\Re\Bigg\{\nabla_{\bz_S}\left(\sum_{\alpha=P,S}\frac{1}{c_\alpha}\nabla_{\bz_S}\Im\left\{\bGam^\omega_\alpha(\bz_S,\bz_D)\right\}:\MM(B_D):\overline{\nabla\bu^{P}_j(\bz_D)}\right)
 :\MM(B_S):\nabla\bu^{P}_j(\bz_S)\Bigg\}
\nonumber
\\
=&
-\frac{\gamma_d\omega\epsilon^d }{c_P^2}\frac{1}{n}\sum_{j=1}^n\Re\Bigg\{\nabla_{\bz_S}\left(\sum_{\alpha=P,S}\frac{1}{c_\alpha}\nabla_{\bz_S}\Im\left\{\bGam^\omega_\alpha(\bz_S,\bz_D)\right\}:\MM(B_D):\bTheta^j\right):\MM(B_S):\bTheta^j e^{i\K_P(\bz_S-\bz_D)\cdot\btheta_j}\Bigg\}.
\label{eq:33b}
\end{align}
This can be equivalently written, using definition of the contraction operators and invoking approximation \eqref{identity7}, as
\begin{align*}
\mathcal{I}_{\rm TD}^n&(\bz_S;P)
\\
\approx & -\frac{\gamma_d\omega\epsilon^d}{c_P^2}\Re
\Bigg\{
\sum_{p,q,l,m,p',q',l',m'=1}^d\sum_{\alpha=P,S}\frac{1}{c_\alpha}\Im\left\{\left(\partial_{lp'}\bGam_\alpha^\omega(\bz_S,\bz_D)\right)_{mq'}\right\}
 m_{p'q'}^{l'm'}(B_D) m_{pq}^{lm}(B_S)
\left[\frac{1}{n}\sum_{j=1}^n\bTheta^j_{l'm'}\bTheta^j_{pq} e^{i\K_P(\bz_S-\bz_D)\cdot\btheta_j}\right]
\Bigg\}
\\
\approx &\frac{4\rho_0\gamma_dc_P^2\epsilon^d}{\omega}\left(\frac{\pi}{\K_P}\right)^{d-2}
\sum_{p,q,l,m,p',q',l',m'=1}^d\sum_{\alpha=P,S}\frac{1}{c_\alpha}\Im\left\{\left(\partial_{lp'}{\bGam}_\alpha^\omega(\bz_S,\bz_D)\right)_{mq'}\right\}
m_{p'q'}^{l'm'}(B_D)
\Im\left\{\left(\partial_{l'p}{\bGam}^\omega_P(\bz_S,\bz_D)\right)_{m'q}\right\} m_{pq}^{lm}(B_S) 
\\
= & \frac{4\rho_0\gamma_dc_P^2\epsilon^d}{\omega}\left(\frac{\pi}{\K_P}\right)^{d-2}
\sum_{l,m, l',m'=1}^d\sum_{\alpha=P,S}\frac{1}{c_\alpha}\left(\sum_{p',q'=1}^d\Im\left\{\left(\partial_{lp'}{\bGam}_\alpha^\omega(\bz_S,\bz_D)\right)_{mq'}\right\}
m_{p'q'}^{l'm'}(B_D)\right)
\\
&
\times 
\left(\sum_{p,q=1}^d\Im\left\{\left(\partial_{l'p}{\bGam}^\omega_P(\bz_S,\bz_D)\right)_{m'q}\right\} m_{pq}^{lm}(B_S)\right)
\\
= & \frac{4\rho_0\gamma_dc_P^2\epsilon^d}{\omega}\left(\frac{\pi}{\K_P}\right)^{d-2}\sum_{\alpha=P,S}\frac{1}{c_\alpha}\left(\MM(B_D)\bullet\Im\left\{\nabla_{\bz_S}^2\bGam_\alpha^\omega(\bz_S,\bz_D)\right\}\right) :\left(\MM(B_S)\bullet\Im\left\{\nabla_{\bz_S}^2{\bGam}^\omega_P(\bz_S,\bz_D)\right\}\right)^\top.
\end{align*}
The conclusion follows from the definition of $\Psi_{\alpha,\beta}$. Similarly, for $\alpha=S$, if $\widetilde{\bTheta}^{j,\ell}:=\btheta_{j}\otimes \btheta_j^{\perp,\ell}$ then
\begin{align*}
\mathcal{I}_{\rm TD}^n (\bz_S;S) 
\approx&
-\frac{\gamma_d\omega\epsilon^d }{c_S^2}\frac{1}{n}\sum_{j=1}^n\sum_{\ell=1}^{d-1}\Re\Bigg\{
\nabla_{\bz_S}\left(\nabla_{\bz_S} \Im\left\{\sum_{\alpha=P,S}\frac{1}{c_\alpha}\bGam_\alpha^\omega(\bz_S,\bz_D)\right\}:\MM(B_D):\widetilde{\bTheta}^{j,\ell}\right)
:\MM(B_S):\widetilde{\bTheta}^{j,\ell} e^{i\K_S(\bz_S-\bz_D)\cdot\btheta_j}
\Bigg\},
\end{align*}
or equivalently, 
\begin{align*}
\mathcal{I}_{\rm TD}^n(\bz_S;S)\approx & -\frac{\gamma_d\omega\epsilon^d}{c_S^2}\Re
\Bigg\{
\sum_{p,q,l,m,p',q',l',m'=1}^d\sum_{\alpha=P,S}\frac{1}{c_\alpha}\Im\left\{\left(\partial_{lp'}\bGam_\alpha^\omega(\bz_S,\bz_D)\right)_{mq'}\right\}
\\
&
\times
m_{p'q'}^{l'm'}(B_D) m_{pq}^{lm}(B_S)
\left[\frac{1}{n}\sum_{j=1}^n\sum_{\ell=1}^{d-1}\widetilde{\bTheta}^{j,\ell}_{l'm'}\widetilde{\bTheta}^{j,\ell}_{pq} e^{i\K_S(\bz_S-\bz_D)\cdot\btheta_j}\right]
\Bigg\}.
\end{align*}
Therefore, thanks to approximation \eqref{identity8},  
\begin{align*}
\mathcal{I}_{\rm TD}^n(\bz_S;S)\approx 
& \frac{4\rho_0\gamma_dc_S^2\epsilon^d}{\omega}\left(\frac{\pi}{\K_S}\right)^{d-2}
\sum_{p,q,l,m,p',q',l',m'=1}^d\sum_{\alpha=P,S}\frac{1}{c_\alpha}\Im\left\{\left(\partial_{lp'}\bGam_\alpha^\omega(\bz_S,\bz_D)\right)_{mq'}\right\}
\\
&
\times 
m_{p'q'}^{l'm'}(B_D)
\Im\left\{\left(\partial_{l'p}{\bGam}^\omega_S(\bz_S,\bz_D)\right)_{m'q}\right\} m_{pq}^{lm}(B_S)
\\
=& \frac{4\rho_0\gamma_dc_S^2\epsilon^d}{\omega}\left(\frac{\pi}{\K_S}\right)^{d-2}
\sum_{l,m, l',m'=1}^d\sum_{\alpha=P,S}\frac{1}{c_\alpha}\left(\sum_{p',q'=1}^d\Im\left\{\left(\partial_{lp'}\bGam_\alpha^\omega(\bz_S,\bz_D)\right)_{mq'}\right\}
m_{p'q'}^{l'm'}(B_D)\right)
\\
&
\times
\left(\sum_{p,q=1}^d\Im\left\{\left(\partial_{l'p}{\bGam}^\omega_S(\bz_S,\bz_D)\right)_{m'q}\right\} m_{pq}^{lm}(B_S)\right)
\\
=& \frac{4\rho_0\gamma_dc_S^2\epsilon^d}{\omega}\left(\frac{\pi}{\K_S}\right)^{d-2}\sum_{\alpha=P,S}\frac{1}{c_\alpha}\left(\MM(B_D)\bullet\Im\left\{\nabla_{\bz_S}^2\bGam_\alpha^\omega(\bz_S,\bz_D)\right\}\right)
:\left(\MM(B_S)\bullet\Im\left\{\nabla_{\bz_S}^2{\bGam}^\omega_S(\bz_S,\bz_D)\right\}\right)^\top,
\end{align*}
which leads to the desired result by definition of $\Psi_{\alpha,\beta}$.
\end{proof}

 It is clear from Eq. \eqref{intermid14} that $\bz_S\mapsto\mathcal{I}_{\rm TD}^n(\bz_S;\alpha)$ may not attain its maximum increase near $\bz_D$. In fact, the result obtained in Lemma  \ref{LemTD-elasticity} indicates that  $\mathcal{I}_{\rm TD}^n$ is actually proportional to the weighted sum $\sum_{\beta}\Psi_{\alpha,\beta}/c_\beta$ in elasticity contrast case. The term $\Psi_{\alpha,\alpha}$ achieves a sharp peak when $\bz_S\to\bz_D$, however, the coupling term $\Psi_{P,S}$ degenerates the localization capability of the indicator $\bz_S\mapsto\mathcal{I}_{\rm TD}^n(\bz_S;\alpha)$. On  the other hand, even if the indicator somehow shows the maximum increase near $\bz_D$, e.g., when the coupling term is relatively small, only a sub-optimal resolution can  be achieved. Therefore, in order to ensure a guaranteed detection of the inclusion in the elasticity contrast case with optimal resolution as per Rayleigh criterion, one must get rid of the coupling term. Towards this end,  (properly) weighted location indicators are introduced in the next section based on the Helmholtz decomposition of  the back-propagators. 

\subsection{Weighted location indicators}

Let us define the weighted single-short and multi-short location indicators, respectively, by
\begin{align}
\label{ITDW}
\mathcal{I}_{\rm W}[\bu^{\rm inc}](\bz_S):=& \sum_{\alpha=P,S} c_\alpha
\Re\Bigg\{
a^\alpha_d\overline{\nabla\Pk_\alpha\left[\mathcal{H}[\bu^\infty_{\epsilon,\alpha}]\right](\bz_S)}:\MM(B_S):\nabla\Pk_\alpha\left[\bu^{\rm inc}\right](\bz_S)
\nonumber
\\
&-a^\alpha_d\omega^2(\rho_0-\rho_2)|B_S|\overline{\Pk_\alpha\left[\mathcal{H}[\bu^\infty_{\epsilon,\alpha}]\right](\bz_S)}\cdot\Pk_\alpha[\bu^{\rm inc}](\bz_S) 
\Bigg\},
\\
\mathcal{I}_{\rm W}(\bz_S; P):=& \frac{1}{n}\sum_{j=1}^n\mathcal{I}_{\rm W}[\bu^P_j](\bz_S)
\quad\text{and}\quad
\mathcal{I}_{\rm W}(\bz_S; S):= \frac{1}{n}\sum_{j=1}^n\sum_{\ell=1}^{d-1}\mathcal{I}_{\rm W}[\bu^{S}_{j,\ell}](\bz_S).
\label{IWn}
\end{align}
Here, the P- and S-components of the back-propagator are correlated with the corresponding components of the incident fields  so that the coupling terms can be avoided. 
 In fact, the function $\mathcal{I}_{\rm W}[\bu^{\rm inc}]$ is roughly the same as $\mathcal{I}_{\rm TD}[\bu^{\rm inc}]$ in the density contrast case except for the constant weighting with $c_\alpha$. It is  therefore  linked to the topological derivative of a cost functional thanks to the non-conversion of waveforms in the presence of the inclusion. However, for the elasticity contrast case, $\mathcal{I}_{\rm W}[\bu^{\rm inc}]$   cannot be the topological derivative of some discrepancy function. Instead,  it is from the class of Kirchhoff-type location indicators, yet it is topologically \emph{sensitive} with respect to the nucleation of a trial inclusion at any search location. 

The aim of this section is to establish that $\mathcal{I}_{\rm W}[\bu^{\rm inc}]$ provides guaranteed localization of the inclusion in both density and elasticity contrast cases. As already mentioned, the weighted location indicator is the same as $\mathcal{I}_{\rm TD}[\bu^{\rm inc}]$ up to a constant weighting with $c_\alpha$ for density contrast case.   In fact,  
\begin{align}
\label{IW-DC}
\mathcal{I}_{\rm W}[\bu^{\rm inc}](\bz_S)= -\sum_{\alpha=P,S} \Re\left\{c_\alpha a_{\alpha,d}\omega^2(\rho_0-\rho_2)|B_S|\overline{\Pk_\alpha\left[\mathcal{H}[\bu^{\infty}_{\epsilon,\alpha}]\right](\bz_S)}\cdot\Pk_\alpha\left[\bu^{\rm inc}\right](\bz_S)\right\}.
\end{align}
Since $\mathcal{H}[\bu^{\infty}_{\epsilon, P}](\bz_S)$ and $\mathcal{H}[\bu^{\infty}_{\epsilon, S}](\bz_S)$ are respectively solenoidal and irrotational when $\epsilon^d\to 0$, Lemma \ref{LemHexp} renders
\begin{align*}
\mathcal{I}_{\rm W}[\bu^{\rm inc}](\bz_S)\approx \omega^3\gamma_d\tau\epsilon^d \Re\left\{\sum_{\alpha=P,S}\Im\left\{\bGam^\omega_\alpha(\bz_S,\bz_D)\right\} \overline{\bu^{\rm inc}(\bz_D)}\cdot\Pk_\alpha[\bu^{\rm inc}](\bz_S)
\right\}.
\end{align*}
On the other hand, in elasticity contrast case,
\begin{align}
\label{IW-EC}
\mathcal{I}_{\rm W}[\bu^{\rm inc}](\bz_S) = \sum_{\alpha=P,S} c_\alpha
\Re\left\{a^\alpha_d\overline{\nabla\Pk_\alpha\left[\mathcal{H}[\bu^\infty_{\epsilon,\alpha}]\right](\bz_S)}:\MM(B_S):\nabla\Pk_\alpha\left[\bu^{\rm inc}\right](\bz_S)\right\}.
\end{align}
On substituting Eq. \eqref{Herglotz-EC} in Eq. \eqref{IW-EC} and  after fairly easy manipulations, one arrives at
\begin{align*}
\mathcal{I}_{\rm W}[\bu^{\rm inc}](\bz_S) 
\approx&
-\frac{\gamma_d}{\omega}\epsilon^d \Re\Bigg\{\sum_{\alpha=P,S} \nabla_{\bz_S}\left( \nabla_{\bz_S}\Im\left\{\bGam^\omega_\alpha(\bz_S,\bz_D)\right\}:\MM(B_D):\overline{\nabla\bu^{\rm inc}(\bz_D)}\right)
 :\MM(B_S):\Pk_\alpha\left[\nabla\bu^{\rm inc}\right](\bz_S)\Bigg\}.
\end{align*}

Lemma \ref{LemW} highlights the topological sensitivity features of the weighted location indicators and can be easily proved using similar arguments as in Lemma \ref{LemTD-elasticity}.  
\begin{lem}\label{LemW}
If $n\in\NN$  is sufficiently large then, for any search point $\bz_S\in\RR^d$, 
\begin{align}
&\mathcal{I}_{\rm W}^n(\bz_S; \alpha) \approx 4\gamma_d\rho_0c_\alpha^2 \omega^3\tau\epsilon^d\left( \frac{\pi}{\K_\alpha}\right)^{d-2}\Phi_{\alpha,\alpha}(\bz_S,\bz_D), &\text{(density contrast)},
\label{Eq1:LemW}
\\
&\mathcal{I}_{\rm W}^n(\bz_S; \alpha) \approx \frac{4\gamma_d\rho_0c_\alpha^2\epsilon^d}{\omega} \left( \frac{\pi}{\K_\alpha}\right)^{d-2}\Psi_{\alpha,\alpha}(\bz_S,\bz_D), &\text{(elasticity contrast)},
\label{Eq2:LemW}
\end{align}
where $\Phi_{\alpha,\beta}$ and $\Psi_{\alpha,\beta}$ are defined in \eqref{T} and \eqref{S} respectively.
\end{lem}
The result in Lemma \ref{LemW} clearly shows that the coupling term $\Psi_{P,S}$ has disappeared and the weighted location indicator $\mathcal{I}_{\rm W}^n$ is proportional to $\Phi_{\alpha,\alpha}$ and $\Psi_{\alpha,\alpha}$ in density and elasticity contrast cases, respectively. It can be established that these functions attain their maximum increase   near the true location of the inclusion thanks to the imaginary parts of the fundamental solutions $\bGam^\omega_\alpha$. For example, the following expressions for $\Psi_{\alpha,\alpha}$ were derived by Ammari et al  \cite[Prop. 4.3, 4.4]{TDelastic} for circular and spherical  inclusions when EMT has the form \eqref{Mball}.
\begin{lem}
Let $D$ be a ball in $\RR^d$, for $d=2,3$. Then, 
\begin{align*}
\Psi_{P,P}(\bz,\bz_D) = & 
a^2\Big| \nabla^2\left[\Im\{\bGam_P^\omega(\bz-\bz_D)\}\right]\Big|^2
+
2ab\Big| \Delta\left[\Im\{\bGam_P^\omega(\bz-\bz_D)\}\right]\Big|^2
+ 
b^2\Big| \Delta {\rm Tr}\left[\Im\{\bGam_P^\omega(\bz-\bz_D)\}\right]\Big|^2,
\\
\Psi_{S,S}(\bz,\bz_D) = & 
\frac{a^2}{\mu_0^2}\Bigg[\frac{1}{\K_S^2}\sum_{p,q,l,m, m\neq l}\Big| \partial_{pqlm}\left[\Im\{g_S^\omega(\bz-\bz_D)\}\right]\Big|^2
+
\frac{(d-2)}{4}\Big| \nabla^2\left[\Im\{g_S^\omega(\bz-\bz_D)\}\right]\Big|^2
+
\frac{\K_S^4}{4}\Big| \Im\{g_S^\omega(\bz-\bz_D)\}\Big|^2
\Bigg].
\end{align*}
\end{lem}
In Fig. \ref{Fig1}, numerical realizations are provided for the focal spot of the functions $\Phi_{\alpha,\alpha}$ and $\Psi_{\alpha,\alpha}$ which clearly indicate that $\bz_S\mapsto\mathcal{I}_{\rm W}[\bu^{\rm inc}](\bz_S)$ attains its most prominent increase in the neighborhood of the true location (in this case, the origin) and therefore has a sharp peak at $\bz_D$ with a focal spot size of the order of half the wavelength of the corresponding wave-mode. 
\begin{figure}[!htb]
\begin{center}
\subfigure[{$\Phi_{P,P}$}]{\includegraphics[width=0.45\textwidth]{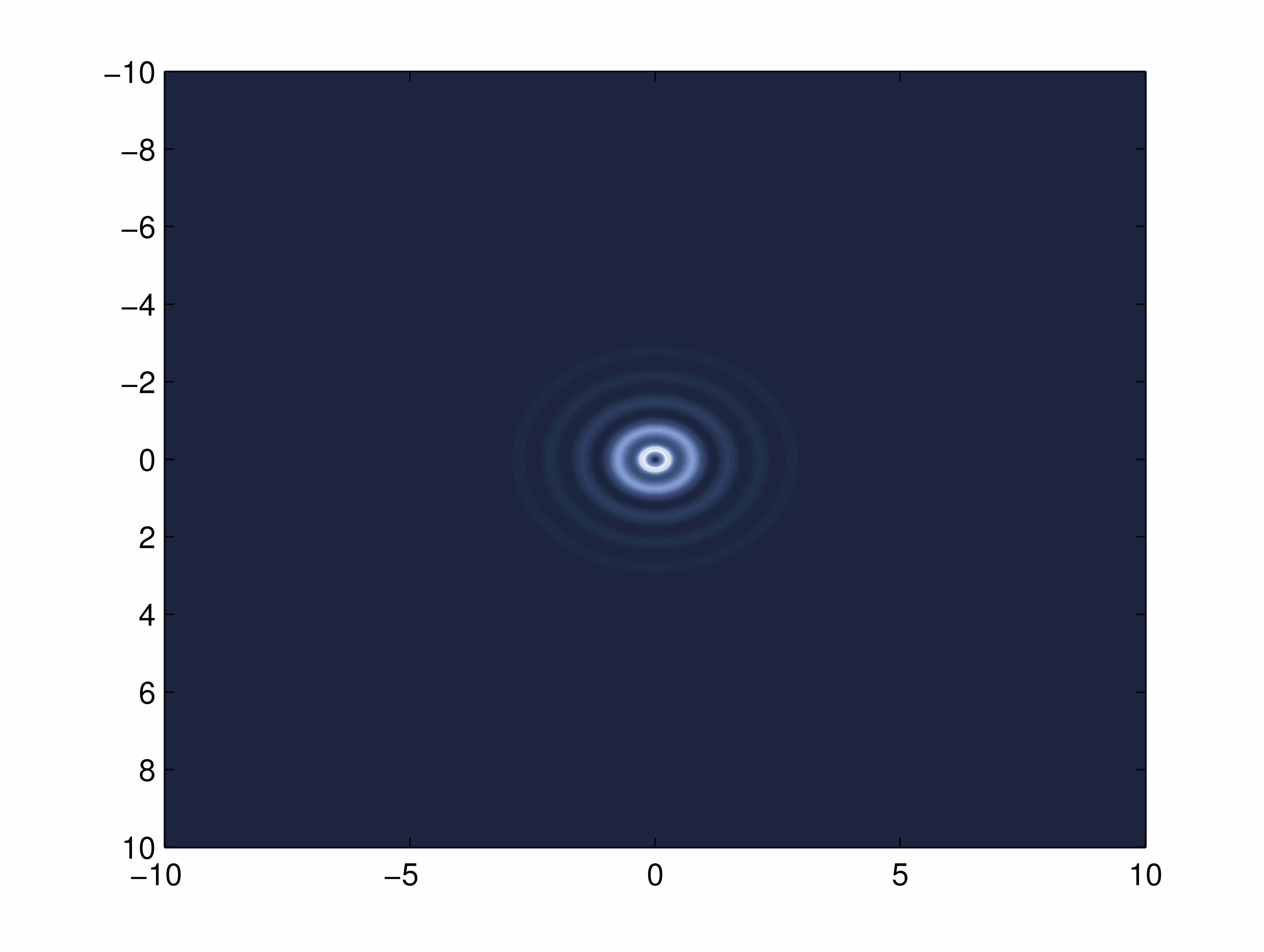}\label{Tpp} }
\hspace{0.5cm}
\subfigure[{$\Phi_{S,S}$}]{\includegraphics[width=0.45\textwidth]{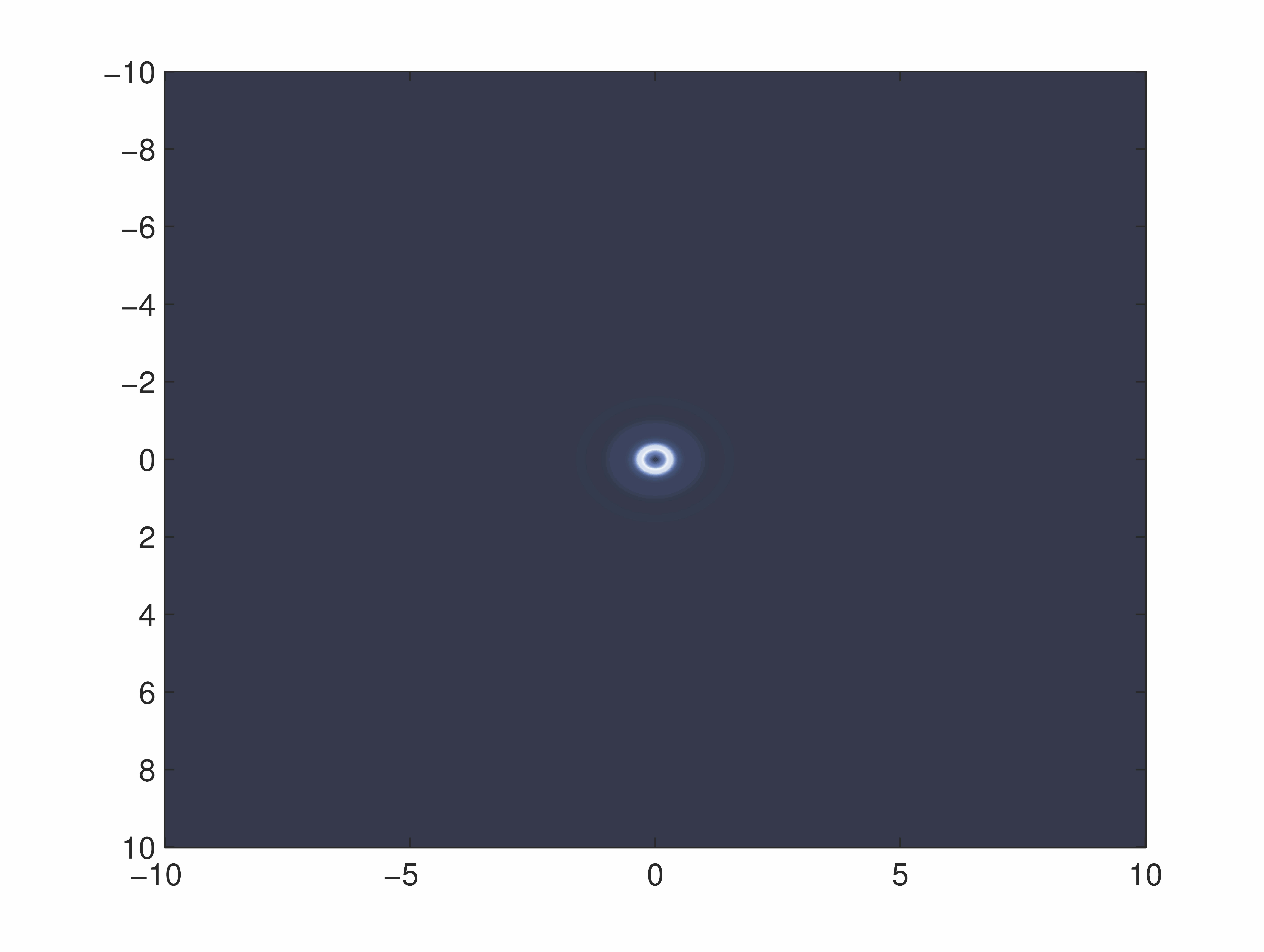}\label{Tss}}
\\
\subfigure[{$\Psi_{P,P}$}]{\includegraphics[width=0.45\textwidth]{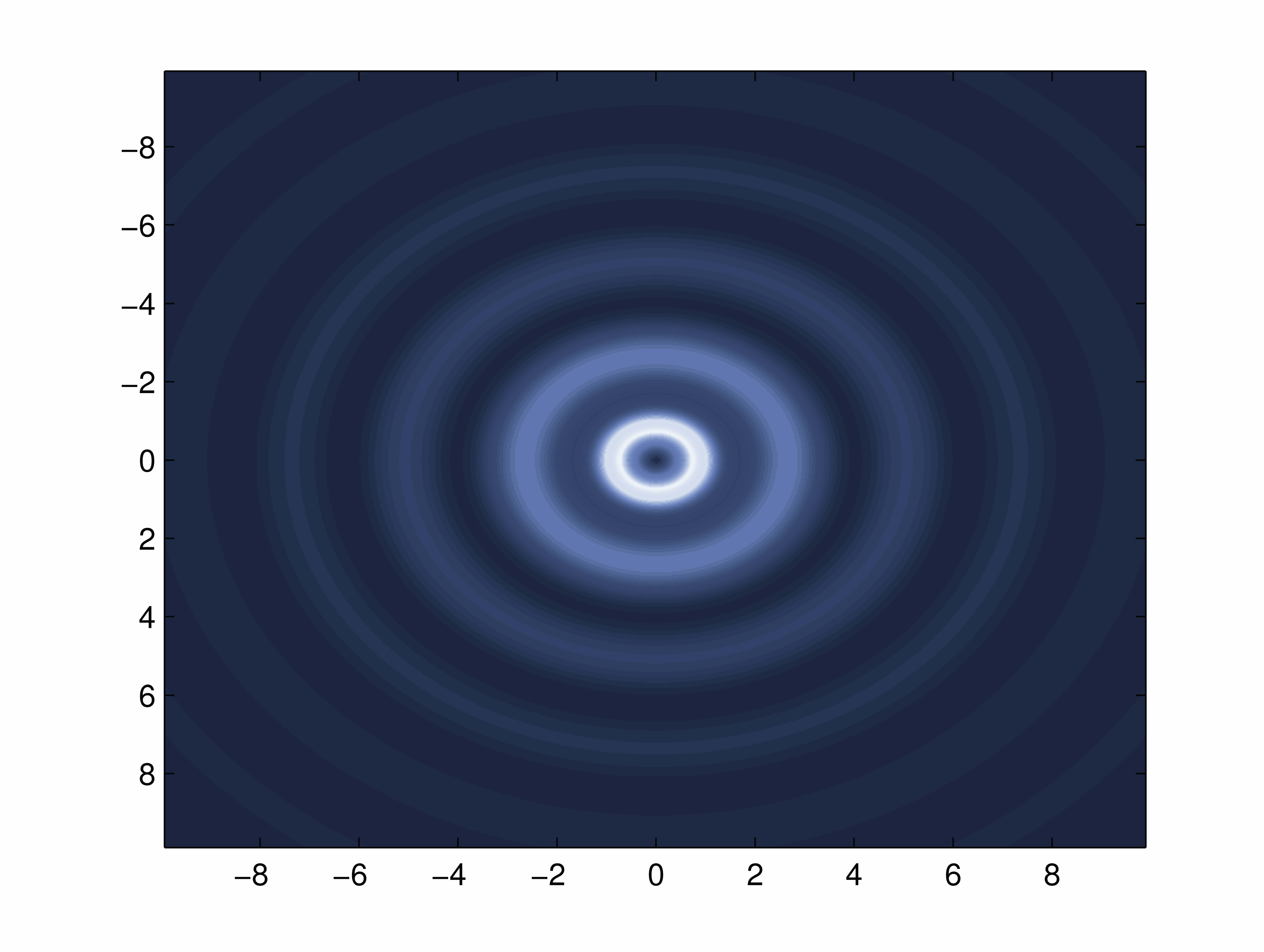}\label{Spp} }
\hspace{0.5cm}
\subfigure[{$\Psi_{S,S}$}]{\includegraphics[width=0.45\textwidth]{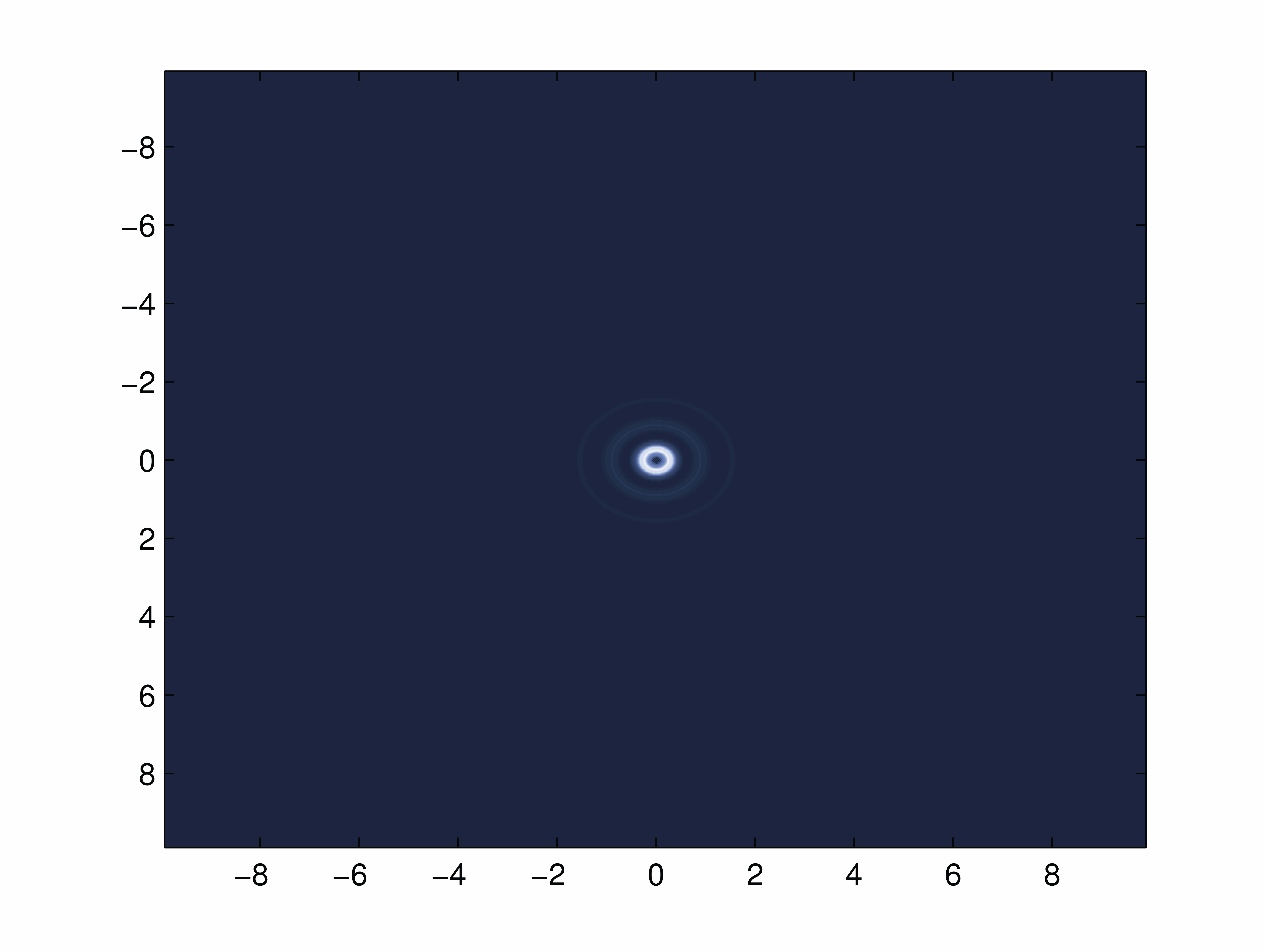}\label{Sss}}
\caption{Focal spots of the functions $\Phi_{\alpha,\alpha}$ and $\Psi_{\alpha,\alpha}$ for $\alpha=P$ or $S$, and  $\lambda_0=10kPa$, $\mu_0=1kPa$, and $\rho_0=1000kg/m^3$.}
\label{Fig1}
\end{center}
\end{figure}

In rest of this article, the statistical stability of the multi-short weighted location indicator is discussed with respect to the measurement and medium noises.

\section{Statistical stability with measurement noise}\label{s:measurement}

Let us analyze the stability of  weighted location indicator $\mathcal{I}^n_{\rm W}$ with respect to an additive measurement noise. It is assumed that the measurements of the far-field amplitudes are corrupted by a circular zero mean Gaussian noise $\bxi:\mathbb{S}^{d-1}\to\CC^d$, i.e., 
\begin{align}
\label{NoiseModel}
\bu^\infty_{\epsilon}(\hbx):=\widetilde{\bu^\infty_\epsilon}(\hbx)+\bxi(\hbx), \qquad \hbx\in\mathbb{S}^{d-1}, 
\end{align}
where $\bu^\infty_\epsilon$ represents the corrupted total far-field pattern and $ \widetilde{\bu^\infty_\epsilon}(\hbx)$ indicates the true data without noise corruption. Let $\sigma^2_\bxi$ be the noise covariance of $\bxi$ such that for all $\hbz,\hbz'\in\mathbb{S}^{d-1}$, 
\begin{align}
\label{NoiseExp}
\mathbb{E}\left[\bxi(\hbz)\otimes\overline{\bxi(\hbz')}\right] = \sigma_\bxi ^2\delta_{\hbz}(\hbz')\I_d
\quad\text{and}\quad
\mathbb{E}\left[\bxi^k(\hbz)\otimes\overline{\bxi^{k'}(\hbz)}\right] = \sigma_\bxi^2 \delta_{kk'}\delta_{\hbz}(\hbz')\I_d,
\end{align}
where $k$ and $k'$, for $k,k'\in\{1,\cdots,n\}$, indicate the $k-$th and $k'-$th measurements and $\mathbb{E}$ denotes the expectation with respect to the statistics of noise $\bxi$. By the expectations \eqref{NoiseExp}, the following  assumptions are made. 
\begin{enumerate}
\item The noise  at two different points $\hbz\in\mathbb{S}^{d-1}$ and $\hbz'\in\mathbb{S}^{d-1}$ is uncorrelated if $\hbz\neq\hbz'$. 
\item The noise $\bxi(\hbz)$, for all $\hbz\in\mathbb{S}^{d-1}$, has uncorrelated real and imaginary parts. 
\item The individual components of the noise are uncorrelated. 
\item The noises for different measurements, i.e., $\bxi^k$ and $\bxi^{k'}$ for $k-$th and $k'-$th measurements are uncorrelated whenever $k\neq k'$. However, all measurements have identical noise covariance $\sigma_\bxi^2$. 
\end{enumerate}
Henceforth, the  P-part and S-part of $\bxi$  are denoted by $\bxi_P$ and $\bxi_S$, respectively.

It is evident that the noise affects the weighted location indicator  $\mathcal{I}_{\rm W}^n$ through the back-propagator $\mathcal{H}[\bu^\infty_{\epsilon,\alpha}]$. In fact, by linearity of the Herglotz operator,  
\begin{align*}
\mathcal{H}\left[\bu^\infty_{\epsilon,\alpha}\right](\hbz)= \mathcal{H}\left[\widetilde{\bu^\infty_{\epsilon,\alpha}}\right](\hbz)+\mathcal{H}\left[\bxi_\alpha\right](\hbz),
\end{align*}
wherein the first term on the RHS is independent of the noise and yields the true maximum of location indicator $\mathcal{I}_{\rm W}$ as discussed in the noise-free case. On the other hand, the second term on the RHS, $\mathcal{H}\left[\bxi_\alpha\right]$, is a circular Gaussian random process with mean-zero by linearity and by the definition of $\bxi$. In order to fathom the role of measurement noise on the localization capabilities of $\mathcal{I}_{\rm W}^n$ and for ensuing analysis, the covariance of  $\mathcal{H}\left[\bxi_\alpha\right](\hbz)$ is required. Towards this end, the following result holds. 
\begin{lem}\label{LemHcov}
For all points $\bz,\bz'\in\RR^d$,  
\begin{align*}
\mathbb{E}\left[\mathcal{H}\left[\bxi_\alpha\right](\hbz)\otimes \overline{\mathcal{H}\left[\bxi_\alpha\right](\hbz')}\right]
=4\rho_0\sigma_\bxi^2c_\alpha^2\left(\frac{\pi}{\K_\alpha}\right)^{d-2}\Im\left\{\bGam^\omega_{\alpha}(\bz,\bz')\right\}.
\end{align*}
\end{lem}
\begin{proof}
By definition of the Herglotz fields, 
\begin{align}
\mathbb{E}\left[\mathcal{H}\left[\bxi_\alpha\right](\hbz)\otimes \overline{\mathcal{H}\left[\bxi_\alpha\right](\hbz')}\right]=\mathbb{E}\left[\iint_{\mathbb{S}^{d-1}\times\mathbb{S}^{d-1}}\bxi_\alpha(\hbx)\otimes\overline{\bxi_\alpha(\hby)} e^{i\K_\alpha\hbx\cdot\bz}e^{-i\K_\alpha\hby\cdot{\bz'}}d\sigma(\hbx) d\sigma(\hby)\right].
\label{intermid3}
\end{align}
Recall that 
\begin{align}
\left(\I_d-\hbx\otimes\hbx\right)\bxi(\hbx)=\bxi_S(\hbx)\quad\text{and}\quad(\hbx\otimes\hbx)\bxi(\hbx)=\bxi_P(\hbx).
\label{identity9}
\end{align}
Then, for $\alpha=P$, Eqs. \eqref{intermid3} and \eqref{identity9} furnish
\begin{align*}
\mathbb{E}\left[\mathcal{H}\left[\bxi_P\right](\hbz)\otimes \overline{\mathcal{H}\left[\bxi_P\right](\hbz')}\right]
=\iint_{\mathbb{S}^{d-1}\times\mathbb{S}^{d-1}}(\hbx\otimes\hbx) \mathbb{E}\left[\bxi(\hbx)\otimes\overline{\bxi(\hby)}\right](\hby\otimes\hby)
 e^{i\K_P\hbx\cdot\bz}e^{-i\K_P\hby\cdot{\bz'}}d\sigma(\hbx) d\sigma(\hby).
\end{align*}
By virtue of \eqref{NoiseModel}, it can be seen that
\begin{align*}
\mathbb{E}\left[\mathcal{H}\left[\bxi_P\right](\hbz)\otimes \overline{\mathcal{H}\left[\bxi_P\right](\hbz')}\right]
=&\sigma^2_\bxi\int_{\mathbb{S}^{d-1}}(\hbx\otimes\hbx) (\hbx\otimes\hbx)
 e^{i\K_P\hbx\cdot(\bz-{\bz'})}d\sigma(\hbx)
\\
=&\sigma^2_\bxi\int_{\mathbb{S}^{d-1}}(\hbx\otimes\hbx)  
 e^{i\K_P\hbx\cdot(\bz-{\bz'})}d\sigma(\hbx)
 \\
=&-\frac{\sigma^2_\bxi}{\K_P^2}\nabla_\bz\otimes\nabla_\bz\left(\int_{\mathbb{S}^{d-1}}
 e^{i\K_P\hbx\cdot(\bz-{\bz'})}d\sigma(\hbx)\right).
\end{align*}
Therefore, by identity \eqref{identity3} and the definition of $\bGam_P^\omega$, 
\begin{align*}
\mathbb{E}\left[\mathcal{H}\left[\bxi_P\right](\hbz)\otimes \overline{\mathcal{H}\left[\bxi_P\right](\hbz')}\right]
=&
-\frac{4\sigma^2_\bxi}{\K_P^2}\left(\frac{\pi}{\K_P}\right)^{d-2}\nabla_\bz\otimes\nabla_\bz\Im\left\{g^\omega_P(\bz-\bz')\right\}
=
4\rho_0c_P^2\sigma^2_\bxi\left(\frac{\pi}{\K_P}\right)^{d-2}\Im\left\{\bGam^\omega_P(\bz,\bz')\right\},
\end{align*}
which is the required expression for  $\alpha=P$.  Similarly, 
\begin{align*}
\mathbb{E}&\left[\mathcal{H}\left[\bxi_S\right](\hbz)\otimes \overline{\mathcal{H}\left[\bxi_S\right](\hbz')}\right]
=\iint_{\mathbb{S}^{d-1}\times\mathbb{S}^{d-1}}(\I_d-\hbx\otimes\hbx) \mathbb{E}\left[\bxi(\hbx)\otimes\overline{\bxi(\hby)}\right](\I_d-\hby\otimes\hby)
 e^{i\K_S\hbx\cdot\bz}e^{-i\K_S\hby\cdot{\bz'}}d\sigma(\hbx) d\sigma(\hby).
\end{align*}
Thanks to Eq. \eqref{NoiseModel}, 
\begin{align*}
\mathbb{E}\left[\mathcal{H}\left[\bxi_S\right](\hbz)\otimes \overline{\mathcal{H}\left[\bxi_S\right](\hbz')}\right]
=&\sigma^2_\bxi\int_{\mathbb{S}^{d-1}}(\I_d-\hbx\otimes\hbx) (\I_d-\hbx\otimes\hbx)
 e^{i\K_S\hbx\cdot(\bz-{\bz'})}d\sigma(\hbx)
\\
=&\sigma^2_\bxi\int_{\mathbb{S}^{d-1}}(\I_d-\hbx\otimes\hbx)  
 e^{i\K_S\hbx\cdot(\bz-{\bz'})}d\sigma(\hbx)
 \\
=&\sigma^2_\bxi\left(\I_d+\frac{1}{\K_S^2}\nabla_\bz\otimes\nabla_\bz\right)\left(\int_{\mathbb{S}^{d-1}}
 e^{i\K_S\hbx\cdot(\bz-{\bz'})}d\sigma(\hbx)\right).
\end{align*}
Finally, using  identity \eqref{identity3} again and the definition of $\bGam_S^\omega$, one arrives at
\begin{align*}
\mathbb{E}\left[\mathcal{H}\left[\bxi_S\right](\hbz)\otimes \overline{\mathcal{H}\left[\bxi_S\right](\hbz')}\right]
=&
4\sigma^2_\bxi\left(\frac{\pi}{\K_S}\right)^{d-2}\left(\I_d+\frac{1}{\K_S^2}\nabla_\bz\otimes\nabla_\bz\right)\Im\left\{g^\omega_S(\bz-\bz')\right\}
=
4\rho_0c_S^2\sigma^2_\bxi\left(\frac{\pi}{\K_S}\right)^{d-2}\Im\left\{\bGam^\omega_S(\bz,\bz')\right\}.
\end{align*}
This completes the proof. 
\end{proof}

Lemma \ref{LemHcov} shows that the field $\mathcal{H}\left[\bxi_\alpha\right](\hbz)$, for both $\alpha=P$ and $S$, precipitates a speckle pattern. More specifically, it generates a random cloud of hot spots with typical diameters of the order of wavelengths of the respective wave-modes and amplitudes of the order of $\sigma_\bxi/\sqrt{\K_\alpha}$. 

Let us now assess the stability features of the weighted location indicator, and consider the cases of density and elasticity contrasts separately. 

\subsection{Density contrast}\label{s:DensityMeasurment}

In this case, $\mathcal{I}_{\rm W}[\bu^{\rm inc}]$ and  $\mathcal{I}_{\rm W}^n$ are given by expressions \eqref{IW-DC} and \eqref{IWn}, respectively. The cases of incident plane P- and S-waves are considered separately. For P-waves, one can calculate the covariance of the weighted indicator as
\begin{align*}
{\rm Cov}\left(\mathcal{I}_{\rm W}^n(\bz;P),\mathcal{I}_{\rm W}^n(\bz';P)\right)
=&
c_P^2\omega^4\widetilde{\tau}^2
\frac{1}{n^2}\sum_{j,j'=1}^n\mathbb{E}\Bigg[\Re\left\{
a^P_d \overline{\mathcal{H}\left[\bxi^j_P\right](\bz)}\cdot\bu^P_j(\bz)\right\}
\Re\left\{
a^P_d \overline{\mathcal{H}\left[\bxi^{j'}_P\right](\bz')}\cdot\bu^P_{j'}(\bz')\right\}
\Bigg]
\\
=&
c_P^2\omega^4\widetilde{\tau}^2|a^P_d|^2
\frac{1}{2n^2}\sum_{j=1}^n\Re\left\{\bu^P_j(\bz)\cdot\mathbb{E}
\Big[\mathcal{H}\left[\bxi^j_P\right](\bz)\otimes
\overline{\mathcal{H}\left[\bxi^{j}_P\right](\bz')}\Big]\overline{\bu^P_{j}(\bz')}
\right\}, \quad\bz,\bz'\in\RR^d.
\end{align*} 
where $\widetilde{\tau}$ is defined as  $\widetilde{\tau}:=(\rho_0-\rho_2)| B_S|$. 
Using Lemma \ref{LemHcov} and  identity \eqref{identity4}, on obtains 
\begin{align}
{\rm Cov}\left(\mathcal{I}_{\rm W}^n(\bz;P),\mathcal{I}_{\rm W}^n(\bz;P)\right)
=&
\sigma^2_\bxi c_P\omega^3\widetilde{\tau}^2\gamma_d
\frac{1}{2n^2}\sum_{j=1}^n \Re\left\{\bu^P_j(\bz)\cdot\Im\left\{\bGam^\omega_P(\bz,\bz')\right\}\overline{\bu^P_{j}(\bz')}
\right\}
\nonumber
\\
=&
\sigma^2_\bxi c_P\omega^3\widetilde{\tau}^2\gamma_d
\frac{1}{2n^2}\sum_{j=1}^n \Re\left\{\Im\left\{\bGam^\omega_P(\bz,\bz')\right\}
 :\btheta_j\otimes\btheta_j e^{i\K_P\btheta_j\cdot(\bz-\bz')}
\right\}
\nonumber
\\
\approx&
\frac{2\rho_0\sigma^2_\bxi c_P^3\omega^3\widetilde{\tau}^2\gamma_d}{n}\left(\frac{\pi}{\K_P}\right)^{d-2}
 \left\|\Im\left\{\bGam^\omega_P(\bz,\bz')\right\}\right\|^2.
 \label{eq:43b}
\end{align} 
Similarly, for incident plane S-waves, the covariance of the weighted location indicator is given by
\begin{align*}
{\rm Cov}\left(\mathcal{I}_{\rm W}^n(\bz;S),\mathcal{I}_{\rm W}^n(\bz';S)\right)
=&
c_S^2\omega^4\widetilde{\tau}^2|a^S_d|^2
\frac{1}{n^2}\sum_{j,j'=1}^n\sum_{\ell,\ell'=1}^{d-1}\mathbb{E}\Bigg[\Re\left\{
\overline{\mathcal{H}\left[\bxi^{j,\ell}_S\right](\bz)}\cdot\bu^S_{\ell, j}(\bz)\right\}
\Re\left\{
\overline{\mathcal{H}\left[\bxi^{j',\ell'}_S\right](\bz')}\cdot\bu^S_{j',\ell}(\bz')\right\}
\Bigg]
\\
=&
c_S^2\omega^4\widetilde{\tau}^2|a^S_d|^2
\frac{1}{2n^2}\sum_{j=1}^n\sum_{\ell=1}^{d-1}\Re\left\{\bu^S_{j,\ell}(\bz)\cdot\mathbb{E}
\Big[\mathcal{H}\left[\bxi^{j,\ell}_S\right](\bz)\otimes
\overline{\mathcal{H}\left[\bxi^{j,\ell}_S\right](\bz')}\Big]\overline{\bu^S_{j,\ell}(\bz')}
\right\}, \bz,\bz'\in\RR^d.
\end{align*} 
Again, using Lemma \ref{LemHcov} and identity \eqref{identity6}, one can see that
\begin{align}
{\rm Cov}\left(\mathcal{I}_{\rm W}^n(\bz;S),\mathcal{I}_{\rm W}^n(\bz;S)\right)
=&
\sigma^2_\bxi c_S\omega^3\widetilde{\tau}^2\gamma_d
\frac{1}{2n^2}\sum_{j=1}^n \sum_{\ell=1}^{d-1}\Re\left\{\bu^S_{j,\ell|}(\bz)\cdot\Im\left\{\bGam^\omega_S(\bz,\bz')\right\}\overline{\bu^S_{j,\ell}(\bz')}
\right\}
\nonumber
\\
=&
\sigma^2_\bxi c_S\omega^3\widetilde{\tau}^2\gamma_d
\frac{1}{2n^2}\sum_{j=1}^n\sum_{\ell=1}^{d-1} \Re\left\{\Im\left\{\bGam^\omega_S(\bz,\bz')\right\}
 :\btheta_j^{\perp,\ell}\otimes\btheta_j^{\perp,\ell} e^{i\K_S\btheta_j\cdot(\bz-\bz')}
\right\}
\nonumber
\\
\approx&
\frac{2\rho_0\sigma^2_\bxi c_S^3\omega^3\widetilde{\tau}^2\gamma_d}{n}\left(\frac{\pi}{\K_S}\right)^{d-2}\left\|\Im\left\{\bGam^\omega_S(\bz,\bz')\right\}\right\|^2.
\label{eq:43c}
\end{align} 

The following remarks are in order. The covariance of weighted location indicator $\mathcal{I}_{\rm W}^n$ has the same form as expression \eqref{Eq1:LemW} in Lemma \ref{LemW} except for a modulation by $C\sigma_\bxi^2/n$, where $C$ is a positive constant independent of $n$ and $\sigma_\bxi$. This elucidates that the shapes of the hot spots in the speckle field are identical to the shape of $\bz_S\mapsto\mathcal{I}_{\rm W}^n(\bz_S,\alpha)$ in the noise-free case and their amplitudes are $\sqrt{C}\sigma_\bxi/\sqrt{n}$ times that  of the main peak of $\mathcal{I}_{\rm W}^n$. However, the shape or the amplitude of  $\mathcal{I}_{\rm W}^n$ are not affected by the measurement noise.  Therefore, despite a cloud of hot spots surrounding its true peak, the weighted location indicator is expected to perform effectively with noisy data contaminated by the measurement noise.  Moreover, the dependence of the covariance of the speckle field on  $1/{n}$ suggests that the stability of the weighted location indicator increases if the number of the incident fields is increased. 

Let us now provide signal-to-noise ratio (SNR) estimates for $\mathcal{I}^n_{\rm W}$ with respect to the measurement noise in the density contrast case. Towards this end,  it can be easily seen from \eqref{eq:43b}-\eqref{eq:43c} that the variance of  $\mathcal{I}^n_{\rm W}$ is given by 
\begin{align}
\label{VarDensity}
{\rm Var}\left(\mathcal{I}^n_{\rm W}(\bz;\alpha)\right)\approx \frac{2\rho_0\sigma^2_\bxi c_\alpha^3\omega^3\widetilde{\tau}\gamma_d}{n}\left(\frac{\pi}{\K_\alpha}\right)^{d-2}\left\|\Im\left\{\bGam^\omega_\alpha(\bz,\bz)\right\}\right\|^2.
\end{align}
Consequently, thanks to  Eq. \eqref{VarDensity} and Lemma \ref{LemW}, the SNR is given by  
\begin{align*}
{\rm SNR}:=\frac{\mathbb{E}\left(\mathcal{I}^n_{\rm W}(\bz;\alpha)\right)}{\sqrt{{\rm Var}\left(\mathcal{I}^n_{\rm W}(\bz;\alpha)\right)}} 
\approx\epsilon^d|B_D| |\rho_0-\rho_1|\frac{\sqrt{n}}{\sigma_\bxi}\sqrt{2^{4-d}\pi^{d-3}c_\alpha^{d-1}\omega^{5-d}}\left\|\Im\left\{\bGam^\omega_\alpha(\bz,\bz)\right\}\right\|.
\end{align*}
It is evident that the SNR does not only depend on the noise covariance and the number of illuminations but also on the volume and contrast of the inclusion.

\subsection{Elasticity contrast}

In this case, $\mathcal{I}_{\rm W}[\bu^{\rm inc}]$ and  $\mathcal{I}_{\rm W}^n$ are given by Eqs. \eqref{IW-EC} and \eqref{IWn}, respectively. Therefore, for incident P-waves, the covariance of the weighted location indicator, for all $\bz,\bz'\in\RR^d$, is given by
\begin{align}
{\rm Cov}&\left(\mathcal{I}_{\rm W}^n(\bz;P),\mathcal{I}_{\rm W}^n(\bz';P)\right)
\nonumber
\\
=&
c_P^2
\frac{1}{n^2}\sum_{j,j'=1}^n\mathbb{E}\Bigg[\Re\left\{
a^P_d \overline{\nabla\Pk_P\left[\mathcal{H}\left[\bxi^j_P\right]\right](\bz)}:\MM(B_S):\nabla\bu^P_j(\bz)\right\}
\Re\left\{
a^P_d \overline{\Pk_P\left[\mathcal{H}\left[\bxi^{j'}_P\right]\right](\bz')}:\MM(B_S):\nabla\bu^P_{j'}(\bz')\right\}
\Bigg].
\nonumber
\end{align}
Simple manipulations show that
\begin{align}
{\rm Cov}&\left(\mathcal{I}_{\rm W}^n(\bz;P),\mathcal{I}_{\rm W}^n(\bz';P)\right)
\nonumber
\\
=&
c_P^2|a^P_d|^2
\frac{1}{2n^2}\sum_{j,j'=1}^n\Re\Bigg\{
\mathbb{E}
\Bigg[
\left( \overline{\nabla\Pk_P\left[\mathcal{H}\left[\bxi^j_P\right]\right](\bz)}:\MM(B_S):\nabla\bu^P_j(\bz)\right) 
\left(\Pk_P\left[\mathcal{H}\left[\bxi^{j'}_P\right]\right](\bz'):\MM(B_S):\overline{\nabla\bu^P_{j'}(\bz')}\right)
\Bigg]
\Bigg\}
\nonumber
\\
=&
c_P^2 |a^P_d|^2
\frac{1}{2n^2}\sum_{j,j'=1}^n\Re\Bigg\{\nabla\bu^P_j(\bz):\MM(B_S)
:
\left(\mathbb{E}
\left[\nabla\Pk_P\left[\mathcal{H}\left[\bxi^j_P\right]\right](\bz)\otimes
\overline{\nabla\Pk_P\left[\mathcal{H}\left[\bxi^{j'}_P\right]\right](\bz')}\right]:\MM(B_S):\overline{\nabla\bu^P_{j'}(\bz')}\right)
\Bigg\}.
\label{intermid4}
\end{align} 
Remark that, by the same arguments as in Lemma \ref{LemHcov},  
\begin{align}
\label{NablaHcov}
\mathbb{E}
\left[\nabla\Pk_\alpha\left[\mathcal{H}\left[\bxi^j_\alpha\right]\right](\bz)\otimes
\overline{\nabla\Pk_\alpha\left[\mathcal{H}\left[\bxi^{j'}_\alpha\right]\right](\bz')}\right]
= -4\rho_0 c_\alpha^2\sigma_\bxi^2\left(\frac{\pi}{\K_\alpha}\right)^{d-2}\nabla^2_{\bz}\Im\left\{\bGam^\omega_\alpha(\bz,\bz')\right\}\delta_{j,j'}.
\end{align}
Therefore, using Eq. \eqref{NablaHcov} in  Eq. \eqref{intermid4}, one arrives at 
\begin{align}
{\rm Cov}\left(\mathcal{I}_{\rm W}^n(\bz;P),\mathcal{I}_{\rm W}^n(\bz';P)\right)
=&
-\frac{c_P\sigma_\bxi^2\gamma_d}{2\omega}
\frac{1}{n^2}\sum_{j=1}^n\Re\Bigg\{\nabla\bu^P_j(\bz):\MM(B_S)
:\left(\nabla^2_{\bz}\Im\left\{\bGam^\omega_\alpha(\bz,\bz')\right\}:\MM(B_S):\overline{\nabla\bu^P_{j}(\bz')}\right)
\Bigg\}.
\label{intermid5}
\end{align} 
Recall that the term 
\begin{align*}
\frac{1}{n}\sum_{j=1}^n\Re\Bigg\{\nabla\bu^P_j(\bz):\MM(B_S)
:
\left(\nabla^2_{\bz}\Im\left\{\bGam^\omega_\alpha(\bz,\bz')\right\}:\MM(B_S):\overline{\nabla\bu^P_{j}(\bz')}\right)
\Bigg\},
\end{align*}
in Eq. \eqref{intermid5} is of the same form as Eq. \eqref{eq:33b} analyzed component-wise in Lemma \ref{LemTD-elasticity}. Therefore, following the same arguments, one can arrive at the expression for the covariance after fairly easy calculations as 
\begin{align*}
{\rm Cov}&\left(\mathcal{I}_{\rm W}^n(\bz;P),\mathcal{I}_{\rm W}^n(\bz';P)\right)
\approx
\frac{2\rho_0c_P^3\sigma_\bxi^2 \gamma_d }{\omega}\left(\frac{\pi}{\K_P}\right)^{d-2}
\frac{1}{n}\left(\MM(B_S)\bullet\Im\left\{\nabla_{\bz}^2{\bGam}^\omega_P(\bz,\bz')\right\}\right)
:
\left(\MM(B_S)\bullet\Im\left\{\nabla_{\bz}^2{\bGam}^\omega_P(\bz,\bz')\right\}\right)^\top.
\end{align*} 
Similarly, for incident plane S-waves, one finds out that the covariance of the weighted location indicator is 
\begin{align*}
{\rm Cov}&\left(\mathcal{I}_{\rm W}^n(\bz;S),\mathcal{I}_{\rm W}^n(\bz';S)\right)
\approx
\frac{2\rho_0c_S^3\sigma_\bxi^2 \gamma_d }{\omega}\left(\frac{\pi}{\K_S}\right)^{d-2}
\frac{1}{n}\left(\MM(B_S)\bullet\Im\left\{\nabla_{\bz}^2{\bGam}^\omega_S(\bz,\bz')\right\}\right)
:
\left(\MM(B_S)\bullet\Im\left\{\nabla_{\bz}^2{\bGam}^\omega_S(\bz,\bz')\right\}\right)^\top.
\end{align*} 
For analysis sake, let us assume that $\MM(B_S)=\MM(B_D)$. Then, 
\begin{align}
{\rm Cov}&\left(\mathcal{I}_{\rm W}^n(\bz;\alpha),\mathcal{I}_{\rm W}^n(\bz';\alpha)\right)
\approx
\frac{2\rho_0c_\alpha^3\sigma_\bxi^2 \gamma_d }{\omega}\left(\frac{\pi}{\K_\alpha}\right)^{d-2}
\frac{1}{n}\Psi_{\alpha,\alpha}(\bz,\bz').
\label{VarElasticity}
\end{align} 
Once again, as discussed in the density contrast case, the covariance of the speckle field has the same expression as in Eq. \eqref{Eq2:LemW} in Lemma \ref{LemW} up to a scaling by $\widetilde{C}\sigma_\bxi^2/n$, where $\widetilde{C}$ is independent of $\sigma_\xi$ and $n$. Therefore, the measurement noise has same effects on the performance of the weighted location indicator in elasticity contrast case as discussed in Section \ref{s:DensityMeasurment} for the density contrast. Further, if $\MM(B_S)=\MM(B_D)$ then the variance  of $\mathcal{I}^n_{\rm W}$ and the SNR, thanks to  Eq. \eqref{VarElasticity} and Lemma \ref{LemW}, are given by
\begin{align*}
&{\rm Var}\left(\mathcal{I}^n_{\rm W}(\bz;\alpha)\right)\approx\frac{2\rho_0c_S^3\sigma_\bxi^2 \gamma_d }{\omega}\left(\frac{\pi}{\K_S}\right)^{d-2}
\frac{1}{n} \Psi_{\alpha,\alpha}(\bz,\bz),
\\
&{\rm SNR} 
\approx\epsilon^d\frac{\sqrt{n}}{\sigma_\bxi}\sqrt{2^{4-d}\pi^{d-3}c_\alpha^{d-1}\omega^{1-d}}\sqrt{\Psi_{\alpha,\alpha}(\bz,\bz)}.
\end{align*}

\section{Statistical stability with random medium noise}\label{s:Medium}

Let us now investigate the stability of the weighted location indicator with respect to the medium noise, i.e., when the material parameters fluctuate around the reference material parameters. For brevity, it is supposed that either the density parameter is fluctuating around the reference value $\rho_0$ or the shear modulus is fluctuating around the reference value $\mu_0$. The general case, when density and both Lam\'e parameters fluctuate simultaneously, can be dealt with similarly but is much more intriguing and the calculations are tedious. 

\subsection{Randomly fluctuating density parameter}\label{ss:MediumDensity}

Let the density parameter of the elastic material loaded in the background medium $\RR^d$ be randomly fluctuating around the reference density parameter $\rho_0$, i.e.,
\begin{align*}
\rho(\bx):=\rho_0\left(1+\eta(\bx)\right), \qquad\bx\in\RR^d,
\end{align*}
where $\eta:\RR^d\to\RR$ is a (real valued) random fluctuation of typical size $\sigma_{\eta}$. It is assumed that the  fluctuations are weak, i.e., the size $\sigma_{\eta}$ is small enough in the sense that the Born approximation is justified and valid. 

Henceforth,  the quantities related to the random medium in the absence of any  inclusion will be termed as the \emph{background quantities} whereas those related to the medium without random fluctuations in the absence of any inclusion will be termed as the \emph{reference quantities}. The reference quantities will be marked by a superposed $0$. 

\subsubsection{Measurement model}
In order to understand the effects of random medium noise on the stability of  location indicator $\mathcal{I}^n_{\rm W}$, the contaminated data $\bu^\infty_\epsilon$ and the corrupted back-propagator $\mathcal{H}[\bu^\infty_{\epsilon,\alpha}]$ need to be elaborated. Toward this end, remark that the differential measurements of the far-field patterns are now characterized by 
\begin{align}
\bu^{\rm tot}(\bx)-\bu^{{\rm inc},0}(\bx)= \frac{e^{i\K_P|\bx|}}{|\bx|^{(d-1)/2}}\bu^{\infty}_{\epsilon,P}(\hbx)+ \frac{e^{i\K_S|\bx|}}{|\bx|^{(d-1)/2}}\bu^{\infty}_{\epsilon,S}(\hbx)+O\left(\frac{1}{|\bx|^{(d+1)/2}}\right),
\label{uScNoisy}
\end{align}
uniformly in all directions $\hbx\in\mathbb{S}^{d-1}$ as $\bx|\to +\infty$, wherein the reference field $\bu^{\rm inc, 0}$ is used. In fact, the background is fluctuating randomly and the differential measurements are with respect to the reference medium. It shows that a \emph{clutter noise} is induced by the random fluctuation $\eta$ which contaminates the far-field scattering patterns. In order to separate the original scattering signature of the inclusion from that of the clutter noise induced by $\eta$, the contaminated scattered measurements $\bu^{\rm tot}(\bx)-\bu^{{\rm inc},0}(\bx)$ are expressed as the sum of $\bu^{\rm tot}(\bx)-\bu^{{\rm inc}}(\bx)$ and $\bu^{\rm inc}(\bx)-\bu^{{\rm inc},0}(\bx)$. Subsequently, these quantities are  approximated in the Born approximation regime. It allows us to express $\bu^\infty_{\epsilon,\alpha}$  in terms of the reference far-field patterns $\bu^{\infty, 0}_{\epsilon,\alpha}$. 

Note that the incident field $\bu^{\rm inc}$ can be expressed under the Born approximation as 
\begin{align*}
\bu^{\rm inc}(\bx)=\bu^{\rm inc,0}(\bx)-\rho_0\omega^2\int_{\RR^d}\bGam^{\omega,0}(\bx,\by)\eta(\by)\bu^{\rm inc,0}(\by)d\by+o\left(\sigma_{\eta}\right).
\end{align*}
Thanks to asymptotic expansion \eqref{GamFar} in Proposition \ref{PropFar}, for $|\bx|\to+\infty$,
\begin{align*}
\bu^{\rm inc}(\bx)-\bu^{\rm inc,0}(\bx)
= &
-\rho_0\omega^2\Bigg[
a^P_d\frac{e^{i\K_P|\bx|}}{|\bx|^{(d-1)/2}}\int_{\RR^d}\left(\hbx\otimes\hbx \right)e^{-i\K_P\hbx\cdot\by}\eta(\by)\bu^{\rm inc,0}(\by)d\by
\nonumber
\\
&+
a^S_d\frac{e^{i\K_S|\bx|}}{|\bx|^{(d-1)/2}}\int_{\RR^d}\left(\I_d-\hbx\otimes\hbx \right)e^{-i\K_S\hbx\cdot\by}\eta(\by)\bu^{\rm inc,0}(\by)d\by
+o\left(\sigma_{\eta}\right)+O\left(\frac{1}{|\bx|^{(d+1)/2}}\right)
\Bigg]
\nonumber
\\
= &
-\rho_0\omega^2\Bigg[
a^P_d\frac{e^{i\K_P|\bx|}}{|\bx|^{(d-1)/2}}\int_{\RR^d}\left[\left(-\frac{1}{\K_P^2}\nabla\otimes\nabla\right)e^{-i\K_P\hbx\cdot\by}\right]\eta(\by)\bu^{\rm inc,0}(\by)d\by
\nonumber
\\
&+
a^S_d\frac{e^{i\K_S|\bx|}}{|\bx|^{(d-1)/2}}\int_{\RR^d}\left[\left(\I_d+\frac{1}{\K_S^2}\nabla\otimes\nabla\right)e^{-i\K_S\hbx\cdot\by}\right]\eta(\by)\bu^{\rm inc,0}(\by)d\by
+o\left(\sigma_{\eta}\right)+O\left(\frac{1}{|\bx|^{(d+1)/2}}\right)
\Bigg].
\end{align*}
Similarly, under the Born approximation
\begin{align}
\left(\bu^{\rm tot}(\bx)-\bu^{\rm inc}(\bx)\right)
=&
\left(\bu^{\rm tot,0}(\bx)-\bu^{\rm inc, 0}(\bx)\right)
-\rho_0\omega^2\int_{\RR^d}\bGam^{\omega,0}(\bx,\by)\eta(\by)\left(\bu^{\rm tot,0}(\by)-\bu^{\rm inc, 0}(\by)\right)d\by
+o\left(\sigma_{\eta}\right).
\label{intermid9}
\end{align}
It is interesting to note that the first term on the RHS of Eq. \eqref{intermid9} is exactly the scattered field in the reference medium, whereas the second term is of order $o(\omega^{d+2}\epsilon^d\sigma_{\eta})$ thanks to Corollary \ref{corAsymp}. Therefore, 
\begin{align*}
\left(\bu^{\rm tot}(\bx)-\bu^{\rm inc}(\bx)\right)
=&
\left(\bu^{\rm tot,0}(\bx)-\bu^{\rm inc, 0}(\bx)\right)
+o\left(\sigma_{\eta}\right)
+o\left(\omega^{d+2}\epsilon^d\sigma_{\eta}\right) .
\end{align*}
Consequently, when $|\bx|\to +\infty$, 
\begin{align}
\left(\bu^{\rm tot}(\bx)-\bu^{{\rm inc}}(\bx)\right)
=&\frac{e^{i\K_P|\bx|}}{|\bx|^{(d-1)/2}}\bu^{\infty,0}_{\epsilon,P}(\hbx)+ \frac{e^{i\K_S|\bx|}}{|\bx|^{(d-1)/2}}\bu^{\infty,0}_{\epsilon,S}(\hbx)
+o\left(\sigma_{\eta}\right)
+o\left(\omega^{d+2}\epsilon^d\sigma_{\eta}\right) 
+O\left(\frac{1}{|\bx|^{(d+1)/2}}\right).
\label{intermid10}
\end{align}
Finally, combining Eqs. \eqref{uScNoisy}, \eqref{intermid9}, and \eqref{intermid10}, one finds out that 
\begin{align}
\bu^{\infty}_{\epsilon,P}(\hbx)
=&
\bu^{\infty,0}_{\epsilon,P}(\hbx)
-\rho_0\omega^2 a^P_d\int_{\RR^d}\left[\frac{-1}{\K^2_P}\nabla\otimes\nabla\left(e^{-i\K_P\hbx\cdot\by}\right)\right]\eta(\by)\bu^{\rm inc, 0}(\by)d\by
+o\left(\sigma_{\eta}\right)
+o\left(\omega^{d+2}\epsilon^d\sigma_{\eta}\right) 
+O\left(\frac{1}{|\bx|^{(d+1)/2}}\right), \label{UpFarNoisy}
\\
\bu^{\infty}_{\epsilon,P}(\hbx)
=&
\bu^{\infty,0}_{\epsilon,S}(\hbx)
-\rho_0\omega^2 a^S_d\int_{\RR^d}\left[\I_d+\frac{1}{\K^2_S}\nabla\otimes\nabla\left(e^{-i\K_S\hbx\cdot\by}\right)\right]\eta(\by)\bu^{\rm inc, 0}(\by)d\by
+o\left(\sigma_{\eta}\right)
+o\left(\omega^{d+2}\epsilon^d\sigma_{\eta}\right) 
+O\left(\frac{1}{|\bx|^{(d+1)/2}}\right).\label{UsFarNoisy}
\end{align}

The \emph{noisy back-propagator} $\mathcal{H}[\bu^\infty_{\epsilon,\alpha}]$ can be calculated using the noisy far-field patterns \eqref{UpFarNoisy} and \eqref{UsFarNoisy}. After simple manipulations one finds out that, for all $\bz\in\RR^d$, 
\begin{align}
\mathcal{H}[\bu^\infty_{\epsilon,\alpha}](\bz)
= & \mathcal{H} [\bu^{\infty,0}_{\epsilon,\alpha}](\bz)
-4\rho_0^2c_\alpha^2\omega^2a^\alpha_d\left(\frac{\pi}{\K_\alpha}\right)^{d-2}
\int_{\RR^d}\Im\left\{\bGam^{\omega,0}_\alpha(\bz,\by)\right\}\eta(\by)\bu^{\rm inc,0}(\by)d\by
\nonumber
\\
&+o\left(\sigma_{\eta}\right)
+o\left(\omega^{d+2}\epsilon^d\sigma_{\eta}\right) 
+O\left(\frac{1}{|\bx|^{(d+1)/2}}\right).
\label{Hnoisy}
\end{align}

Notice that the first term on the RHS of Eq. \eqref{Hnoisy} is the back-propagator related entirely to the reference medium and will furnish the main peak of $\mathcal{I}^n_{\rm W}$ in the noise-free case under the assumption that the Born approximation is valid.  The second term on the RHS of Eq. \eqref{Hnoisy} is the error term  emerging from the clutter. This will generate a speckle field when supplied to the indicator $\mathcal{I}^n_{\rm W}$.  In order to facilitate the ensuing analysis, let us introduce error operator 
\begin{align*}
\mathcal{E}_\alpha\left[\bw\right](\bz):=-4\rho_0^2c_\alpha^2\omega^2a^\alpha_d\left(\frac{\pi}{\K_\alpha}\right)^{d-2}
\int_{\RR^d}\Im\left\{\bGam^{\omega,0}_\alpha(\bz,\by)\right\}\eta(\by)\bw(\by)d\by.
\end{align*}
For brevity, the shorthand notation $\mathcal{E}_\alpha^{P, j}$ will be used for the error associated to the plane P-wave incident field $\bu^{\rm P,0}_{j}$ and $\mathcal{E}_\alpha^{S, j,\ell}$ will be used for that associated to the plane S-wave $\bu^{\rm S,0}_{j,\ell}$.

\subsubsection{Speckle field analysis for density contrast}\label{sss:speckleDensity}

The back-propagators $\mathcal{E}_{\alpha}^{P,j}$ and  $\mathcal{E}_{\alpha}^{S,j,\ell}$, emerging in  Eq. \eqref{Hnoisy} due to the noise $\eta$, generate speckle fields which corrupt the true peaks of $\bz\mapsto \mathcal{I}_{\rm W}^n(\bz; P)$  and $\bz\mapsto \mathcal{I}_{\rm W}^n(\bz; S)$, respectively. To explain the effects of these speckle fields on the performance and stability of the location indicator, the covariances of these speckle fields is investigated below. Towards this end, the following result holds.
\begin{lem}\label{Lem5.1}
For all $\bz,\bz'\in\RR^d$, 
 \begin{align}
{\rm Cov}\left(\mathcal{I}_W^n(\bz;\alpha),\mathcal{I}_W^n(\bz';\alpha)\right)
\approx
8\rho_0^4c_\alpha^4\omega^6\gamma_d^2\widetilde{\tau}^2 \left(\frac{\pi}{\K_\alpha}\right)^{2d-4}\iint_{\RR^d\times \RR^d}\mathcal{C}_\eta(\by,\by') \Phi_{\alpha,\alpha}(\bz,\by)\Phi_{\alpha,\alpha}(\bz',\by')d\by d\by',\label{CV1}
\end{align}
where $\mathcal{C}_\eta$ is the two-point correlation of $\eta$ (i.e., $\mathcal{C}_\eta(\by,\by') :=\mathbb{E}\left[\eta(\by)\eta(\by')\right]$).
\end{lem} 
\begin{proof}
Let us first evaluate the speckle fields
\begin{align*}
\frac{1}{n}\sum_{j=1}^n a^P_d \overline{\mathcal{E}_P^{P,j}(\bz)}\cdot\bu^{P,0}_j(\bz)
\quad\text{and}\quad
\frac{1}{n}\sum_{j=1}^n \sum_{\ell=1}^{d-1}a^S_d \overline{\mathcal{E}_S^{S,j,\ell}(\bz)}\cdot\bu^{S,0}_{j,\ell}(\bz).
\end{align*}
By definition of the error function and  P-waves,  
\begin{align*}
\frac{1}{n}\sum_{j=1}^n a^P_d \overline{\mathcal{E}_P^{P,j}(\bz)}\cdot\bu^{P,0}_j(\bz)
=&
-4\rho_0^2c_P^2\omega^2|a^P_d|^2\left(\frac{\pi}{\K_P}\right)^{d-2}\frac{1}{n}\sum_{j=1}^n\int_{\RR^d}\Im\left\{\bGam^{\omega,0}_P(\bz,\by)\right\}\eta(\by)\overline{\bu^{P,0}_j(\by)}dy\cdot\bu^{P,0}_j(\bz)
\\
=&
-\frac{\rho_0 \gamma_d\K_P}{n}\sum_{j=1}^n\int_{\RR^d}\Im\left\{\bGam^{\omega,0}_P(\bz,\by)\right\}\eta(\by)\overline{\bu^{P,0}_j(\by)}\cdot\bu^{P,0}_j(\bz) dy
\\
=&
-\rho_0 \gamma_d\K_P\int_{\RR^d}\Im\left\{\bGam^{\omega,0}_P(\bz,\by)\right\}\eta(\by):\left(\frac{1}{n}\sum_{j=1}^n\btheta_j\otimes\btheta_j e^{i\K_P\btheta_j\cdot(\bz-\by)}\right)dy.
\end{align*}
Thanks to approximation \eqref{identity4}, for $n$ sufficiently large,
\begin{align}
\frac{1}{n}&\sum_{j=1}^n a^P_d \overline{\mathcal{E}_P^{P,j}(\bz)}\cdot\bu^{P,0}_j(\bz)
\approx
-4\rho_0^2c_P^2\gamma_d\K_P\left(\frac{\pi}{\K_P}\right)^{d-2}\int_{\RR^d}\eta(\by) \Phi_{P,P}(\bz,\by) d\by.
\label{T4}
\end{align}
Similarly, one easily obtains the approximation
\begin{align}
\frac{1}{n}&\sum_{j=1}^n\sum_{\ell=1}^{d-1} a^S_d \overline{\mathcal{E}_S^{S,j,\ell}(\bz)}\cdot\bu^{S,0}_{j,\ell}(\bz)
\approx
-4\rho_0^2c_S^2\gamma_d\K_S\left(\frac{\pi}{\K_S}\right)^{d-2}\int_{\RR^d}\eta(\by) \Phi_{S,S}(\bz,\by)d\by.
\label{T5}
\end{align}
Therefore, the covariance of the weighted location indicator is evaluated using Eq. \eqref{T4}  as 
\begin{align*}
{\rm Cov}\left(\mathcal{I}_W^n(\bz;P),\mathcal{I}_W^n(\bz';P)\right)
&=
\frac{c_P^2\omega^4\widetilde{\tau}^2}{n^2}\sum_{j,j'=1}^n\mathbb{E}\left[\Re\left\{ a^P_d \overline{\mathcal{E}_P^{P,j}(\bz)}\cdot\bu^{P,0}_j(\bz)\right\}\Re\left\{ a^P_d \overline{\mathcal{E}_P^{P,j'}(\bz')}\cdot\bu^{P,0}_{j'}(\bz')\right\}\right]
\nonumber
\\
&=
\frac{c_P^2\omega^4\widetilde{\tau}^2}{2}\Re\left\{\mathbb{E}\left[ \left(\frac{1}{n}\sum_{j=1}^na^P_d \overline{\mathcal{E}_P^{P,j}(\bz)}\cdot\bu^{P,0}_j(\bz) \right)\overline{\left(\frac{1}{n}\sum_{j'=1}^n a^P_d \overline{\mathcal{E}_P^{P,j'}(\bz')}\cdot\bu^{P,0}_{j'}(\bz')\right)}\right]\right\}
\nonumber
\\
&\approx
8\rho_0^4c_P^4\omega^6\gamma_d^2\widetilde{\tau}^2 \left(\frac{\pi}{\K_P}\right)^{2d-4}\iint_{\RR^d\times \RR^d}\mathbb{E}\left[\eta(\by)\eta(\by')\right] \Phi_{P,P}(\bz,\by)\Phi_{P,P}(\bz',\by')d\by d\by'. 
\end{align*}
By proceeding in the same manner and using approximation \eqref{T5}, on can get
\begin{align*}
{\rm Cov}\left(\mathcal{I}_W^n(\bz;S),\mathcal{I}_W^n(\bz';S)\right)
\approx
8\rho_0^4c_S^4\omega^6\gamma_d^2\widetilde{\tau}^2 \left(\frac{\pi}{\K_S}\right)^{2d-4}\iint_{\RR^d\times \RR^d}\mathbb{E}\left[\eta(\by)\eta(\by')\right] \Phi_{S,S}(\bz,\by)\Phi_{S,S}(\bz',\by')d\by d\by'.
\end{align*}
The conclusion follows from the last two equations and the definition of $\mathcal{C}_\eta$.
\end{proof}

\begin{rem}\label{rem1}
It is interesting to note, from estimates \eqref{T4}-\eqref{T5}, that the speckle field generated by the medium noise $\eta$ for  an $\alpha$-wave is roughly the medium noise smoothed by the integral kernel  $\Phi_{\alpha,\alpha}$. Moreover, its  correlation structure is approximately that of $\eta$ smoothed by the kernel $\Phi_{\alpha,\alpha}$ with a  correlation length given by the maximum of the wavelength of the incident field  and the correlation length of $\eta$, as suggested by Lemma \ref{Lem5.1}. Specifically, the speckle field has a correlation length of the order of the focal spot size of $\bz_S\to \mathcal{I}^n_{\rm W}(\bz_S;\alpha)$ in the neighborhood of $\bz_D$  so that it is not possible to recognize the peak procured by $\bz_S\to \mathcal{I}^n_{\rm W}(\bz_S;\alpha)$ based on its shape. The identification of the true peak is only possible through its height. It is worthwhile precising also  that the main peak of the location indicator is affected due to the medium noise. Finally, the estimates of the speckle field and its covariance both are independent of $n$ which means that the stability of the weighted location indicator relative to medium noise cannot be increased by increasing the number of illuminations on contrary to the measurement noise case. 
\end{rem}

\subsubsection{Speckle field analysis for elasticity contrast}\label{sss:speckleElasticity}

For analyzing the performance of the weighted location indicator in the elasticity contrast case, the covariance of the speckle field is estimated in Lemma \ref{Lem5.2}.
\begin{lem}\label{Lem5.2}
For all $\bz,\bz'\in\RR^d$, 
\begin{align}
{\rm Cov}\left(\mathcal{I}_W^n(\bz;\alpha),\mathcal{I}_W^n(\bz';\alpha)\right)
\approx
8\rho_0^4c_\alpha^4\omega^2\gamma_d^2 \left(\frac{\pi}{\K_\alpha}\right)^{2d-4}\iint_{\RR^d\times \RR^d}\mathcal{C}_\eta(\by,\by')\mathcal{Q}_{\eta}^\alpha[\MM(B_S)](\bz,\by)\mathcal{Q}_{\eta}^\alpha[\MM(B_S)](\bz',\by')d\by d\by',\label{CV2}
\end{align}
where $\mathcal{Q}_{\eta}^\alpha[\mathbb{A}]:\RR^d\times\RR^d\to \RR$, for any rank-four tensor $\mathbb{A}$, is defined by 
\begin{align}
\mathcal{Q}_{\eta}^\alpha[\mathbb{A}](\bx,\by):=
\Im\left\{\nabla_\bz\bGam^{\omega,0}_\alpha(\bx,\by)\right\}:\mathbb{A}\bullet\Im\left\{\nabla_\bz\bGam^{\omega,0}_\alpha(\bx,\by)\right\}.\label{Q-def}
\end{align}
\end{lem}

\begin{proof}
Let us prove the result for $\alpha=P$. The situation when $\alpha=S$ can be dealt with similarly.  Let us first evaluate the speckle field. By the definition of the error function, 
\begin{align*}
\frac{1}{n}\sum_{j=1}^n a^P_d &\overline{\nabla\mathcal{E}_P^{P,j}(\bz)}:\MM(B_S):\nabla\bu^{P,0}_j(\bz)
\\
=&
-4\rho_0^2c_P^2\omega^2|a^P_d|^2\left(\frac{\pi}{\K_P}\right)^{d-2}\frac{1}{n}\sum_{j=1}^n\int_{\RR^d}\Im\left\{\nabla_\bz\bGam^{\omega,0}_P(\bz,\by)\right\}\eta(\by)\overline{\bu^{P,0}_j(\by)} :\MM(B_S):\nabla\bu^{P,0}_j(\bz)d\by
\\
=&
-\frac{\rho_0 \gamma_d\K_P}{n}\sum_{j=1}^n\int_{\RR^d}\Im\left\{\nabla_\bz\bGam^{\omega,0}_P(\bz,\by)\right\}\eta(\by)\overline{\bu^{P,0}_j(\by)}:\MM(B_S):\nabla\bu^{P,0}_j(\bz)d\by
\\
=&
-\frac{\rho_0 \gamma_d\K_P}{n}\sum_{j=1}^n\int_{\RR^d}\eta(\by)\Im\left\{\nabla_\bz\bGam^{\omega,0}_P(\bz,\by)\right\}\btheta_j:\MM(B_S):i\K_P\btheta_j\otimes\btheta_j e^{i\K_P\btheta_j\cdot(\bz-\by)}d\by.
\end{align*}
After proceeding component-wise, one can easily check out that 
\begin{align*}
\frac{1}{n}&\sum_{j=1}^n a^P_d \overline{\nabla\mathcal{E}_P^{P,j}(\bz)}:\MM(B_S):\nabla\bu^{P,0}_j(\bz)
=
-\frac{\rho_0 \gamma_d\K_P}{n}\sum_{j=1}^n\int_{\RR^d}\eta(\by)\MM(B_S)\bullet\Im\left\{\nabla_\bz\bGam^{\omega,0}_P(\bz,\by)\right\}:i\K_P\btheta_j\otimes\btheta_j\otimes\btheta_j e^{i\K_P\btheta_j\cdot(\bz-\by)}d\by.
\end{align*}
Consequently, thanks to approximation \eqref{identity4},
\begin{align}
\frac{1}{n}\sum_{j=1}^n a^P_d \overline{\nabla\mathcal{E}_P^{P,j}(\bz)}:\MM(B_S):\nabla\bu^{P,0}_j(\bz)
\approx
-4\rho_0^2 \gamma_dc_P^2\K_P\left(\frac{\pi}{\K_P}\right)^{d-2}\int_{\RR^d}\eta(\by)\mathcal{Q}_{\eta}^P[\MM(B_S)](\bz,\by)d\by.
\label{T6}
\end{align}

Let us now estimate the covariance of the weighted location indicator. In elasticity contrast case, 
\begin{align*}
{\rm Cov}&\left(\mathcal{I}_W^n(\bz;P),\mathcal{I}_W^n(\bz';P)\right)
\\
&=
\frac{c_P^2  }{n^2}\sum_{j,j'=1}^n\mathbb{E}\Bigg[\Re\left\{ a^P_d \overline{\nabla\mathcal{E}_P^{P,j}(\bz)}:\MM(B_S):\nabla\bu^{P,0}_j(\bz)\right\}
\Re\left\{ a^P_d \overline{\nabla\mathcal{E}_P^{P,j'}(\bz')}:\MM(B_S):\nabla\bu^{P,0}_{j'}(\bz')\right\}\Bigg]
\\
&=
\frac{c_P^2 }{2}\Re\Bigg\{\mathbb{E}\Bigg[ \left(\frac{1}{n}\sum_{j=1}^na^P_d \overline{\nabla\mathcal{E}_P^{P,j}(\bz)}:\MM(B_S):\nabla\bu^{P,0}_j(\bz) \right)
\overline{\left(\frac{1}{n}\sum_{j'=1}^n a^P_d \overline{\nabla\mathcal{E}_P^{P,j'}(\bz')}:\MM(B_S):\nabla\bu^{P,0}_{j'}(\bz')\right)}\Bigg]\Bigg\}.
\end{align*}
Therefore, by virtue of estimate \eqref{T6},
\begin{align*}
{\rm Cov}\left(\mathcal{I}_W^n(\bz;P),\mathcal{I}_W^n(\bz';P)\right)
\approx
8\rho_0^4c_P^4\omega^2\gamma_d^2 \left(\frac{\pi}{\K_P}\right)^{2d-4}\iint_{\RR^d\times \RR^d}\mathcal{C}_\eta(\by,\by')\mathcal{Q}_{\eta}^P[\MM(B_S)](\bz,\by)\mathcal{Q}_{\eta}^P[\MM(B_S)](\bz',\by')d\by d\by'.
\end{align*}
\end{proof}

It can demonstrated that $\mathcal{Q}_\eta^\alpha[\MM(B_S)]$ defined in Eq. \eqref{Q-def} is a non-negative real-valued function for both $\alpha=P$ and $\alpha=S$. For example, when $B_S$ is a ball in $\RR^d$ and $\MM(B_S)=\MM(B_D)$,
\begin{align*}
&\mathcal{Q}_\eta^P[\MM(B_S)](\bx,\by):=a\left\|\Im\left\{\nabla_\bx\bGam^{\omega,0}_P(\bx,\by)\right\}\right\|^2+b\left\|\Im\left\{\nabla_\bx\cdot\bGam^{\omega,0}_P(\bx,\by)\right\}\right\|^2,
\\
&\mathcal{Q}_\eta^S[\MM(B_S)](\bx,\by):=\frac{a}{2}\left\|\Im\left\{\nabla_\bx\bGam^{\omega,0}_S(\bx,\by)\right\}\right\|^2+\frac{a}{2}\Im\left\{\nabla_\bx\bGam^{\omega,0}_S(\bx,\by)\right\}:\Im\left\{\left(\nabla_\bx\bGam^{\omega,0}_S(\bx,\by)\right)^\top\right\}+b\left\|\Im\left\{\nabla_\bx\cdot\bGam^{\omega,0}_S(\bx,\by)\right\}\right\|^2.
\end{align*} 
Thanks to this observation and the comparison of the speckle fields (\eqref{T4}-\eqref{T5} and \eqref{T6}) and their covariances (\eqref{CV1} and \eqref{CV2}), one can draw similar conclusions regarding the effects of medium noise on the performance of  the location indicator in the  elasticity contrast case as in Remark \ref{rem1}. The only difference is that the smoothing kernel $\mathcal{Q}_\eta^\alpha$ replaces the kernel $\Phi_{\alpha,\alpha}$ in the elasticity contrast case (which also is non-negative and achieves its maximum value at $\bz_S=\bz_D$). 

\subsection{Randomly fluctuating elastic parameters}

Let  the Lam\'e parameters of the background medium fluctuate around the reference values. For brevity, it is assumed that only the shear modulus fluctuates around its reference value $\mu_0$. The general case is delicate but amenable to the same treatment. 

Let $\mu$ denote the randomly fluctuating shear modulus defined by 
\begin{align*}
\mu(\bx):=\mu_0\left(1+\varphi(\bx)\right), \qquad\bx\in\RR^d,
\end{align*}
where $\varphi$ is a random process. It is assumed that $\varphi$ is compactly supported in $\RR^d$ with typical size $\sigma_\varphi$ and the fluctuation is weak enough that the terms of the order $o(\sigma_\varphi)$ are negligible  so that the Born approximation is justified. As for the medium with randomly fluctuating  density parameter, first the noisy measurements are modeled with the goal of approximating the noisy back-propagator as the sum of the reference back-propagator and the clutter error term. 

\subsubsection{Measurement model}
The incident field in the background domain now satisfies 
\begin{align*}
\OL_{\lambda_0,\mu_0}\left[\bu^{\rm inc}\right](\bx)+\rho_0\omega^2\bu^{\rm inc}(\bx)= - 2\nabla\cdot \varphi\nabla^s\bu^{\rm inc}(\bx), \qquad\bx\in\RR^d.
\end{align*}
Thanks to the Born approximation and asymptotic expansion \eqref{GamFar}, for $|\bx|\to+\infty$,
\begin{align}
\bu^{\rm inc}(\bx)-\bu^{\rm inc,0}(\bx)
=&
- 2\int_{\RR^d}\bGam^{\omega,0}(\bx,\by)\nabla\cdot\varphi(\by)\nabla^s\bu^{\rm inc,0}(\by)d\by+o\left(\sigma_{\varphi}\right)
\nonumber
\\
= &
-2\Bigg[
a^P_d\frac{e^{i\K_P|\bx|}}{|\bx|^{(d-1)/2}}\int_{\RR^d}\left(\hbx\otimes\hbx \right)e^{-i\K_P\hbx\cdot\by}\nabla\cdot\varphi(\by)\nabla^s\bu^{\rm inc,0}(\by)d\by
\nonumber
\\
&+
a^S_d\frac{e^{i\K_S|\bx|}}{|\bx|^{(d-1)/2}}\int_{\RR^d}\left(\I_d-\hbx\otimes\hbx \right)e^{-i\K_S\hbx\cdot\by}\nabla\cdot\varphi(\by)\nabla^s\bu^{\rm inc,0}(\by)d\by
+o\left(\sigma_{\varphi}\right)+O\left(\frac{1}{|\bx|^{(d+1)/2}}\right)
\Bigg]
\nonumber
\\
= &
-2\Bigg[
a^P_d\frac{e^{i\K_P|\bx|}}{|\bx|^{(d-1)/2}}\int_{\RR^d}\left[\left(-\frac{1}{\K_P^2}\nabla\otimes\nabla\right)e^{-i\K_P\hbx\cdot\by}\right]\nabla\cdot\varphi(\by)\nabla^s\bu^{\rm inc,0}(\by)d\by
\nonumber
\\
&+
a^S_d\frac{e^{i\K_S|\bx|}}{|\bx|^{(d-1)/2}}\int_{\RR^d}\left[\left(\I_d+\frac{1}{\K_S^2}\nabla\otimes\nabla\right)e^{-i\K_S\hbx\cdot\by}\right]\nabla\cdot\varphi(\by)\nabla^s\bu^{\rm inc,0}(\by)d\by
+o\left(\sigma_{\varphi}\right)+O\left(\frac{1}{|\bx|^{(d+1)/2}}\right)
\Bigg].
\label{intermid11}
\end{align}
Moreover, 
\begin{align}
&\left(\bu^{\rm tot}(\bx)-\bu^{\rm inc}(\bx)\right)
= 
\left(\bu^{\rm tot,0}(\bx)-\bu^{\rm inc, 0}(\bx)\right)
-2\int_{\RR^d}\bGam^{\omega,0}(\bx,\by)\nabla\cdot\varphi\nabla^s\left(\bu^{\rm tot,0}(\by)-\bu^{\rm inc, 0}(\by)\right)d\by+o\left(\sigma_{\varphi}\right).
\label{intermid12}
\end{align}
As the second term on the RHS of Eq. \eqref{intermid12} is of the order $o(\omega^{d+2}\epsilon^d\sigma_{\varphi})$ thanks to Corollary \ref{corAsymp}, 
\begin{align*}
\left(\bu^{\rm tot}(\bx)-\bu^{\rm inc}(\bx)\right)
=&
\left(\bu^{\rm tot,0}(\bx)-\bu^{\rm inc, 0}(\bx)\right)
+o\left(\sigma_{\varphi}\right)
+o\left(\omega^{d+2}\epsilon^d\sigma_{\varphi}\right) .
\end{align*}
Consequently, from Eqs. \eqref{intermid11} and \eqref{intermid12},  for $|\bx|\to+\infty$,
\begin{align}
\bu^{\rm tot}(\bx)-\bu^{{\rm inc,0}}(\bx)
=&\frac{e^{i\K_P|\bx|}}{|\bx|^{(d-1)/2}}
\left(\bu^{\infty,0}_{\epsilon,P}(\hbx)-2a^P_d\int_{\RR^d}\left[\left(-\frac{1}{\K_P^2}\nabla\otimes\nabla\right)e^{-i\K_P\hbx\cdot\by}\right]\nabla\cdot\varphi(\by)\nabla^s\bu^{\rm inc,0}(\by)d\by\right)
\nonumber
\\
&+ \frac{e^{i\K_S|\bx|}}{|\bx|^{(d-1)/2}}
\left(\bu^{\infty,0}_{\epsilon,S}(\hbx)-2a^S_d
\int_{\RR^d}\left[\left(\I_d+\frac{1}{\K_S^2}\nabla\otimes\nabla\right)e^{-i\K_S\hbx\cdot\by}\right]\nabla\cdot\varphi(\by)\nabla^s\bu^{\rm inc,0}(\by)d\by
\right)
\nonumber
\\
&
+o\left(\sigma_{\varphi}\right)
+o\left(\omega^{d+2}\epsilon^d\sigma_{\varphi}\right) 
+O\left(\frac{1}{|\bx|^{(d+1)/2}}\right).
\label{intermid13}
\end{align}
Since the noisy data $\bu^\infty_{\epsilon,\alpha}(\hbx)$  is defined by
\begin{align*}
\bu^{\rm tot}(\bx)-\bu^{{\rm inc},0}(\bx)= \frac{e^{i\K_P|\bx|}}{|\bx|^{(d-1)/2}}\bu^{\infty}_{\epsilon,P}(\hbx)+ \frac{e^{i\K_S|\bx|}}{|\bx|^{(d-1)/2}}\bu^{\infty}_{\epsilon,S}(\hbx)+O\left(\frac{1}{|\bx|^{(d+1)/2}}\right),
\end{align*} 
it is easy to see that, by virtue of Eq. \eqref{intermid13}, 
\begin{align}
\bu^{\infty}_{\epsilon,P}(\hbx)
=&
\bu^{\infty,0}_{\epsilon,P}(\hbx)
-2 a^P_d\int_{\RR^d}\left[\frac{-1}{\K^2_P}\nabla\otimes\nabla\left(e^{-i\K_P\hbx\cdot\by}\right)\right]\nabla\cdot\varphi(\by)\nabla^s\bu^{\rm inc, 0}(\by)d\by
+o\left(\sigma_{\varphi}\right)
+o\left(\omega^{d+2}\epsilon^d\sigma_{\varphi}\right) 
+O\left(\frac{1}{|\bx|^{(d+1)/2}}\right), \label{UpFarNoisy2}
\\
\bu^{\infty}_{\epsilon,P}(\hbx)
=&
\bu^{\infty,0}_{\epsilon,S}(\hbx)
-2 a^S_d\int_{\RR^d}\left[\I_d+\frac{1}{\K^2_S}\nabla\otimes\nabla\left(e^{-i\K_S\hbx\cdot\by}\right)\right]\nabla\cdot\varphi(\by)\nabla^s\bu^{\rm inc, 0}(\by)d\by
+o\left(\sigma_{\varphi}\right)
+o\left(\omega^{d+2}\epsilon^d\sigma_{\varphi}\right) 
+O\left(\frac{1}{|\bx|^{(d+1)/2}}\right).\label{UsFarNoisy2}
\end{align}
Finally, the noisy back-propagator $\mathcal{H}[\bu^\infty_{\epsilon,\alpha}]$ can be approximated using the noisy far-field patterns \eqref{UpFarNoisy2} and \eqref{UsFarNoisy2} as
\begin{align*}
\mathcal{H}[\bu^\infty_{\epsilon,\alpha}](\bz)
\approx & \mathcal{H} [\bu^{\infty,0}_{\epsilon,\alpha}](\bz)
-\widetilde{\mathcal{E}}_\alpha\left[\bu^{\rm inc,0}\right](\bz),\quad  \bz\in\RR^d,
\end{align*}
where the error operator is now defined by 
\begin{align*}
\widetilde{\mathcal{E}}_\alpha\left[\bw\right](\bz):=-8\rho_0c_\alpha^2a^\alpha_d\left(\frac{\pi}{\K_\alpha}\right)^{d-2}
\int_{\RR^d}\Im\left\{\bGam^{\omega,0}_\alpha(\bz,\by)\right\}\nabla\cdot\varphi(\by)\nabla^s\bw(\by)d\by.
\end{align*}
Henceforth, the  notation $\widetilde{\mathcal{E}}_\alpha^{P, j}$ and $\widetilde{\mathcal{E}}_\alpha^{S, j,\ell}$ is used for the errors associated to  $\bu^{\rm P,0}_{j}$ and $\bu^{\rm S,0}_{j,\ell}$, respectively. Let us now study the speckle field generated by the back-propagator $\widetilde{\mathcal{E}}_\alpha$. Only the results for the density contrast case are presented. The elasticity contrast case can be discussed  as in Section \ref{sss:speckleElasticity}.

\subsubsection{Speckle pattern for density contrast}

Let us proceed  as in Section \ref{ss:MediumDensity} and evaluate the speckle fields,
\begin{align*}
\frac{1}{n}\sum_{j=1}^n a^P_d \overline{\widetilde{\mathcal{E}}_P^{P,j}(\bz)}\cdot\bu^{P,0}_j(\bz)
\quad\text{and}\quad
\frac{1}{n}\sum_{j=1}^n \sum_{\ell=1}^{d-1}a^S_d \overline{\widetilde{\mathcal{E}}_S^{S,j,\ell}(\bz)}\cdot\bu^{S,0}_{j,\ell}(\bz).
\end{align*}
Accordingly, for incident P-waves, 
\begin{align*}
\frac{1}{n}\sum_{j=1}^n a^P_d \overline{\widetilde{\mathcal{E}}_P^{P,j}(\bz)}\cdot\bu^{P,0}_j(\bz)
=&
-8\rho_0c_P^2|a^P_d|^2\left(\frac{\pi}{\K_P}\right)^{d-2}\frac{1}{n}\sum_{j=1}^n \int_{\RR^d}\Im\left\{\bGam^{\omega,0}_{P}(\bz,\by)\right\}\nabla\cdot\varphi(\by)\overline{\nabla^s\bu^{P,0}_j(\by)}d\by\cdot\bu^{P,0}_j(\bz)
\\
=&
-\frac{2\gamma_d}{c_P\omega}\frac{1}{n}\sum_{j=1}^n \int_{\RR^d}\Im\left\{\bGam^{\omega,0}_{P}(\bz,\by)\right\}\overline{\left(i\K_P\nabla\varphi(\by)\cdot\btheta_j-\K_P^2\varphi(\by)\right)\bu^{P,0}_j(\by)}\cdot\bu^{P,0}_j(\bz)d\by.
\end{align*}
One substituting the expressions for $\bu^{P,0}_j$ and doing a few simple manipulations, one arrives at 
\begin{align*}
\frac{1}{n}\sum_{j=1}^n a^P_d \overline{\widetilde{\mathcal{E}}_P^{P,j}(\bz)}\cdot\bu^{P,0}_j(\bz)
=&
\frac{2\gamma_d}{c_P\omega}\int_{\RR^d}\Bigg[
\left(\frac{1}{n}\sum_{j=1}^ni\K_P\btheta_j\otimes\btheta_j\otimes\btheta_je^{i\K_P\btheta_j\cdot(\bz-\by)}\right):\nabla\varphi(\by)\otimes\Im\left\{\bGam^{\omega,0}_{P}(\bz,\by)\right\}
\\
&\qquad
+
\K_P^2\left(\frac{1}{n}\sum_{j=1}^n\btheta_j\otimes\btheta_je^{i\K_P\btheta_j\cdot(\bz-\by)}\right):\varphi(\by)\Im\left\{\bGam^{\omega,0}_{P}(\bz,\by)\right\}
\Bigg]
d\by.
\end{align*}
By invoking identities \eqref{identity10} and \eqref{identity4}, one gets
\begin{align*}
\frac{1}{n}\sum_{j=1}^n& a^P_d \overline{\widetilde{\mathcal{E}}_P^{P,j}(\bz)}\cdot\bu^{P,0}_j(\bz)
\\
\approx&
\frac{8\rho_0\gamma_d}{\K_P}\left(\frac{\pi}{\K_P}\right)^{d-2}\int_{\RR^d}\Bigg[
\Im\left\{\nabla_\bz\bGam^{\omega,0}_{P}(\bz,\by)\right\}:\nabla\varphi(\by)\otimes\Im\left\{\bGam^{\omega,0}_{P}(\bz,\by)\right\}
+
\K_P^2\Im\left\{\bGam^{\omega,0}_{P}(\bz,\by)\right\}:\varphi(\by)\Im\left\{\bGam^{\omega,0}_{P}(\bz,\by)\right\}
\Bigg]
d\by.
\end{align*}
Since $\varphi$ is compactly supported, one can simplify the above expression as 
\begin{align}
\frac{1}{n}\sum_{j=1}^n a^P_d \overline{\widetilde{\mathcal{E}}_P^{P,j}(\bz)}\cdot\bu^{P,0}_j(\bz)
\approx&
\frac{4\rho_0\gamma_d}{\K_P}\left(\frac{\pi}{\K_P}\right)^{d-2}
\int_{\RR^d}
\varphi(\by)\left(2\K_P^2+\Delta\right)\left\|\Im\left\{\bGam^{\omega,0}_{P}(\bz,\by)\right\}\right\|^2
d\by
\nonumber
\\
=&
\frac{4\rho_0\gamma_d}{\K_P}\left(\frac{\pi}{\K_P}\right)^{d-2}
\int_{\RR^d}
\left[\left(2\K_P^2+\Delta\right)\varphi(\by)\right] \Phi_{P,P}(\bz,\by)
d\by.
\label{T8}
\end{align}
Similarly, for incident S-waves, 
\begin{align*}
\frac{1}{n}&\sum_{j=1}^n \sum_{\ell=1}^{d-1} a^S_d \overline{\widetilde{\mathcal{E}}_S^{S,j,\ell}(\bz)}\cdot\bu^{S,0}_{j,\ell}(\bz)
\\
=&
\frac{2\gamma_d}{c_S\omega}\frac{1}{n}\sum_{j=1}^n \sum_{\ell=1}^{d-1}\int_{\RR^d}\Im\left\{\bGam^{\omega,0}_{S}(\bz,\by)\right\}
\left[\left(i\K_S\nabla\varphi(\by)\cdot\btheta_j+\K_S^2\varphi(\by)\right)\overline{\bu^{S,0}_{j,\ell}(\by)}+i\K_S\left(\nabla\varphi(\by)\cdot\overline{\bu^{S,0}_{j,\ell}(\by)}\right)\btheta_j\right]\cdot\bu^{S,0}_{j,\ell}(\bz)d\by.
\end{align*}
On substituting the expressions for $\bu^{S,0}_{j,\ell}$ and after a few simple manipulations, one arrives at 
\begin{align*}
\frac{1}{n}\sum_{j=1}^n\sum_{\ell=1}^{d-1}a^S_d \overline{\widetilde{\mathcal{E}}_S^{S,j,\ell}(\bz)}\cdot\bu^{S,0}_{j,\ell}(\bz)
=&
\frac{2\gamma_d}{c_S\omega}\int_{\RR^d}\Bigg[
\left(\frac{1}{n}\sum_{j=1}^n\sum_{\ell=1}^{d-1}i\K_S\btheta_j^{\perp,\ell}\otimes\btheta_j^{\perp,\ell}\otimes\btheta_je^{i\K_S\btheta_j\cdot(\bz-\by)}\right):\nabla\varphi(\by)\otimes\Im\left\{\bGam^{\omega,0}_{S}(\bz,\by)\right\}
\\
&+
\frac{2\gamma_d}{c_S\omega}\int_{\RR^d}\Bigg[
\left(\frac{1}{n}\sum_{j=1}^n\sum_{\ell=1}^{d-1}i\K_S\btheta_j\otimes\btheta_j^{\perp,\ell}\otimes\btheta_j^{\perp,\ell}e^{i\K_S\btheta_j\cdot(\bz-\by)}\right):\nabla\varphi(\by)\otimes\Im\left\{\bGam^{\omega,0}_{S}(\bz,\by)\right\}
\\
&+
\K_S^2\left(\frac{1}{n}\sum_{j=1}^n\sum_{\ell=1}^{d-1}\btheta_j^{\perp,\ell}\otimes\btheta_j^{\perp,\ell} e^{i\K_S\btheta_j\cdot(\bz-\by)}\right):\varphi(\by)\Im\left\{\bGam^{\omega,0}_{S}(\bz,\by)\right\}
\Bigg]
d\by.
\end{align*}
Finally, using the identities in Proposition \ref{Prop:identities},
\begin{align}
\frac{1}{n}\sum_{j=1}^n\sum_{\ell=1}^{d-1}a^S_d \overline{\widetilde{\mathcal{E}}_S^{S,j,\ell}(\bz)}\cdot\bu^{S,0}_{j,\ell}(\bz)
\approx&
\frac{8\rho_0\gamma_d}{\K_S}\left(\frac{\pi}{\K_S}\right)^{d-2}\int_{\RR^d}\Bigg[
\Im\left\{\nabla_\bz\bGam^{\omega,0}_{S}(\bz,\by)\right\}:\nabla\varphi(\by)\otimes\Im\left\{\bGam^{\omega,0}_{S}(\bz,\by)\right\}
\nonumber
\\
&+
\Im\left\{\nabla_\bz\bGam^{\omega,0}_{S}(\bz,\by)\right\}^\top:\nabla\varphi(\by)\otimes\Im\left\{\bGam^{\omega,0}_{S}(\bz,\by)\right\}
+
\K_S^2\Im\left\{\bGam^{\omega,0}_{S}(\bz,\by)\right\}:\varphi(\by)\Im\left\{\bGam^{\omega,0}_{S}(\bz,\by)\right\}
\Bigg]
d\by.
\label{T9}
\end{align}

The expressions \eqref{T8} and \eqref{T9}  suggest that the salient features of the speckle fields remain the same as discussed in Section \ref{ss:MediumDensity} when there is a random fluctuation in the background shear modulus . Precisely, the speckle field can still be regarded as a smoothing transform of the random fluctuation or its derivatives. Although, the kernel of this integral transformation is very complicated but has a diameter of the order of wavelength. Therefore, it can be concluded that the speckle field has a correlation length proportional to the maximum of the focal spot size of $\bz_S\mapsto\mathcal{I}^n_{\rm W}(\bz_S,\alpha)$ and the correlation length of $\varphi$. Therefore, the identification of the true peak of $\mathcal{I}^n_{\rm W}$ is possible only through its height that may also be affected due to the medium noise. Moreover, the stability features of the weighted location indicator with respect to the fluctuations in the shear modulus  are independent of the number of illuminations so that no more stability can be gained by increasing the number of incident fields. 

\section{Conclusion}\label{s:conclusion}

In this article, single- and multi-short topological sensitivity based location indicators for detecting homogeneous elastic inclusions were analyzed. It was demonstrated that the classical location indicators may fail to locate an inclusion when the mode-conversion artifacts are strong as the corresponding topological derivative is not expected to observe its most-prominent decrease at  the location of the inclusion. In order to tackle the mode-conversion phenomenon, weighted location indicators were designed by correlating the shear and  pressure components of the back-propagators with those of the incident fields and taking a weighted aggregate. These indicators appeared to be topological sensitivity based indicators (i.e., associated to the topological derivative of a cost functional) when the inclusion had only a contrast with the background medium in the density parameter. However, when the contrast between inclusion and the background was in terms of the Lam\'e parameters, the weighted indicator was found to be rather of the Kirchhoff type. The stability of the weighted indicators was discussed rigorously with respect to an additive measurement noise and a random medium noise. It was demonstrated that the weighted indicators are stable with respect to measurement noise whereas moderately stable with respect to random medium noise. 

\appendix

\section{Proof of proposition \ref{Prop:identities}}\label{A:Identities}

The first three identities are only derived here, the rest can be derived analogously. The identity \eqref{identity3} easily follows from the fact that 
\begin{align*}
\int_{\mathbb{S}^{d-1}}e^{i\K_\alpha\hbx\cdot(\by-\bz)}d\sigma(\hbx) =
\begin{cases}
j_0(\K_P|\by-\bz|), & d=3,
 \\
 J_0(\K_P|\by-\bz|), & d=2,
 \end{cases}
\end{align*}
where $j_0$ and $J_0$ are spherical and cylindrical Bessel functions of first kind and order zero, respectively. In order to prove the rest of the identities, note that, for $n$ sufficiently large, 
\begin{align*}
\frac{1}{n}\sum_{j=1}^n\btheta_j\otimes\btheta_j e^{i\K_P(\by-\bz)\cdot\btheta_j}
\approx
-\frac{1}{\K_P^2} \nabla\otimes \nabla \int_{\mathbb{S}^{d-1}}  e^{i\K_P(\by-\bz)\cdot\btheta} d\btheta,
\end{align*}
and 
$$
\btheta_j\otimes\btheta_j+\sum_{\ell=1}^{d-1}\btheta_j^{\perp,\ell}\otimes\btheta_j^{\perp,\ell}=\I_d,
$$
since $\{\btheta_j,\btheta_j^{\perp,\ell}\}$ forms a basis for $\RR^d$ for each $j=1,\cdots, n$. Therefore, the above identities together with Eq. \eqref{identity3} furnish,
\begin{align*}
&\frac{1}{n}\sum_{j=1}^n\btheta_j\otimes\btheta_j e^{i\K_P(\by-\bz)\cdot\btheta_j}\approx-\frac{4}{\K_P^2}\left(\frac{\pi}{\K_P}\right)^{d-2} \nabla\otimes \nabla\left( \Im\left\{ g_P(\by-\bz)\right\}\right)
= 4\rho_0 c_P^2\left(\frac{\pi}{\K_P}\right)^{d-2}  \Im\left\{ \bGam^\omega_P(\by,\bz)\right\},
\\
&\frac{1}{n}\sum_{j=1}^n\sum_{\ell=1}^{d-1} \btheta_j^{\perp,\ell}\otimes\btheta_j^{\perp,\ell} e^{i\K_S(\by-\bz)\cdot\btheta_j}
\approx 
\left(\I_d+\frac{1}{\K_S^2} \nabla\otimes \nabla\right)\int_{\mathbb{S}^{d-1}}  e^{i\K_S(\by-\bz)\cdot\btheta} d\btheta 
\approx 
4\rho_0 c_S^2\left(\frac{\pi}{\K_S}\right)^{d-2}  \Im\left\{ \bGam^\omega_S(\by,\bz)\right\}.
\end{align*}

\section{Proof of proposition \ref{PropFar}}\label{A:PropFar}
Let us prove the second identity for $d=3$. The case for $d=2$ is identical. For the first identity, refer, for instance, to \cite{AlvesKress, Gintides}.  Remark that, for all $\K\in\RR_+$, 
\begin{align}
\nabla_\by\left(\frac{(\bx-\by)\otimes (\bx-\by)}{|\bx-\by|^3}e^{i\K|\bx-\by|}\right)
=&\sum_{i,j,k} \frac{\partial}{\partial x_i}\left(\frac{(x_j-y_j)(x_k-y_k)}{|\bx-\by|^3}e^{i\K|\bx-\by|}\right) \hbe_i\otimes\hbe_j\otimes\hbe_k
\nonumber
\\
=&\sum_{i,j,k} i\K\left(\frac{(x_i-y_i)(x_j-y_j)(x_k-y_k)}{|\bx-\by|^4}e^{i\K|\bx-\by|}\right) \hbe_i\otimes\hbe_j\otimes\hbe_k
+\mathcal{O}\left(\frac{1}{|\bx|^2}\right)
\nonumber
\\
=& i\K\, e^{-\i\K\hbx\cdot\by} \frac{e^{i\K|\bx|} }{|\bx|}\hbx\otimes\hbx\otimes\hbx+\mathcal{O}\left(\frac{1}{|\bx|^2}\right).
\label{T1}
\end{align}
Similarly, for all $\K\in\RR_+$,
\begin{align}
\nabla_\by\left(\frac{e^{i\K|\bx-\by|}}{|\bx-\by|}\I_3\right) 
=&\sum_{i}\frac{\partial}{\partial x_i}\left(\frac{e^{i\K|\bx-\by|}}{|\bx-\by|}\right) \hbe_i \otimes \I_3+\mathcal{O}\left(\frac{1}{|\bx|^2}\right)
=  i\K\, e^{-\i\K\hbx\cdot\by} \frac{e^{i\K|\bx|} }{|\bx|}\hbx\otimes\I_3+\mathcal{O}\left(\frac{1}{|\bx|^2}\right).
\label{T2}
\end{align}
Therefore, thanks to Eqs. \eqref{T1} and \eqref{T2}, the expression \eqref{Green_fun} leads to 
\begin{align}
\nabla_\by\bGam^\omega(\bx,\by) 
=& 
\frac{1}{\mu_0}\nabla_\by\left[\left(\I_3+\frac{1}{\K_S^2}\nabla_\by\otimes\nabla_\by\right)\frac{e^{i\K_S|\bx-\by|}}{4\pi |\bx-\by|}-\frac{1}{\K_S^2}\nabla_\by\otimes\nabla_\by\frac{e^{i\K_P|\bx-\by|}}{4\pi|\bx-\by|}\right] 
\nonumber
\\
=&
\frac{1}{4\pi\mu_0}\nabla_\by\left[\frac{e^{i\K_S|\bx-\by|}}{|\bx-\by|}\I_3-\frac{(\bx-\by)\otimes(\bx-\by)}{|\bx-\by|^3}e^{i\K_S|\bx-\by|}\right]
+\frac{1}{4\pi\mu_0}\frac{\K_P^2}{\K_S^2}\nabla_\by\left[\frac{(\bx-\by)\otimes(\bx-\by)}{|\bx-\by|^3}e^{i\K_P|\bx-\by|} \right]
\nonumber
\\
=&\frac{i\K_Se^{-\i\K_S\hbx\cdot\by} }{4\pi\mu_0}\left[\hbx\otimes\I_3-\hbx\otimes\hbx\otimes\hbx\right]\frac{e^{i\K_S|\bx|} }{|\bx|}
+\frac{i\K_P e^{-\i\K_P\hbx\cdot\by} }{4\pi(\lambda_0+2\mu_0)}\hbx\otimes\hbx\otimes\hbx \frac{e^{i\K_P|\bx|} }{|\bx|}+\mathcal{O}\left(\frac{1}{|\bx|^2}\right).
\label{T3}
\end{align}
The conclusion follows from Eq. \eqref{T3} and the definition \eqref{a_d}.


\bigskip
 
\begin{thebibliography}{99}

\bibitem{ESC} T. Abbas, H. Ammari, G. Hu, A. Wahab, and J. C. Ye, Two-dimensional elastic scattering coefficients and enhancement of nearly elastic cloaking, {\sl J. Elast.}, 128(2):(2017), pp. 203--243.

\bibitem{Princeton}  H. Ammari, E. Bretin, J. Garnier, H. Kang, H. Lee, and A. Wahab, {\sl Mathematical Methods in Elasticity Imaging}, Princeton Series in Applied Mathematics, Princeton University Press, New Jersey, USA, 2015.

\bibitem{colton2} D. Colton and R. Kress, {\sl Inverse Acoustic and Electromagnetic Scattering Theory}, Applied Mathematical Science, Vol. 93, Springer-Verlag, New York, 2013.

\bibitem{colton} D. Colton, and R. Kress, {\sl Integral Equation Methods in Scattering Theory}, Pure and Appl. Math., John Wiley \& Sons Inc., New York, 1983. 

\bibitem{OptEx} J. Lim, A. Wahab, G. Park, K. Lee, Y. Park, and J. C. Ye, Beyond Born-Rytov limit for super-resolution optical diffraction tomography, {\sl Opt. Express}, 25(24):(2017), pp. 30445--30458.


\bibitem{DirectElast} H. Ammari, P. Calmon, and E. Iakovleva,  Direct elastic imaging of a small inclusion, {\sl SIAM J. Imaging Sci.} 1(2):(2008), pp. 169--187.

\bibitem{Gintides} A. Charalambopoulos, D. Gintides, and K. Kiriaki, The linear sampling method for non-absorbing penetrable elastic bodies, {\sl Inverse Probl.}, 19:(2003), pp. 549--561.

\bibitem{DL} O. Dorn and D. Lesselier, Level set methods for inverse scattering, {\sl Inverse Probl.}, 22:(2006), pp. R67--R131.

\bibitem{Anna} E. Beretta, E. Bonnetier, E. Francini, and A. L. Mazzucato, Small volume asymptotics for anisotropic elastic inclusions,  {\sl Inverse Probl. Imaging}, 6(1):(2012), pp. 1-23. 

\bibitem{Elena} E. Beretta, E. Francini, E. Kim, and J.-Y. Lee, Algorithm for the determination of a linear crack in an elastic body from boundary measurements, {\sl Inverse Probl.}, 26(8): (2010), 085015.

\bibitem{GWL} S. Gdoura, A. Wahab, and D. Lesselier, Electromagnetic time reversal and scattering by a small dielectric inclusion, {\sl J. Phys. Conf. Ser.}, 386:(2012), 012010.

\bibitem{Li1} J. Li, H. Liu, Z. Shang, and H. Sun, Two single shot methods for locating multiple electromagnetic scatterers, {\sl SIAM J. Appl. Math.}, 73(4):(2013), pp. 1721--1746.

\bibitem{Li2} J. Li, H. Liu, and Q. Wang, Locating multiple multiscale electromagnetic scatterers by a single far-field measurement, {\sl SIAM J. Imaging Sci.}, 6(4):(2013), pp. 2285--2309. 

\bibitem{Li3} J. Li, H. Liu, and J. Zou, Locating multiple multiscale acoustic scatterers, {\sl Multiscale Model. Simul.}, 12(3):(2014), pp. 927--952.

\bibitem{JS} J. Yoo, Y. Jung, M. Lim, A. Wahab, and J. C. Ye, A joint sparse recovery framework for accurate reconstruction of inclusions in elastic media, {\sl SIAM J. Imaging Sci.}, 10(3):(2017), pp. 1104--1138.


\bibitem{Amer} A. Wahab, A. Rasheed, T. Hayat, and R. Nawaz, Electromagnetic time reversal algorithms and source localization in lossy dielectric media, {\sl  Commun. Theor. Phys.}, 62(6): (2014), pp. 779--789.


\bibitem[Ahn et al.(2014)]{Ahn} C. Y. Ahn, K. Jeon, Y.-K. Ma, and W.-K. Park, A study on the topological derivative-based imaging of thin electromagnetic inhomogeneities in limited-aperture problems, {\sl Inverse Probl.}, 30:(2014), 105004.

\bibitem{Bao} G. Bao, J. Lin, and S. M. Mefire, Numerical reconstruction of electromagnetic inclusions in three dimensions, {\sl SIAM J. Imaging Sci.}, 7(1):(2014), pp. 558--577.

\bibitem{Bellis} C. Bellis, M. Bonnet, and F. Cakoni, Acoustic inverse scattering using topological derivative of far-field measurements-based $L^2-$cost functionals, {\sl Inverse Probl.}, 29:(2013), 075012.

\bibitem{bonnet} M. Bonnet, Fast identification of cracks using higher-order topological sensitivity for 2-D potential problems, {\sl Eng. Anal. Bound. Elem.}, 35:(2011), pp. 223--235.

\bibitem{BG} M. Bonnet, and B. B. Guzina, Sounding of finite solid bodies by way of topological derivative, {\sl Internat. J. Numer. Methods Engrg.}, 61:(2004), pp. 2344--2373.

\bibitem{carpio1} A. Carpio and M. L. Rap\'{u}n, Solving inhomogeneous inverse problems by topological derivative methods, {Inverse Probl.}, 24:(2008), 045014. 

\bibitem{carpio2} A. Carpio and M. L. Rap\'{u}n, Hybrid topological derivative and gradient-based methods for electrical impedance tomography, {\sl Inverse Probl.}, 28:(2012), 095010. 

\bibitem{CGGM} J. C{\'e}a, S. Garreau, P. Guillaume, and M. Masmoudi, The shape and topological optimization connection, {\sl Comput. Methods Appl. Mech. Engrg.}, 188:(2001), pp. 703--726.

\bibitem{DG} N. Dominguez and V. Gibiat, Non-destructive imaging using the time domain topological energy method, {\sl Ultrasonics}, 50:(2010), pp. 172--179.

\bibitem{DGE} N. Dominguez, V. Gibiat, and Y. Esquerrea, Time domain topological gradient and time reversal analogy: An inverse method for ultrasonic target detection, {\sl Wave Motion}, 42:(2005), pp. 31--52.

\bibitem{partial} J. F. Funes, J. M. Perales, M.-L. Rap\'{u}n, and J. M. Vega, Defect detection from multi-frequency limited data via topological sensitivity, {\sl J. Math. Imaging Vis.}, 55(1):(2016), pp. 19--35.

\bibitem{Guzina} B. B. Guzina and I. Chikichev, From imaging to material identification: A generalized concept of topological sensitivity, {\sl J. Mech. Phys. Solids}, 55:(2007), pp. 245--279.

\bibitem{HL} M. Hinterm{\"u}ller and A. Laurain, Electrical impedance tomography: From topology to shape, {\sl Control Cybernet.}, 37:(2008), pp. 913--933.

\bibitem{Louer} F. Le Louer and M.-L. Rap\'{u}n, Topological sensitivity for solving inverse multiple scattering problems in three-dimensional electromagnetism. Part I: One step method, {\sl SIAM J. Imaging Sci.}, 10(3):(2017), pp. 1291--1321.


\bibitem{MPS} M. Masmoudi, J. Pommier, and B. Samet, The topological asymptotic expansion for the Maxwell equations and some applications, {\sl Inverse Probl.}, 21:(2005), pp. 547--564.


\bibitem{park} W.-K. Park, Topological derivative strategy for one-step iteration imaging of arbitrary shaped thin, curve-like electromagnetic inclusions, {\sl J. Comput. Phys.}, 231(4): (2012), pp. 1426--1439.

\bibitem{park2} W.-K. Park, Multi-frequency topological derivative for approximate shape acquisition of curve-like thin electromagnetic inhomogeneities, {\sl J. Math. Anal. Appl.}, 404(2):(2013), pp. 501--518. 


\bibitem{AGJK} H. Ammari, J. Garnier, V. Jugnon, and H. Kang, Stability and resolution analysis for a topological derivative based imaging functional, {\sl SIAM J. Control Optim.}, 50(1):(2012), pp. 48--76.

\bibitem{TDelastic} H. Ammari, E. Bretin, J. Garnier, W. Jing, H. Kang, and A. Wahab, Localization, stability, and resolution of topological derivative based imaging
functionals in elasticity, {\sl SIAM J. Imaging Sci.}, 6(4):(2013), pp. 2174--2212.

\bibitem{wahab} A. Wahab, Stability and resolution analysis of topological derivative based localization of small electromagnetic inclusions, {\sl SIAM J. Imaging Sci.}, 8(3):(2015), pp. 1687--1717.

\bibitem{JCM} A. Wahab, T. Abbas, N. Ahmed, and Q. M. Z. Zia, Detection of electromagnetic inclusions using topological sensitivity, {\sl J. Comput. Math.}, 35(5):(2017), pp. 642--671.

\bibitem{JPC} A. Wahab, N. Ahmed, and T. Abbas, Far field imaging of a dielectric inclusion, {\sl J. Phys. Conf. Ser.}, 657:(2015) 012001.

\bibitem{A1} J. Eom, H. Kang, G. Nakamura, and Y. C. Wang, Reconstruction of the shear modulus of viscoelastic systems in a thin cylinder: An inversion scheme and experiments, {\sl Inverse Probl.}, 32:(2016), 095007.

\bibitem{A2} A. B. Weglein, F. V. Ara\'{u}jo, P. M. Carvalho, R. H. Stolt, K. H. Matson, R. T. Coates, D. Corrigan, D. J. Foster, S. A. Shaw, and H. Zhang, Inverse scattering series and seismic exploration, {\sl Inverse Probl.}, 19:(2003), pp. 589--602.

\bibitem{A3} R. Sinkus, J. Lorenzen, J. Schrader, M. Lorenzen, M. Dargatz, and D. Holz, High-resolution tensor MR elastography for breast tumour detection, {\sl Phys. Med. Biol.}, 45:(2000), pp.  1649--1664.

\bibitem{A4} R. Sinkus, M. Tanter, T. Xydeas, S. Catheline, J. Bercoff, and M. Fink, Viscoelastic shear properties of in vivo breast lesions measured by MR elastography, {\sl Magn. Reson. Imaging}, 23:(2005), pp. 159--165.

\bibitem{Kup} V. D. Kupradze, 
{\sl Potential Methods in the Theory of Elasticity}, Danial Davey \& Co., New York, 1965.

\bibitem{Kup2} V. D. Kupradze, T. G. Gegelia, M. O. Basheleishvili, and T. V. Burchuladze, 
{\sl Three Dimensional Problems of the Mathematical Theory of Elasticity and Thermoelasticity}, North-Holland Series in Applied Mathematics and Mechanics, Vol. 25, North-Holland Publishing Co., Amsterdam, 1979.

\bibitem{HahnerHsiao} P. H\"{a}hner and G. C. Hsiao, Uniqueness theorems in inverse obstacle scattering of elastic waves,  {\sl Inverse Probl.} 9(5):(1993), pp. 525--534.

\bibitem{AK-Pol} H. Ammari and H. Kang,
{\sl Polarization and Moment tensors: with Applications to Inverse Problems and Effective Medium Theory}, Applied Mathematical Sciences Series, Vol. 162, Springer-Verlag, New York, 2007.

\bibitem{AKNT-02} H. Ammari, H. Kang, G. Nakamura, and K. Tanuma, Complete asymptotic expansions of solutions of the system of elastostatics in the presence of inhomogeneities of small diameter, {\sl J. Elast.}, 67:(2002), pp. 97--129.

\bibitem{Morse} P. M. Morse and H. Feshbach, {\sl Methods of Theoretical Physics}, Vol. I and II, McGraw-Hill, New York, 1953.

\bibitem{Dassios93} G. Dassios and Z. Rigou, On the density of traction traces in scattering of elastic waves, {\sl SIAM J. Appl. Math.}, 53:(1993), pp. 141--153.

\bibitem{Dassios95} G. Dassios and Z. Rigou, Elastic Herglotz functions, {\sl SIAM J. Appl. Math.}, 55:(1995), pp. 1345--1361. 

\bibitem{NIST} F. W. J. Olver, D. W. Lozier, R. F. Boisvert, and C. W. Clark, eds., {\sl NIST Handbook of Mathematical Functions}, Cambridge University Press, Cambridge, UK, 2010.

\bibitem{Mode} V. A. Korneev and L. R. Johnson, Scattering of P and  S waves by a spherically symmetric inclusion, {\sl Pure Appl. Geophys.}, 147:(1996), pp. 675--718.

\bibitem{AlvesKress} C. J. S. Alves and R. Kress, On the far-field operator in elastic obstacle scattering, {\sl IMA J. Appl. Math.}, 67(1):(2002), pp. 1--21.
\end{thebibliography}
\end{document}